\documentclass[11pt]{article}

\usepackage{amsmath,amsthm,amscd,latexsym}
\usepackage{amsfonts,pdfsync,color,graphicx}
\usepackage[psamsfonts]{amssymb}
\usepackage{epsfig}
\usepackage{color}
\usepackage[noadjust]{cite}
\usepackage{accents}

\usepackage{listings}
\usepackage{subfig}
\usepackage{multicol}
\usepackage{float}
\usepackage{hyperref}
\usepackage{tikz-cd}

\newcommand{\N}{{\mathbb N}}
\newcommand{\R}{{\mathbb R}}

\usepackage[top=1.5in,bottom=1.5in,left=1.3in,right=1.0in]{geometry}

\newtheoremstyle{mystyle}               % Name
{}                % Space above
{}                % Space below
{}        % Body font
{}                % Indent amount
{\bfseries \itshape}       % Theorem head font
{.}      % Punctuation after theorem head
{ }      % Space after theorem head, ' ', or \newline
{}       % Theorem head spec

\newtheorem{theorem}{Theorem}[section]
\newtheorem{proposition}[theorem]{Proposition}
\newtheorem{lemma}[theorem]{Lemma} 
\newtheorem{corollary}[theorem]{Corollary}
\theoremstyle{definition}
\newtheorem{definition}[theorem]{Definition}
\newtheorem{example}[theorem]{Example}	
\theoremstyle{mystyle}
\newtheorem{remark}[theorem]{Remark}
\numberwithin{equation}{section}
\newtheorem{notation}[theorem]{Notation}
\title{Spectral characterization of the constant sign derivatives of Green's function related to two point boundary value conditions}
\date{ }
\author{Alberto Cabada, Luc{\' i}a L\'opez-Somoza and Mouhcine Yousfi\\
	$^1$ CITMAga, 15782, Santiago de Compostela, Galicia, Spain\\
	$^2$ Departamento de Estatística, Análise Matemática e Optimización\\
	Facultade de Matem\'aticas, Universidade de Santiago de Com\-pos\-te\-la, Spain.\\
	alberto.cabada@usc.es; lucia.lopez.somoza@usc.es; yousfi.mouhcine@usc.es}
\begin{document}
	\maketitle
	\begin{abstract}
		In this paper we will study the set of parameters in which certain partial derivatives of the Green's function, related to a $n$-order linear operator $T_{n}[M]$, depending on a real parameter $M$, coupled to different two-point boundary conditions, are of constant sign. We will do it without using the explicit expression of the Green's function. The constant sign interval will be characterized by the first eigenvalue  related to suitable boundary conditions of the studied operator. As a consequence of the main result, we will  be able to give sufficient conditions to ensure that the derivatives of Green's function cannot be nonpositive (nonnegative). These characterizations and the obtained results can be used to deduce the existence of solutions of nonlinear problems under additional conditions on the nonlinear part. To illustrate the obtained results, some examples are given.
	\end{abstract}



	\section{Introduction}
	Many papers in the literature have studied the qualitative properties of solutions of differential equations coupled with different boundary conditions. Several techniques have been used to ensure the existence,  non-existence and multiplicity of solutions. Among them, stands out the method of lower and upper solutions, coupled with monotone iterative techniques which, as it is very well-known, are equivalent to the constant sign of the Green's function of the corresponding linear problem. Moreover the method of constructing suitable cones in Banach spaces where to apply the well-known Krasnosel'ski\u{\i}'s fixed-point theorem or the classical index theory, see for instance \cite{coster,ladde,krasno,torres}, is a powerful tool to deduce the existence of positive solutions of the nonlinear problem. Such cones are constructed using constant sign properties of the Green's function and some of its derivatives. So, as much information we have about such constant signs, we may construct smaller cones, that allow us to ensure, both the positiveness of the solutions we are looking for, and some of its derivatives. Thus, for instance, we may prove the existence of positive and convex solutions if both the Green's function and its second partial derivative with respect to $t$ are nonnegative on their square of definition.
	
	The study of the sign of the Green's function has been made for many authors on the literature. Usually some sufficient conditions are obtained to deduce the constant sign of such function and some of its partial derivatives with respect to the first variable. Different kind of two-point boundary conditions have been considered and, usually, the differential linear operator is assumed to be disconjugate or disfocal on a suitable interval. We refer to the reader the classical works \cite{Butler-Erbe, Elias-2, Nehari, peterson-JMAA, peterson-RMJM, eloe, eloe-henderson} and the recent ones of Almenar and J\'odar \cite{almenar-jodar-BVP-1, almenar-jodar-BVP-2, almenar-jodar-math-methods, almenar-jodar-mathematical-modeling, almenar-jodar-mathematics} and references therein.

	To be concise, in this paper, we consider the following $n$-order linear differential equation 
	\begin{equation}\label{Ec::T_n[M]}
		T_n[M]\,u(t)\equiv u^{(n)}(t)+M\,u(t)=0\,,\ t\in I:=[a,b]\,,
	\end{equation}
	where $M\in\mathbb{R}$ a is real  parameter, coupled to the two-point boundary conditions
	%\begin{eqnarray}
	%		\label{Ec::cfa} u^{(\sigma_1)}(a)=\cdots=u^{(\sigma_k)}(a)&=&0\,,\\
	%		\label{Ec::cfb} u^{(\varepsilon_1)}(b)=\cdots=u^{(\varepsilon_{n-k})}(b)&=&0\,,
	%	\end{eqnarray}
\begin{eqnarray}
	\label{Ec::cfa} u^{(\sigma_1)}(a)=& \cdots&= u^{(\sigma_k)}(a)\hspace{0.3cm}= 0, \\
	\label{Ec::cfb} u^{(\varepsilon_1)}(b)=& \cdots&= u^{(\varepsilon_{n-k})}(b)= 0,
\end{eqnarray}
with $k\in\{1,\dots, n-1\}$ and the sets of indices $\{\sigma_1,\dots, \sigma_k\}$, $\{\varepsilon_1,\dots, \varepsilon_{n-k}\}\subset\{0,\dots,n-1\}$, satisfying that
\[  0\le\sigma_1<\sigma_2<\cdots<\sigma_k\leq n-1\,,\quad  0\le \varepsilon_1<\varepsilon_2<\cdots<\varepsilon_{n-k}\leq n-1\,.\]  

This paper is devoted to study the set of parameters $M$ in which certain partial derivatives of the Green's function $g_{M}$, related to problem \eqref{Ec::T_n[M]}--\eqref{Ec::cfb}, are of constant sign without using the explicit expression of the Green's function. It should be noted that, in general, it is not possible to obtain the explicit expression of the Green's function, specially for non constant coefficients. Moreover, it is important to point out that the spectral characterization of the constant sign of the related Green's function to the general $n$-th order linear differential operator, with non constant coefficients, and coupled to the boundary conditions \eqref{Ec::cfa}--\eqref{Ec::cfb}, has been obtained in \cite{CabSaab}. Such existence results generalize the ones obtained by the same authors in \cite{CabSaa} for the $(k,n-k)$ boundary conditions. The main assumption on that references consists on the disconjugacy property of the considered operator for a given value of the parameter $\bar{M}$ and the, so-called, strongly inverse positive (negative) property of the Green's function. Under this assumption, it is obtained the exact interval of parameters $M$ for which the Green's function has constant sign. It is important to point out that such interval does not coincide with the interval of parameters in which the considered equation is disconjugate (see \cite{Disconju} for details). It must be noted that the arguments developed in references \cite{CabSaa, CabSaab} do not hold for the partial derivatives of the Green's function, because such derivatives are not the solution of the homogeneous linear equation. Is for this, that we must restrict our study to the operator $T_n[M]$.

For the particular case of equation \eqref{Ec::T_n[M]}, together with boundary conditions \eqref{Ec::cfa}--\eqref{Ec::cfb}, the interval of parameters for which the Green's function has constant sign has been characterized in \cite[Theorem 8.1]{CabSaab}. Also for boundary conditions of the type $(k,n-k)$, i.e., $\sigma_{j}=j-1$ for all $j=1,\dots,k$, and $\varepsilon_{j}=j-1$ for all $j=1,\dots,n-k$, which are a particular case of the previous ones, a characterization of the constant sign intervals of the Green's function is given in \cite{CabSaa}. Moreover, the exact expression of the eigenvalues that characterizes these intervals up to $n=8$ is obtained.

As we have previously mention, the main assumption is the disconjugacy character of the considered operator. In our case, it is very well known that the operator $T_{n}[M]$, is disconjugate for the value $M=0$ (see \cite{CabSaa, Coppel, elias}) at any arbitrary interval $[a,b]$. In the literature, sufficient conditions have been given to ensure the disconjugacy character of linear operators, we refer to \cite{Coppel,simons,elias,Zetel} and the references therein. A characterization of this property is proved in \cite{Disconju}, in which the optimal extremes of the intervals are given as the eigenvalues of suitable $(k,n-k)$ problems. 

It is in this context that the corresponding spectral theory arises:  we will prove that the constant sign interval is characterized by the first eigenvalues of the studied operator, related to suitable boundary conditions.

The study starts from conditions \eqref{Ec::cfa}--\eqref{Ec::cfb} and constructs the conditions satisfied by the derivatives of $g_{M}$, defining the necessary hypotheses that we need to obtain the main results, and deducing the relationships between the adjoint spaces of the respective related spaces. This way, we will obtain results that characterize the constant sign interval of the corresponding derivatives. 

The novelty of the work is that the results generalize \cite[Theorem 8.1]{CabSaab} for the particular operator $T_n[M]$ defined in \eqref{Ec::T_n[M]}. Also, as a consequence of these results, we will deduce  some cases in which the derivatives of the Green's function cannot be of constant sign. Furthermore, we will obtain a necessary condition that allow us to ensure that the corresponding derivatives are strictly positive (strictly negative) on the interior of the square of definition. We point out that, in general, the interval of parameters $M$ where the corresponding partial derivative has constant sign does not coincide with the interval of values for which the equation is disconjugate on the interval $I$.

The paper is organized as follows: in a preliminary Section $2$ we introduce the fundamental concepts, the results and the main hypotheses that are needed in the development of the article.  Section $3$ is devoted to prove the results that characterize the constant sign of the derivatives of Green's function by distinguishing three cases. The main results are proved through spectral theory. This section includes also a corollary that provides sufficient conditions for derivatives to be nonpositive or nonnegative, a necessary condition to ensure the negative (positive) sign of certain derivatives with respect to $t$ of $g_{M}$, and examples of application of our main results. Finally, Section $4$ includes an application to ensure the existence of positive solution of nonlinear problems.

\section{Preliminaries, hypotheses and main assumptions}
In this section, we will make a survey of several results and properties that will be used along the paper. The main part of these results are proven in \cite{CabSaab}, and many of them are given in \cite{CabSaa}.

Let us  consider the $n^{\rm th}$-order linear differential equation \eqref{Ec::T_n[M]} coupled to the two-point boundary conditions \eqref{Ec::cfa}--\eqref{Ec::cfb}.

\begin{definition}\label{D::Na}
	Let us say that $\{\sigma_1,\dots, \sigma_k\}-\{\varepsilon_1,\dots, \varepsilon_{n-k}\}$ satisfy property $(N_a)$ if
	\begin{equation*}%\label{Ec::Na}
		\sum_{\sigma_j< h}1+\sum_{\varepsilon_j< h}1\geq h\,,\quad \forall h\in\{1,\dots, n-1\}\,.\end{equation*}	
\end{definition}

\begin{remark}
	\label{r-admissible}
We point out that condition $(N_a)$ is a classical condition assumed when studying the sign of the Green's function and its corresponding partial derivatives coupled to a linear differential operator and with boundary conditions \eqref{Ec::cfa}--\eqref{Ec::cfb}. See for instance \cite{Butler-Erbe}. In reference \cite{almenar-jodar-mathematics} the authors denote such property as {\it admissible conditions}.
\end{remark}

\begin{notation}\label{nota}
	
	Let us denote 
	\begin{align}
		&\label{Ec::alpha} \alpha=\min \left\lbrace i\in \{0,\dots,n-1\}:\;\; i\notin \{\sigma_{1},\dots,\sigma_{k}\} \right\rbrace\,,\\
		&\label{Ec::beta}  \beta =\min \left\lbrace i\in \{0,\dots,n-1\}:\;\; i\notin \{\varepsilon_{1},\dots,\varepsilon_{n-k}\} \right\rbrace \,.
	\end{align}
\end{notation}

We introduce the following set of functions related to the boundary conditions \eqref{Ec::cfa}--\eqref{Ec::cfb}:
\begin{equation*}	%\label{Ec::X_se}
	X_{\{\sigma_1,\dots, \sigma_k\}}^{\{\varepsilon_1,\dots, \varepsilon_{n-k}\}}=\left\lbrace u\in C^n(I)\ \mid\ u^{(\sigma_1)}(a)=\cdots=u^{(\sigma_k)}(a)=u^{(\varepsilon_1)}(b)=\cdots=u^{(\varepsilon_{n-k})}(b)=0\right\rbrace.
\end{equation*}
Also, we define the following sets of functions:
{\footnotesize \begin{equation*}
		\begin{split}
			X_{\{\sigma_1,\dots, \sigma_{k-1}|\alpha\}}^{\{\varepsilon_1,\dots, \varepsilon_{n-k}\}}&=\left\lbrace u\in C^n(I)\ \mid\ u^{(\sigma_1)}(a)=\cdots=u^{(\sigma_{k-1})}(a)=0\,,\ u^{(\alpha)}(a)=0\,, u^{(\varepsilon_1)}(b)=\cdots=u^{(\varepsilon_{n-k})}(b)=0\right\rbrace,\\
			X_{\{\sigma_1,\dots, \sigma_{k}|\alpha\}}^{\{\varepsilon_1,\dots, \varepsilon_{n-k-1}\}}&=\left\lbrace u\in C^n(I)\ \mid\ u^{(\sigma_1)}(a)=\cdots=u^{(\sigma_{k})}(a)=0\,,\ u^{(\alpha)}(a)=0\,, u^{(\varepsilon_1)}(b)=\cdots=u^{(\varepsilon_{n-k-1})}(b)=0\right\rbrace,\\
			X_{\{\sigma_1,\dots, \sigma_{k-1}\}}^{\{\varepsilon_1,\dots, \varepsilon_{n-k}|\beta\}}&=\left\lbrace u\in C^n(I)\ \mid\ u^{(\sigma_1)}(a)=\cdots=u^{(\sigma_{k-1})}(a)=0\,, u^{(\varepsilon_1)}(b)=\cdots=u^{(\varepsilon_{n-k})}(b)=0\,,\ u^{(\beta)}(b)=0\right\rbrace,\\
			X_{\{\sigma_1,\dots, \sigma_{k}\}}^{\{\varepsilon_1,\dots, \varepsilon_{n-k-1}|\beta\}}&=\left\lbrace u\in C^n(I)\ \mid\ u^{(\sigma_1)}(a)=\cdots=u^{(\sigma_{k})}(a)=0\,, u^{(\varepsilon_1)}(b)=\cdots=u^{(\varepsilon_{n-k-1})}(b)=0\,,\ u^{(\beta)}(b)=0\right\rbrace. \end{split} 
\end{equation*}}
\begin{definition}	
	Let us say that the operator $T_{n}[M]$ satisfies the property $(T_d)$ in $X_{\{\sigma_1,\dots,\sigma_k\}}^{\{\varepsilon_1,\dots,\varepsilon_{n-k}\}}$ if, and only if, there exists the following decomposition:
	\begin{equation}
		\label{Ec::Td1} 
		T_0\,u(t)=u(t)\,,\quad T_l\,u(t)=\dfrac{d}{dt}\left( \dfrac{T_{l-1}\,u(t)}{v_l(t)}\right) \,,\ l= 1,\dots, n\,,\; t\in I,
	\end{equation}
	where $v_l>0$ on $I$, $v_l\in C^{n}(I)$ such that
	\begin{equation*}%\label{Ec::Td2}
		T_{n}[M]\,u(t)=v_1(t)\,\dots\,v_n(t)\,T_n\,u(t)\,,\ t\in I\,,
	\end{equation*}
	and, moreover, such a decomposition satisfies, for every $u\in  X_{\{\sigma_1,\dots,\sigma_k\}}^{\{\varepsilon_1,\dots,\varepsilon_{n-k}\}}$:
	\begin{eqnarray}
		\nonumber T_{\sigma_1}\,u(a)\hspace{0.04cm}=&\cdots&=T_{\sigma_k}\,u(a)\hspace{0.3cm}=0\,,\\
		\nonumber T_{\varepsilon_1}\,u(b)=&\cdots&=T_{\varepsilon_{n-k}}\,u(b)=0\,.
	\end{eqnarray}
\end{definition}
\begin{remark}\label{recordNa}
	Note that decomposition \eqref{Ec::Td1} is not unique, it depends on the choice of $v_{k}$ for $k=1,\dots, n$. Throughout this work, for the operator $T_n[0]\,u(t)=u^{(n)}(t)$ we choose the following decomposition:
	\begin{equation*}
		T_0\,u(t)=u(t)\,,\quad T_k\,u(t)=\dfrac{d}{dt}\left( \dfrac{T_{k-1}\,u(t)}{v_k(t)}\right) \,,\ k = 1,\dots, n\,,\;\; t\in I,
	\end{equation*}
	where $v_k\equiv 1$ on $I$. That is, $T_{k}\,u(t)=u^{(k)}(t)$, $t\in I$. In particular:
	\begin{equation*}T_n[0]\,u(t)=v_1(t)\,\dots\,v_n(t)\,T_n\,u(t)\,,\ t\in I\,,\end{equation*}
	and, moreover, this decomposition satisfies, for every $u\in  X_{\{\sigma_1,\dots,\sigma_k\}}^{\{\varepsilon_1,\dots,\varepsilon_{n-k}\}}$:
	\begin{eqnarray*}
		\nonumber 
		T_{\sigma_1}\,u(a)=u^{(\sigma_1)}(a)=0\,,&\dots\,,& T_{\sigma_k}\,u(a)\hspace{0.3cm}=u^{(\sigma_k)}(a)\hspace{0.3cm}=0\,,\\
		\nonumber T_{\varepsilon_1}\,u(b)=u^{(\varepsilon_1)}(b)=0\,,&\dots\,,&T_{\varepsilon_{n-k}}\,u(b)=u^{(\varepsilon_{n-k})}(b)=0\,,
	\end{eqnarray*}
	or which is the same $T_{n}[0]$ satisfies property $(T_{d})$ in $X_{\{\sigma_1,\dots,\sigma_k\}}^{\{\varepsilon_1,\dots,\varepsilon_{n-k}\}}$.
\end{remark}
\begin{lemma}\cite[Lemma 3.8]{CabSaab}\label{LemaNa}
	Let $\bar{M}\in\mathbb{R}$ be such that $T_n[\bar{M}]$ satisfies the property  $(T_d)$ in $X_{\{\sigma_1,\dots,\sigma_k\}}^{\{\varepsilon_1,\dots,\varepsilon_{n-k}\}}$. Then  $\{\sigma_1,\dots,\sigma_k\}-\{\varepsilon_1,\dots,\varepsilon_{n-k}\}$ satisfy the property $(N_a)$ if, and only if, $M=0$ is not an eigenvalue of $T_n[\bar{M}]$ in $X_{\{\sigma_1,\dots,\sigma_k\}}^{\{\varepsilon_1,\dots,\varepsilon_{n-k}\}}$.
\end{lemma}
Let us denote by $G_{M}$ the Green's matrix associated with the equivalent $n$-dimensional first order system associated to \eqref{Ec::T_n[M]}--\eqref{Ec::cfb}, which is given by the expression
\begin{equation}\label{Ec:MG} G_{M}(t,s)=\left( \begin{array}{ccccc}
		g_1(t,s)&g_2(t,s)&\cdots&g_{n-1}(t,s)&g_M(t,s)\\&&&&\\
		\dfrac{\partial }{\partial t}\,g_1(t,s)& \dfrac{\partial }{\partial t}\,g_2(t,s)&\cdots&\dfrac{\partial }{\partial t}\,g_{n-1}(t,s)& \dfrac{\partial }{\partial t}\,g_M(t,s)\\
		\vdots&\vdots&\cdots&\vdots&\vdots\\
		\dfrac{\partial^{n-1} }{\partial t^{n-1}}\,g_1(t,s)&\dfrac{\partial^{n-1} }{\partial t^{n-1}}\,g_2(t,s)&\cdots&\dfrac{\partial^{n-1} }{\partial t^{n-1}}\,g_{n-1}(t,s)&\dfrac{\partial^{n-1}} {\partial t^{n-1}}\,g_M(t,s)\end{array} \right) ,\end{equation}
where $g_M$ is the scalar Green's function related to operator $T_{n}[M]$ in $X_{\{\sigma_1,\dots,\sigma_k\}}^{\{\varepsilon_1,\dots,\varepsilon_{n-k}\}}$ (see \cite[Section 1.4]{system} for details). Moreover, by the axiomatic definition of $G_{M}$, we have that for all $s\in(a,b)$, the following equality holds
\begin{equation} \label{equation} \dfrac{\partial }{\partial t}\, G_{M}(t,s)=A\,G_{M}(t,s),\quad \text{for all } t\in I\backslash \left\lbrace s\right\rbrace,
\end{equation}
where $A=\left(
\begin{array}{c|c}
	0 & I_{n-1} \\
	\hline
	-M& 0
\end{array}\right)$, with $I_{n-1}$ the $(n-1)$ dimensional identity matrix, denotes the matrix that defines the $n$-dimensional system equivalent to our problem \eqref{Ec::T_n[M]}--\eqref{Ec::cfb} and 
\begin{equation}\label{relacion}
	B\,G_{M}(a,s)+C\,G_{M}(b,s)=0,\;\;\text{for all}\;\; s\in(a,b),
\end{equation}
with $ B,\,\ C\in \mathcal{M}_{n\times n}$, defined by $b_{j,1+\sigma_j}=1$ for $j=1,\dots,k$ and $c_{j+k,1+\varepsilon_j}=1$ for ${j=1,\dots,n-k}$; otherwise, $b_{ij}=0$ and $c_{ij}=0$.

In particular, it is easy to see, using the expressions \eqref{Ec:MG} and \eqref{equation} (see \cite[Section 1.4]{system} for details), that $g_M\in C^{n-2}(I \times I)$. Moreover, it is a $C^n$ function on the triangles $a\le s < t \le b$ and $a\le t < s \le b$, it satisfies, as a function of $t$ for any $s\in (a,b)$ fixed, the two-point boundary value conditions \eqref{Ec::cfa}-\eqref{Ec::cfb} and solves equation \eqref{Ec::T_n[M]} for all $t\in I\setminus \{s\}$.

Moreover, for any $t\in (a,b)$ it satisfies 
\begin{equation}\label{e-salto}
	\dfrac{\partial^{n-1}}{\partial t^{n-1}}\,g_{M}(t^{+},t)=\dfrac{\partial^{n-1}}{\partial t^{n-1}}\,g_{M}(t^{-},t)+1.
\end{equation}

Studying the matrix Green's function $G_{M}$, it is proved in \cite[page 13]{CabSaa} that, for this particular case, $g_{n-j}(t,s)$ may be expressed as a function of $g_M(t,s)$ for all $j=1,\dots,n-1$ as follows:\begin{equation}
	\label{Ec::gj} g_{n-j}(t,s)=(-1)^j\dfrac{\partial^j}{\partial s^j}\,g_M(t,s)\,.
\end{equation}

The adjoint operator of $T_{n}[M]$ follows the expression 
\begin{equation}\label{EC::Ad}
	T_n^*[M]v (t)\equiv (-1)^n \,v^{(n)}(t)+M\,v(t)\,,
\end{equation}
defined on the domain 
\begin{equation*}
	\footnotesize
	D(T_n^*[M])=\left\lbrace v\in C^n(I)\ \mid \sum_{i=0}^{n-1} (-1)^{n-1-i}\,v^{(n-1-i)}(b)\,u^{(i)}(b)=\sum_{i=0}^{n-1} (-1)^{n-1-i}\,v^{(n-1-i)}(a)\,u^{(i)}(a)\, \right\rbrace,
\end{equation*}
for all $u\in D(T_n[M])$.

We denote by $g_M^*(t,s)$  the Green's function related to operator $T_{n}^*[M]$. Furthermore, the following equality is satisfied (see \cite[section 1.4]{system})
\begin{equation}\label{gg2}
	g^*_M(t,s)=g_M(s,t),\,\;\; (t,s)\in I\times I.
\end{equation}

In such a case, if we define the following operator 
\begin{equation}\label{Ec::Tg}
	\widehat{T}_{n}[(-1)^n\,M]:=(-1)^n\, T_{n}^*[M] \,,
\end{equation}
we have, according to the above equality, that
\begin{equation}\label{Ec::gg1}
	\widehat{g}_{(-1)^n\,M}(t,s)=(-1)^n\,g_M^*(t,s)=(-1)^n\,g_{M}(s,t)\,,
\end{equation}
where $\widehat g_{(-1)^n M}(t,s)$ is the scalar Green's function related to operator $\widehat T_{n}[(-1)^nM]$ in $D\left( T_{n}^*[M]\right)$. 

As a particular case of the obtained results in \cite[section 4]{CabSaab}, we can deduce that $$D(T_{n}^*[M])\equiv X_{\ \, \{\sigma_1,\dots,\sigma_k\}}^{*\{\varepsilon_1,\dots,\varepsilon_{n-k}\}}=X_{ \{\tau_1,\dots,\tau_{n-k}\}}^{\{\delta_1,\dots,\delta_{k}\}},$$ 
being $\{\delta_1,\dots,\delta_k\}\,,\{\tau_1,\dots,\tau_{n-k}\}\subset\{0,\dots,n-1\}$, such that $\delta_i<\delta_{i+1}$ and $\tau_j<\tau_{j+1}$, for $i=1,\dots,k-1$ and $j=1,\dots,n-k-1$, satisfying:
\[\begin{split}	\{\sigma_1,\dots,\sigma_k,n-1-\tau_1,\dots,\,n-1-\tau_{n-k}\}=&\{0,\dots,n-1\},\\
	\{\varepsilon_1,\dots,\varepsilon_{n-k},n-1-\delta_1,\dots,\,n-1-\delta_k\}=&\{0,\dots,n-1\}.
\end{split}\]
\begin{notation}
	Let us define $\eta$, $\gamma$ in the following way:
	\begin{align}
		&\label{Ec::eta} \eta=\min \left\lbrace i\in \{0,\dots,n-1\}:\;\; i\notin \{\tau_{1},\dots,\tau_{n-k}\} \right\rbrace\,,\\
		&\label{Ec::gamma} \gamma=\min \left\lbrace i\in \{0,\dots,n-1\}:\;\; i\notin \{\delta_{1},\dots,\delta_{k}\} \right\rbrace\,.
	\end{align}
\end{notation}
Analogously, let us denote by $ \widehat G(t,s)$ the Green's matrix associated with the equivalent $n$-dimensional problem of $\widehat{T}_n[(-1)^n\,M]\,v(t)=0$, $v\in X_{\{\tau_1,\dots,\tau_{n-k}\}}^{\{\delta_1,\dots,\delta_{k}\}}$ given by
{\small\begin{equation}\label{Ec:MGh} \widehat G(t,s)=\left( \begin{array}{llll}
			\widehat g_1(t,s)&\cdots&\widehat g_{n-1}(t,s)&\widehat g_{(-1)^n\,M}(t,s)\\&&&\\
			\dfrac{\partial }{\partial t}\,\widehat g_1(t,s)& \cdots&\dfrac{\partial }{\partial t}\,\widehat g_{n-1}(t,s)& \dfrac{\partial }{\partial t}\,\widehat g_{(-1)^n\,M}(t,s)\\
			\qquad	\vdots&\cdots&\qquad\vdots&\qquad\vdots\\
			\dfrac{\partial^{n-1} }{\partial t^{n-1}}\,\widehat g_1(t,s)&\cdots&\dfrac{\partial^{n-1} }{\partial t^{n-1}}\,\widehat g_{n-1}(t,s)&\dfrac{\partial^{n-1}} {\partial t^{n-1}}\,\widehat g_{(-1)^n\,M}(t,s)\end{array} \right),\end{equation}}
where $\widehat g_{(-1)^n\,M}(t,s)$ is the scalar Green's function related to operator $\widehat{T}_n[(-1)^n\,M]$ in $X_{\{\tau_1,\dots,\tau_{n-k}\}}^{\{\delta_1,\dots,\delta_{k}\}}$.

Similarly to \eqref{Ec::gj} we can deduce that 
\begin{equation}
	\label{Ec::gjh} \widehat g_{n-j}(t,s)=(-1)^j\dfrac{\partial^j}{\partial s^j}\,\widehat g_{(-1)^nM}(t,s)\,.
\end{equation}

By definition,  $\widehat G(t,s)$ satisfies the equality
\begin{equation}\label{relacion2}
	\widehat B\,\widehat G(a,s)+\widehat C\,\widehat G(b,s)=0,\;\;\text{for all}\;\; s\in (a,b),
\end{equation}
where $\widehat B$, $\widehat C\in \mathcal{M}_{n\times n}$, defined as $ (\widehat B)_{i,\,\tau_i+1}=1$ for $i=1,\dots,n-k$ and $ (\widehat C)_{j+n-k,\,1+\delta_{j}}=1$ for $j=1,\dots,k$; otherwise, $( \widehat B)_{i\,j}=0$ and $( \widehat C)_{i\,j}=0$.
\begin{remark}\label{ref1}
	From the previous definitions, as it is pointed out in \cite[Remarks 4.1 and 4.2]{CabSaab} we have that $\alpha=n-1-\tau_{n-k}$, $\beta=n-1-\delta_k$, $\eta=n-1-\sigma_k$ and $\gamma=n-1-\varepsilon_{n-k}$.
\end{remark}
\begin{example}\label{ejemp}
	Let us consider operator $T_{4}[M]$ coupled with the boundary conditions $$u(0)=u'(0)=u''(0)=u''(1)=0.$$ 
	In this case, $\{\sigma_1,\sigma_2,\sigma_3\}=\{0,1,2\}, \{\varepsilon_1\}=\{2\}, \{\tau_1\}=\{0\}$ and $\{\delta_1,\delta_2,\delta_3 \}=\{0,2,3\}$. Thus, we deduce that $X_{\ \, \{0,1,2\}}^{*\{2\}}=X_{ \{0\}}^{\{0,2,3\}}.$ Moreover, $\alpha=3\,,\beta=0,\,\eta=1$ and $\gamma=1$.
\end{example}
\begin{definition}\label{d-IP}
	Operator $T_{n}[M]$ is said to be inverse positive (negative) on $X_{\{\sigma_1,\dots,\sigma_k\}}^{\{\varepsilon_1,\dots,\varepsilon_{n-k}\}}$ if every $u \in X_{\{\sigma_1,\dots,\sigma_k\}}^{\{\varepsilon_1,\dots,\varepsilon_{n-k}\}}$ such that $T_n[M]\, u \ge 0$ on $I$, satisfies $u\geq 0$ ($u\leq 0$) on $I$.
\end{definition}
The above definition is equivalent to the constant sign of Green's function $g_{M}$ related to $T_{n}[M]$ in $X_{\{\sigma_1,\dots,\sigma_k\}}^{\{\varepsilon_1,\dots,\varepsilon_{n-k}\}}$ as it is shown in the following result whose proof is analogous the one made  in \cite[Theorem 2.9]{CabSaab}.

\begin{theorem}\label{T::in1}
	Operator $T_{n}[M]$ is inverse positive (negative) in $X_{\{\sigma_1,\dots,\sigma_k\}}^{\{\varepsilon_1,\dots,\varepsilon_{n-k}\}}$ if, and only if, the Green's function related to problem \eqref{Ec::T_n[M]}--\eqref{Ec::cfb} is non-negative (non-positive) on its square of definition.
\end{theorem}
\begin{definition}\label{d-SIP}
	Operator $T_{n}[M]$ is said to be strongly inverse positive (negative) in $X_{\{\sigma_1,\dots,\sigma_k\}}^{\{\varepsilon_1,\dots,\varepsilon_{n-k}\}}$ if every  $u \in X_{\{\sigma_1,\dots,\sigma_k\}}^{\{\varepsilon_1,\dots,\varepsilon_{n-k}\}}$ such that $T_{n}[M]\, u \gneqq 0$ on $I$, satisfies that $u> 0$ ($<0$)  on $(a,b)$ and, moreover, $u^{(\alpha)}(a)>0$ $(<0)$ and $(-1)^{\beta}u^{(\beta)}(b)>0$ $(<0)$, where $\alpha$ and $\beta$ are defined in \eqref{Ec::alpha} and \eqref{Ec::beta}, respectively.
\end{definition}
Analogously to Theorem \ref{T::in1}, the following characterization can be shown:
\begin{theorem}\label{T::in2}
	Operator $T_{n}[M]$ is strongly inverse positive (negative)  in $X_{\{\sigma_1,\dots,\sigma_k\}}^{\{\varepsilon_1,\dots,\varepsilon_{n-k}\}}$ if, and only if, the Green's function related to problem \eqref{Ec::T_n[M]}--\eqref{Ec::cfb} satisfies the following properties:
	\begin{itemize}
		\item $g_{M}(t,s)>0\,(<0)$  a.e. on $(a,b)$.
		\item $\frac{\partial^\alpha}{\partial t^\alpha}\,g_{M}(t,s)>0\,(<0)$  for a.e. $s\in(a,b)$. 
		\item  $(-1)^{\beta}\frac{\partial^\beta}{\partial t^\beta}\,g_{M}(t,s)>0\,(<0)$ for a.e. $s\in(a,b)$.
	\end{itemize}
\end{theorem}
Let $u\in X_{\{\sigma_1,\dots,\sigma_k\}}^{\{\varepsilon_1,\dots,\varepsilon_{n-k}\}}$ be such that $T_{n}[M]\,u(t)=0$, $t\in I$. Then, differentiating equation \eqref{Ec::T_n[M]} with respect to $t$, as a direct consequence of equations \eqref{Ec:MG} and \eqref{equation} we deduce that, for any $s\in (a,b)$ given, the $q$-th derivative $v_{s}^{q}[M](t):=\dfrac{\partial^{q} }{\partial t^{q}}\,g_{M}(t,s)$,  with $1\le q\le n-1$,  satisfies the equation $T_{n}[M]\,v_{s}^{q}[M](t)=0$, for all $t\in I\setminus \{s\}$, together with the new boundary conditions 
\begin{eqnarray}
	\label{condicion1}
	(v_{s}^{q})^{(\mu_{1}^{q})}[M](a)=&\cdots&=(v_{s}^{q})^{(\mu_{k}^{q})}[M](a)\hspace{0.3cm}=0\,,\\
	\label{condicion2}
	(v_{s}^{q})^{(\rho_{1}^{q})}[M](b)=&\cdots&=(v_{s}^{q})^{(\rho_{n-k}^{q})}[M](b)=0\,.
\end{eqnarray}
Here, the sets of indices $\{\mu_{1}^{q},\dots, \mu_{k}^{q}\},\; \{\rho_{1}^{q},\dots, \rho_{n-k}^{q}\}\subset\{0,\dots,n-1\}$, satisfy that 
\[ 0\leq \mu_{1}^{q}<\mu_{2}^{q}<\cdots<\mu_{k}^{q}\leq n-1\,,\quad  0\leq \rho_{1}^{q}<\rho_{2}^{q}<\cdots<\rho_{n-k}^{q}\leq n-1\,,\] and are defined in the following way:
\begin{itemize}
	\item If $A_{q}=\{i\in\{1,\dots,k\}\,/\;\; \sigma_{i}\ge q\}\neq \varnothing$, then   \[\mu_{1}^{q}=\sigma_{j}-q,\,\mu_{2}^{q}=\sigma_{j+1}-q,\,\dots,\mu_{k-(j-1)}^{q}=\sigma_{k}-q,\,\mu_{k-(j-2)}^{q}=\sigma_{1}-q+n ,\dots,\,\mu_{k}^{q}=\sigma_{j-1}-q+n\,,\]
	with \begin{equation}\label{j} 
		j=\min A_q\geq 1. 
	\end{equation}
	\item If $A_{q}=\varnothing$, then 
	\[\mu_{i}^{q}=\sigma_{i}-q+n,\;\; i=1,\ldots,k.\]
	\item If $B_{q}=\{i\in\{1,\dots,n-k\}\,/\;\; \varepsilon_{i}\ge q\}\neq \varnothing$, then 
	\begin{equation*}
		\begin{aligned}
			\rho_{1}^{q}&=\varepsilon_{r}-q,\,\rho_{2}^{q}=\varepsilon_{r+1}-q,\,\dots,\rho_{n-k-(r-1)}^{q}=\varepsilon_{n-k}-q,\,\rho_{n-k-(r-2)}^{q}=\varepsilon_{1}-q+n ,\dots,\\
			\rho_{n-k}^{q}&=\varepsilon_{r-1}-q+n,
		\end{aligned}
	\end{equation*}
	with 
	\begin{equation}\label{r}
		r=\min B_q\geq 1.
	\end{equation}
	\item If $B_{q}=\varnothing$, then 
	\[\rho_{i}^{q}=\varepsilon_{i}-q+n,\;\; i=1,\ldots,n-k.\] 
\end{itemize}
\begin{remark}
	The above process of calculating the boundary conditions of $v_{s}^{q}[M]$ is valid for both $M\neq 0$ and $M=0$.
\end{remark}
Taking into account previous computations, we may define $\phi$ as the function which maps the space $X_{\{\sigma_1,\dots,\sigma_k\}}^{\{\varepsilon_1,\dots,\varepsilon_{n-k}\}}$ into the space of the boundary conditions that satisfies the first derivative, that is, $X_{\{\mu_{1}^{1},\dots,\mu_{k}^{1}\}}^{\{\rho_{1}^{1},\dots,\rho_{n-k}^{1}\}}$.
This way, we can extend this definition to the $q$-th derivative as $$\phi^{q}=\phi\circ \overset{q)}{\cdots}\circ \phi:X_{\{\sigma_1,\dots,\sigma_k\}}^{\{\varepsilon_1,\dots,\varepsilon_{n-k}\}}\longrightarrow X_{\{\mu_{1}^{q},\dots,\mu_{k}^{q}\}}^{\{\rho_{1}^{q},\dots,\rho_{n-k}^{q}\}}.$$
In particular, $\phi^{n}$ is the identity function.
\begin{example}\label{example1}
	Consider the operator $T_{5}[M]$ defined in $X_{\{0,2,4\}}^{\{0,2\}}$. Applying the above definition of the spaces of the derivatives we obtain the following sequence of spaces:
	\[ \begin{tikzcd}
		X_{\{0,2,4\}}^{\{0,2\}} \arrow{r}{\phi} &  X_{\{1,3,4\}}^{\{1,4\}} \arrow{r}{\phi}  &  X_{\{0,2,3\}}^{\{0,3\}}\arrow{r}{\phi}  & X_{\{1,2,4\}}^{\{2,4\}} \arrow{r}{\phi} & X_{\{0,1,3\}}^{\{1,3\}} \arrow{r}{\phi}& X_{\{0,2,4\}}^{\{0,2\}}.
	\end{tikzcd}
	\]
\end{example}
Analogously to the Notation \ref{nota}, let us define $\alpha^{q}$, $\beta^{q}$ as follows:
\begin{align}
	&\label{alphaq}\alpha^{q}=\min \left\lbrace i\in \{0,\dots,n-1\}:\;\; i\notin \{\mu_{1}^{q},\dots,\mu_{k}^{q}\} \right\rbrace\,,\\
	&\label{betaq}\beta^{q}=\min \left\lbrace i\in \{0,\dots,n-1\}:\;\; i\notin \{\rho_{1}^{q},\dots,\rho_{n-k}^{q}\} \right\rbrace\,.
\end{align}
\begin{remark}
	It is important to point out that for any $q\ge 1$, the function $v_{s}^{q}[M]$ is not the Green's function related to the operator $T_{n}[M]$ in $X_{\{\mu_{1}^{q},\dots,\mu_{k}^{q}\}}^{\{\rho_{1}^{q},\dots,\rho_{n-k}^{q}\}}$. This is due to the fact that $v_{s}^{q}[M]\in C^{n-2-q}(I\times I)$.
	
	%Moreover, it is easy to verify that if $T_{n}[M]\,u(t)=0$, $t\in I$ and $u\in X_{\{\sigma_1,\dots,\sigma_k\}}^{\{\varepsilon_1,\dots,\varepsilon_{n-k}\}}$ then $u^{(q)}\in X_{\{\mu_{1}^{q},\dots,\mu_{k}^{q}\}}^{\{\rho_{1}^{q},\dots,\rho_{n-k}^{q}\}}.$
\end{remark}
\begin{remark}\label{lemma::1}
	The operator $T_{n}[0]$ satisfies the property $(T_d)$ in  $X_{\{\mu_{1}^{q},\dots,\mu_{k}^{q}\}}^{\{\rho_{1}^{q},\dots,\rho_{n-k}^{q}\}}$. From Remark~\ref{recordNa}, we know that operator $T_n[0]\,u(t)=u^{(n)}(t)$ satisfies the following decomposition:
	\begin{equation*}
		T_0\,u(t)=u(t)\,,\quad T_k\,u(t)=\dfrac{d}{dt}\left( \dfrac{T_{k-1}\,u(t)}{v_k(t)}\right) \,,\ k = 1,\dots, n\,,
	\end{equation*}
	where $v_k\equiv 1$ on $I$ and
	\begin{equation*}T_n[0]\,u(t)=v_1(t)\,\dots\,v_n(t)\,T_n\,u(t)\,,\ t\in I\,,\end{equation*}
	that is, $T_{k}\,u(t)=u^{(k)}(t)$.
	
	Moreover, this decomposition satisfies, for every $u\in  X_{\{\mu_{1}^{q},\dots,\mu_{k}^{q}\}}^{\{\rho_{1}^{q},\dots,\rho_{n-k}^{q}\}}$:
	\begin{eqnarray*}
		\nonumber T_{\mu_{1}^{q}}\,u(a)=u^{(\mu_{1}^{q})}(a)=0\,,&\dots\,,& T_{\mu_k^{q}}\,u(a)\hspace{0.3cm}=u^{(\mu_k^{q})}(a)\hspace{0.3cm}=0\,,\\
		\nonumber T_{\rho_1^{q}}\,u(b)=u^{(\rho_1^{q})}(b)=0\,,&\dots\,,&T_{\rho_{n-k}^{q}}\,u(b)=u^{(\rho_{n-k}^{q})}(b)=0\,.
	\end{eqnarray*} 
\end{remark}
We consider the general expression of the solution of equation $$T_{n}[0]\,u(t)+\lambda\,u(t)\equiv  u^{(n)}(t)+ \lambda\, u(t)=0,\;\; t\in I,$$ to obtain the different eigenvalues.

Next we present a result concerning the spectrum of the spaces $X_{\{\mu_{1}^{q},\dots,\mu_{k}^{q}\}}^{\{\rho_{1}^{q},\dots,\rho_{n-k}^{q}\}}$.
\begin{lemma}\label{spectrum}
	$\bar \lambda\neq 0$ is an eigenvalue of $T_{n}[0]$ in $X_{\{\sigma_1,\dots,\sigma_k\}}^{\{\varepsilon_1,\dots,\varepsilon_{n-k}\}}$ if and only if $\bar \lambda$ is an eigenvalue of $T_{n}[0]$ in  $X_{\{\mu_{1}^{q},\dots,\mu_{k}^{q}\}}^{\{\rho_{1}^{q},\dots,\rho_{n-k}^{q}\}}$ for all $q\in \{1,\dots,n-1\}$.
\end{lemma}
\begin{proof}
	Suppose that $\bar \lambda\neq 0$ is an eigenvalue of $T_{n}[0]$ in $X_{\{\sigma_1,\dots,\sigma_k\}}^{\{\varepsilon_1,\dots,\varepsilon_{n-k}\}}$. So, there is $u\not\equiv 0$ in $I$  such that $u^{(n)}(t)+\bar \lambda\,u(t)=0$, $t\in I$ and $u\in X_{\{\sigma_1,\dots,\sigma_k\}}^{\{\varepsilon_1,\dots,\varepsilon_{n-k}\}}$. Since $u\not\equiv 0$ in $I$ we have that $u^{(n)}\not\equiv 0$ in $I$. Therefore, $u^{(j)}\not\equiv 0$ in $I$ for all $j\in \{1,\dots,n-1\}$.
	
	Taking into account that $T_{n}[0]\,u^{(j)}(t)=u^{(j+n)}(t)=-\bar \lambda\,u^{(j)}(t)$ for all $t\in I$ and $j\in \N$, we have that $u^{(j)}$ is an eigenvector related to the eigenvalue $\bar \lambda$ such that $u^{(j)}\in X_{\{\mu_{1}^{j},\dots,\mu_{k}^{j}\}}^{\{\rho_{1}^{j},\dots,\rho_{n-k}^{j}\}}$ for all $j\in\{1,\dots,n-1\}$. 
	
	For $j=n$, we have that $\mu_{i}^{n}=\sigma_{i}$ for all $i\in\{1,\dots,k\}$ and $\rho_{l}^{n}=\varepsilon_{l}$ for all $l\in \{1,\dots,n-k\}$. Therefore, the result holds.
\end{proof}
\begin{remark}
	As a direct consequence of Remark \ref{recordNa}, we have, from Lemma \ref{LemaNa}, that ${\{\sigma_1,\dots,\sigma_k\}}-{\{\varepsilon_1,\dots,\varepsilon_{n-k}\}}$ satisfies condition $(N_a)$ if and only if $M=0$ is not an eigenvalue of $u^{(n)}$ in $X_{\{\sigma_1,\dots,\sigma_k\}}^{\{\varepsilon_1,\dots,\varepsilon_{n-k}\}}$.
\end{remark}
\begin{remark}
	Note that if $\lambda=0$ is not an eigenvalue of $T_n[0]$ in $X_{\{\mu_{1}^{q},\dots,\mu_{k}^{q}\}}^{\{\rho_{1}^{q},\dots,\rho_{n-k}^{q}\}}$, then the spaces $X_{\{\sigma_1,\dots,\sigma_k\}}^{\{\varepsilon_1,\dots,\varepsilon_{n-k}\}}$ and $X_{\{\mu_{1}^{q},\dots,\mu_{k}^{q}\}}^{\{\rho_{1}^{q},\dots,\rho_{n-k}^{q}\}}$ have the same spectrum. However, it is very important to point out that $\lambda=0$ can appear as an eigenvalue in some space $X_{\{\mu_{1}^{q},\dots,\mu_{k}^{q}\}}^{\{\rho_{1}^{q},\dots,\rho_{n-k}^{q}\}}$, as it is shown in the following example. 
\end{remark}
\begin{example}
	Let us consider again the Example \ref{example1}. As we know, the following sequence of spaces is fulfilled:
	\[\begin{tikzcd}
		X_{\{0,2,4\}}^{\{0,2\}}\arrow{r}{\phi}& X_{\{1,3,4\}}^{\{1,4\}}\arrow{r}{\phi} &X_{\{0,2,3\}}^{\{0,3\}}\arrow{r}{\phi} &X_{\{1,2,4\}}^{\{2,4\}}\arrow{r}{\phi}& X_{\{0,1,3\}}^{\{1,3\}}\arrow{r}{\phi}& X_{\{0,2,4\}}^{\{0,2\}}.
	\end{tikzcd}\]
	
	In this case, we have that $\lambda=0$ is not an eigenvalue of $T_{5}[0]$ in $X_{\{0,2,4\}}^{\{0,2\}}$, $X_{\{0,2,3\}}^{\{0,3\}}$ and $X_{\{0,1,3\}}^{\{1,3\}}$, while $\lambda=0$ is an eigenvalue of $T_{5}[0]$ in $X_{\{1,3,4\}}^{\{1,4\}}$ and $X_{\{1,2,4\}}^{\{2,4\}}$.
\end{example}
Now, we will show how to determine the adjoint space of  $X_{\{\mu_{1}^{q},\dots,\mu_{k}^{q}\}}^{\{\rho_{1}^{q},\dots,\rho_{n-k}^{q}\}}$ from the adjoint space of $X_{\{\sigma_1,\dots,\sigma_k\}}^{\{\varepsilon_1,\dots,\varepsilon_{n-k}\}}$. We arrive to the next result.
\begin{theorem}
	The adjoint space of $X_{\{\mu_{1}^{q},\dots,\mu_{k}^{q}\}}^{\{\rho_{1}^{q},\dots,\rho_{n-k}^{q}\}}$ is given by $X_{ \{\tau_{1}^{q},\dots,\tau_{n-k}^{q}\}}^{\{\delta_{1}^{q},\dots,\delta_{k}^{q}\}}$, where the indices $\{\tau_{1}^{q},\dots,\tau_{n-k}^{q}\}\,, \{\delta_{1}^{q},\dots,\delta_{k}^{q}\}\subset\{0,\dots,n-1\}$ satisfy
	\[ 0\leq \tau_{1}^{q}<\tau_{2}^{q}<\cdots<\tau_{n-k}^{q}\leq n-1\,,\quad  0\leq \delta_{1}^{q}<\delta_{2}^{q}<\cdots<\delta_{k}^{q}\leq n-1\,,\]
	and are defined as follows:
	\begin{itemize}
		\item If $C_{q}=\{i\in\{1,\dots,n-k\}\,/\;\; \tau_{i}+q\leq n-1\}\neq \varnothing$, then   
		\begin{equation*}
			\begin{aligned} \tau_{1}^{q}&=\tau_{l+1}+q-n,\,\tau_{2}^{q}=\tau_{l+2}+q-n,\,\dots,\tau_{n-k-l}^{q}=\tau_{n-k}+q-n,\,\tau_{n-k-l+1}^{q}=\tau_{1}+q ,\dots,\,\\
				\tau_{n-k}^{q}&=\tau_{l}+q, 
			\end{aligned}
		\end{equation*}
		with 
		\begin{equation}\label{l}
			l=\max C_q.
		\end{equation}
		\item If $C_{q}=\varnothing$, then 
		\[\tau_{i}^{q}=\tau_{i}+q-n,\;\; i=1,\ldots,n-k.\]
		\item If $D_{q}=\{i\in\{1,\dots,k\}\,/\;\; \delta_{i}+q\leq n-1\}\neq \varnothing$, then 
		\begin{equation*}
			\delta_{1}^{q}=\delta_{p+1}+q-n,\,\delta_{2}^{q}=\delta_{p+2}+q-n,\,\dots,\delta_{k-p}^{q}=\delta_{k}+q-n,\,\delta_{k-p+1}^{q}=\delta_{1}+q ,\dots,\delta_{k}^{q}=\delta_{p}+q, 
		\end{equation*}
		with 
		\begin{equation}\label{p}
			p=\max D_q.
		\end{equation}
		\item If $D_{q}=\varnothing$, then 
		\[\delta_{i}^{q}=\delta_{i}+q-n,\;\; i=1,\ldots,k.\] 
	\end{itemize}
	As a consequence, the following diagram
	\begin{equation}\label{diagrama}
		\begin{tikzcd}
			X_{\{\sigma_1,\dots,\sigma_k\}}^{\{\varepsilon_1,\dots,\varepsilon_{n-k}\}} \arrow{r}{\phi^{q}} \arrow{d}{*} & X_{\{\mu_{1}^{q},\dots,\mu_{k}^{q}\}}^{\{\rho_{1}^{q},\dots,\rho_{n-k}^{q}\}}  \arrow{d}{*}\\
			X_{\{\tau_{1},\dots,\tau_{n-k}\}}^{\{\delta_{1},\dots,\delta_{k}\}} & X_{\{\tau_{1}^{q},\dots,\tau_{n-k}^{q}\}}^{\{\delta_{1}^{q},\dots,\delta_{k}^{q}\}}  \arrow{l}{\phi^{q}}
		\end{tikzcd}
	\end{equation}
	is commutative.
\end{theorem}
\begin{proof}
	We do the proof only for $\tau_{i}^{q}$. For $\delta_{i}^{q}$, the proof is done in a similar way. To prove it, we start from the definition of $\tau_{i}$, we have that $\tau_{i}<\tau_{i+1}$ for all $i\in \{1,\dots,n-k-1\}$ and 
	\[\{\sigma_1,\dots,\sigma_k,n-1-\tau_{n-k},\dots,\,n-1-\tau_{1}\}\equiv\{0,\dots,n-1\}.\]
	Taking into account the definition of $X_{\{\mu_{1}^{q},\dots,\mu_{k}^{q}\}}^{\{\rho_{1}^{q},\dots,\rho_{n-k}^{q}\}}$ we have that 
	\begin{equation*}
		\begin{tikzcd}
			\{\sigma_1,\dots,\sigma_k,n-1-\tau_{n-k},\dots,\,n-1-\tau_{1}\} \arrow{r}{\phi^q} & \{\mu_1^q,\dots,\mu_k^{q},n-1-\tau_{n-k}^q,\dots,\,n-1-\tau_{1}^q\},
		\end{tikzcd}
	\end{equation*}
	with $\{\mu_1^q,\dots,\mu_k^{q},n-1-\tau_{n-k}^q,\dots,\,n-1-\tau_{1}^q\}\equiv\{0,\dots,n-1\}$ where $\tau_{i}^{q}$ has to fulfill that 
	\[ 0\leq \tau_{1}^{q}<\tau_{2}^{q}<\cdots<\tau_{n-k}^{q}\leq n-1\,.\]
	To calculate them, we distinguish two cases:
	\begin{itemize}
		\item[1)] If $C_{q}\neq \varnothing$, then:
		\begin{itemize}
			\item If $n-1-\tau_{i}\geq q$ (i.e, $ i\le l$), $n-1-\tau_{i}$ becomes $n-1-\tau_{i}-q$ by applying $\phi^q$.
			\item If $n-1-\tau_{i}<q$ (i.e, $ i>l$),  $n-1-\tau_{i}$ becomes $n-1-\tau_{i}-q+n$ by applying $\phi^q$.
		\end{itemize}
		Reordering the elements, we obtain that:
		\begin{equation*}
			\begin{aligned}
				0&\leq n-1-\tau_{l}-q<n-1-\tau_{l-1}-q<\cdots<n-1-\tau_{1}-q\\
				&<n-1-\tau_{n-k}-q+n<\cdots<n-1-\tau_{l+1}-q+n\leq n-1,
			\end{aligned}
		\end{equation*}
		or, equivalently,
		\[ 0\leq \tau_{l+1}+q-n<\tau_{l+2}+q-n<\cdots<\tau_{n-k}+q-n<\tau_{1}+q<\cdots<\tau_{l}+q\leq n-1\,.\]
		Taking 
		\[\begin{aligned}
			\tau_{1}^{q}&=\tau_{l+1}+q-n,\,\tau_{2}^{q}=\tau_{l+2}+q-n,\,\dots,\tau_{n-k-l}^{q}=\tau_{n-k}+q-n,\,\tau_{n-k-l+1}^{q}=\tau_{1}+q ,\dots,\,\\
			\tau_{n-k}^{q}&=\tau_{l}+q\,,\end{aligned}\]
		the first statement yields.
		
		\item[2)] If $C_{q}=\varnothing$, then $n-1-\tau_{i}$ becomes $n-1-\tau_{i}-q+n$ by applying $\phi^{q}$ and we obtain that 
		\[ 0\leq n-1-\tau_{n-k}-q+n<\cdots<n-1-\tau_{2}-q+n<n-1-\tau_{1}-q+n\leq n-1\,,\]
		and, reordering, we infer that 
		\[ 0\leq \tau_{1}+q-n<\tau_{2}+q-n<\cdots<\tau_{n-k}+q-n\le n-1\,.\]
		Taking $\tau_{i}^{q}=\tau_{i}+q-n,\;\; i=1,\ldots,n-k$, the second statement yields.
	\end{itemize}
	Therefore, $X_{\ \, \{\mu_1^q,\dots,\mu_k^q\}}^{*\{\rho_1^q,\dots,\rho_{n-k}^q\}}=X_{\{\tau_{1}^{q},\dots,\tau_{n-k}^{q}\}}^{\{\delta_{1}^{q},\dots,\delta_{k}^{q}\}}$.
\end{proof}
\begin{remark}
	Note that, unlike the direct arrows between the leading spaces, the arrows between the adjoint spaces are in the inverse direction.
\end{remark}
\begin{example}
	Consider the operator $T_{4}[M]$ coupled with the boundary conditions 
	\[u(a)=u''(a)=u'(b)=u''(b).\]
	In this case, $u\in X_{\{0,2\}}^{\{1,2\}}$ and $X_{\ \, \{0,2\}}^{*\{1,2\}}=X_{\{0,2\}}^{\{0,3\}}$. Using the diagram \eqref{diagrama}, we get that
	\[ \begin{tikzcd}
		X_{\{0,2\}}^{\{1,2\}} \arrow{r}{\phi} \arrow[swap]{d}{*} & X_{\{1,3\}}^{\{0,1\}} \arrow{r}{\phi} \arrow{d}{*} &  X_{\{0,2\}}^{\{0,3\}} \arrow{r}{\phi} \arrow{d}{*} & X_{\{1,3\}}^{\{2,3\}} \arrow{d}{*} \\%
		X_{\{0,2\}}^{\{0,3\}} & X_{\{1,3\}}^{\{0,1\}}  \arrow{l}{\phi} &  X_{\{0,2\}}^{\{1,2\}} \arrow{l}{\phi}& X_{\{1,3\}}^{\{2,3\}} \arrow{l}{\phi}
	\end{tikzcd}
	\]
\end{example}

Using the fact that the adjoint operators have the same spectrum as the original one, from Lemma \ref{spectrum}, we arrive at the following result.

\begin{lemma}
	$ \lambda^*\neq 0$ is an eigenvalue of $T_{n}[0]$ in $X_{\{\tau_{1},\dots,\tau_{n-k}\}}^{\{\delta_{1},\dots,\delta_{k}\}}$ if and only if $\lambda^*$ is an eigenvalue of $T_{n}[0]$ in  $X_{\{\tau_{1}^q,\dots,\tau_{n-k}^q\}}^{\{\delta_{1}^q,\dots,\delta_{k}^q\}}$ for all $q\in\{0,\dots,n-1\}$. Moreover, for each fixed $q\in\{0,\dots,n-1\}$, $ \bar \lambda\in \R$ is an eigenvalue of $T_{n}[0]$ in $X_{\{\mu_{1}^{q},\dots,\mu_{k}^{q}\}}^{\{\rho_{1}^{q},\dots,\rho_{n-k}^{q}\}}$  if and only if $\bar \lambda$ is an eigenvalue of $T_{n}[0]$ in $X_{\{\tau_{1}^q,\dots,\tau_{n-k}^q\}}^{\{\delta_{1}^q,\dots,\delta_{k}^q\}}$. 
\end{lemma}
Following the line of Definition \ref{d-SIP} and the characterization of Theorem \ref{T::in2}, we introduce the following concept for $v_{s}^{q}[M]$. 
\begin{definition}\label{def:::1}
	Let $q\in\{0,\dots,n-1\}$ be fixed. We said that the function $v_{s}^{q}[M]$ is strongly positive (strongly negative) on $I\times I$ if it satisfies the following properties:
	\begin{itemize}
		\item $v_{s}^{q}[M]>0\,(<0)$  a.e. on $(a,b)$.
		\item $(v_{s}^{q})^{(\alpha^{q})}[M](a)>0\,(<0)$  for a.e. $s\in(a,b)$, with $\alpha^{q}$ defined in \eqref{alphaq}.
		\item  $(-1)^{\beta^{q}}(v_{s}^{q})^{(\beta^{q})}[M](b)>0\,(<0)$ for a.e. $s\in(a,b)$, with $\beta^{q}$ defined in \eqref{betaq}.
	\end{itemize}
\end{definition}
Let us consider the following two conditions on $g_M(t,s)$ introduced in \cite[pages 78 and 86]{system} as follows:
\begin{itemize}
	\item[$(P_g$)] Suppose that there is a continuous function $\phi(t)>0$ for all $t\in (a,b)$ and $k_1,\ k_2\in L^1(I)$, such that $0<k_1(s)<k_2(s)$ for a.e. $s\in I$, satisfying
	\[\phi(t)\,k_1(s)\leq g_M(t,s)\leq \phi(t)\, k_2(s)\,,\quad \text{for a.e. } (t,s)\in I \times I \,.\]
	\item[($N_g$)] Suppose that there is a continuous function $\phi(t)>0$ for all $t\in (a,b)$ and $k_1,\ k_2\in L^1(I)$, such that $k_1(s)<k_2(s)<0$ for a.e. $s\in I$, satisfying
	\[\phi(t)\,k_1(s)\leq g_M(t,s)\leq \phi(t)\, k_2(s)\,,\quad \text{for a.e. }(t,s)\in I \times I\,.\]
\end{itemize}
As a particular case of \cite[Theorem 5.1]{CabSaab}, the following result is attained.
\begin{theorem}\label{signo}
	Suppose that ${\{\sigma_1,\dots,\sigma_k\}}-{\{\varepsilon_1,\dots,\varepsilon_{n-k}\}}$ satisfies condition $(N_a)$. Then the following properties are fulfilled:
	\begin{itemize}
		\item If $n-k$ is even, then $T_n[0]$ is strongly inverse positive in $X_{\{\sigma_1,\dots,\sigma_k\}}^{\{\varepsilon_1,\dots,\varepsilon_{n-k}\}}$ and, moreover, the related Green's function, $g_{0}(t,s)$, satisfies $(P_g)$.
		\item If $n-k$ is odd, then $T_n[0]$ is strongly inverse negative in $X_{\{\sigma_1,\dots,\sigma_k\}}^{\{\varepsilon_1,\dots,\varepsilon_{n-k}\}}$ and, moreover, the related Green's function, $g_{0}(t,s)$, satisfies $(N_g)$.
	\end{itemize}
\end{theorem}

Let us define ${\{\sigma_{1}^{q},\dots,\sigma_{c_q}^{q}\}}, {\{\varepsilon_{1}^{q},\dots,\varepsilon_{d_{q}}^{q}\}} \subset \{0,\dots,n-1\}$, in the following way:
\begin{eqnarray}
	\label{cq}\{\sigma_{1}^{q},\dots,\sigma_{c_q}^{q}\}&=&\{\sigma_{1}-q,\dots,\sigma_{k}-q\}\cap \{0,\dots,n-q-1\},\\
	\label{dq} \{\varepsilon_{1}^{q},\dots,\varepsilon_{d_q}^{q}\}&=&\{\varepsilon_{1}-q,\dots,\varepsilon_{n-k}-q\}\cap \{0,\dots,n-q-1\}.
\end{eqnarray}

\begin{remark}
	Taking into account that function $v_{s}^{q}(t)=\frac{\partial^{q}}{\partial t^{q}}\,g_{0}(t,s)$ is the Green's function related to operator $T_{n-q}[0]$ in the space $X_{\{\sigma_{1}^{q},\dots,\sigma_{c_{q}}^{q}\}}^{\{\varepsilon_{1}^{q},\dots,\varepsilon_{d_q}^{q}\}}$, then operator $T_{n-q}[0]$ is strongly inverse positive (negative) in the space $X_{\{\sigma_{1}^{q},\dots,\sigma_{c_{q}}^{q}\}}^{\{\varepsilon_{1}^{q},\dots,\varepsilon_{d_q}^{q}\}}$ if and only if it satisfies the Definition~\ref{d-SIP} with the same $\alpha^{q}$ and $\beta^{q}$ since $$\mu_{1}^{q}=\sigma_{1}^{q},\,\mu_{2}^{q}=\sigma_{2}^{q},\,\dots,\mu_{k-(j-1)}^{q}=\sigma_{c_{q}}^{q},$$ and $$\rho_{1}^{q}=\varepsilon_{1}^{q},\,\rho_{2}^{q}=\varepsilon_{2}^q,\,\dots,\rho_{n-k-(r-1)}^{q}=\varepsilon_{d_{q}}^q.$$
\end{remark}

If ${\{\sigma_1,\dots,\sigma_k\}}-{\{\varepsilon_1,\dots,\varepsilon_{n-k}\}}$ satisfies the condition $(N_a)$ then $c_{q}+d_{q}\leq n-q$. Moreover, if $c_{q}+d_{q}<n-q$ then ${\{\sigma_{1}^{q},\dots,\sigma_{c_{q}}^{q}\}}-{\{\varepsilon_{1}^{q},\dots,\varepsilon_{d_q}^{q}\}}$ does not satisfy $(N_a)$, and in this case  $\lambda=0$ is an eigenvalue of $X_{\{\sigma_{1}^{q},\dots,\sigma_{c_{q}}^{q}\}}^{\{\varepsilon_{1}^{q},\dots,\varepsilon_{d_q}^{q}\}}$. We will show now that $c_{q}+d_{q}=n-q$ is a necessary condition for the $q$-th partial derivative of the Green's function to have constant sign. 

\begin{lemma}\label{lem:signo_cte_g_parcial}
	Let $q\in \{1,\dots, n-1\}$. Suppose that ${\{\sigma_1,\dots,\sigma_k\}}-{\{\varepsilon_1,\dots,\varepsilon_{n-k}\}}$ satisfies condition $(N_a)$. If there exists $M\in \mathbb{R}$ such that $\frac{\partial^{q}}{\partial t^{q}} g_{M}$ has constant sign on $I\times I$, then $c_{q}+d_{q}=n-q$ and $g_M$ has constant sign on $I\times I$.
\end{lemma}
\begin{proof}
	If $v_{s}^{q}[M](t)=\frac{\partial^{q}}{\partial t^{q}} g_{M}(t,s)$ has constant sign, then $v_{s}^{q-1}[M](t)=\frac{\partial^{q-1}}{\partial t^{q-1}} g_{M}(t,s)$ is a monotone and continuous function. Thus, $v_{s}^{q-1}[M]$ has at most one zero on $I$. 
	
	By recurrence, we have that $v_{s}^{q-l}[M](t)=\frac{\partial^{q-l}}{\partial t^{q-l}} g_{M}(t,s)$ has at most $l$ zeros on $I$. In particular, $g_{M}(\cdot,s)$ has at most $q$ zeros on $I$. 
	
	Now, since condition $(N_a)$ implies that $g_{M}(\cdot,s)$ has at least $q$ zeros of order smaller or equal to $q$ on the boundary, we have, arguing as in the proof of \cite[Theorem 8.1, Step 5]{CabSaab}, that the existence of any of such zeros implies that function $g_{M}(\cdot,s)$ lost one zero on the interior of $I$. As a consequence, 	we deduce that:
	\begin{itemize}
		\item $g_{M}(\cdot,s)$  has exactly $q$ zeros of order smaller or equal to $q$ on the boundary, that is, $c_{q}+d_{q}=n-q$.
		\item $g_{M}(\cdot,s)$ cannot have any zero on $(a,b)$, that is, $g_M$ has constant sign.
	\end{itemize} 
\end{proof}

Next, for any $q\in \{1,\dots,n-1\}$ we determine the positive (negative) sign of $v_{s}^{q}(t):=v_{s}^{q}[0](t)$ on the interval $I$.
\begin{theorem}\label{resultSigno}
	Let $q\in \{1,\dots,n-1\}$. Suppose that ${\{\sigma_1,\dots,\sigma_k\}}-{\{\varepsilon_1,\dots,\varepsilon_{n-k}\}}$ satisfies condition $(N_a)$ and $c_q+d_q=n-q$. The following properties are fulfilled:
	\begin{itemize}
		\item [(a)] If $c_{q}\geq 1$ and $d_{q}\geq 1$, then $v_{s}^{q}$ is strongly positive on $I\times I$ if $n-q-c_{q}$ is even and   $v_{s}^{q}$ is strongly negative on $I\times I$ if $n-q-c_{q}$ is odd.
		\item [(b)] If $c_{q}=n-q$ and $d_{q}= 0$, then $v_{s}^{q}(t)>0,\dots,v_{s}^{n-1}(t)>0$ in $(s,b]$, ${v_{s}^{q}(t)=\dots=v_{s}^{n-2}(t)=0}$ in $[a,s]$ and $v_{s}^{n-1}(t)=0$ in $[a,s)$.
		\item [(c)] If $c_{q}=0$ and $d_{q}=n-q$, then $v_{s}^{q}(t)>0$ in $[a,s)$ if $q=n-l$ with $l$ even and $v_{s}^{q}(t)<0$ in $[a,s)$ if $q=n-l$ with $l$ odd. Moreover, $v_{s}^{q}(t)=\dots=v_{s}^{n-2}(t)=0$ in $[s,b]$ and $v_{s}^{n-1}(t)=0$ in $(s,b]$. 
	\end{itemize}
\end{theorem}
\begin{proof} 
	\begin{itemize}
		\item [(a)] In this case, it easy to verify that the space $X_{\{\sigma_{1}^{q},\dots,\sigma_{c_{q}}^{q}\}}^{\{\varepsilon_{1}^{q},\dots,\varepsilon_{d_q}^{q}\}}$ satisfies the condition $(N_a)$ since \begin{equation*}
			\sum_{\sigma_j-q< h}1+\sum_{\varepsilon_j-q< h} 1\geq h\,,\quad \forall h\in\{1,\dots, n-q-1\}\,.\end{equation*}	
		So, taking into account the fact that $v_{s}^{q}$ is the Green's function related to operator $T_{n-q}[0]$ in the space $X_{\{\sigma_{1}^{q},\dots,\sigma_{c_{q}}^{q}\}}^{\{\varepsilon_{1}^{q},\dots,\varepsilon_{d_q}^{q}\}}$, we apply Theorem \ref{signo}, and we have proved the first statement.
		\item[(b)] Since $v_{s}^{(n)}(t)=0$ for all $t\neq 0$, we have, by \eqref{e-salto} and $v_{s}^{(n-1)}(a)=0$, that $\frac{\partial^{n-1}}{\partial t^{n-1}}\,g_{0}(t,s)=0$ if $t<s$ and $\frac{\partial^{n-1}}{\partial t^{n-1}}\,g_{0}(t,s)=1$ if $t>s$.
		
		Since $v_{s}^{h}(a)=0$ for all $q\le h\le n-2$, and $g_{0}\in C^{n-2}(I\times I)$, we deduce that $\frac{\partial^{h}}{\partial t^{h}}\,g_{0}(t,s)=0$ if $t<s$ and $\frac{\partial^{h}}{\partial t^{h}}\,g_{0}(t,s)>0$ if $t>s$. So, we conclude the proof of second claim.
		\item[(c)] In this case arguing as in the previous case, since $v_{s}^{(n)}(b)=0$, we have that\linebreak 
		$\frac{\partial^{n-1}}{\partial t^{n-1}}\,g_{0}(t,s)=0$ if $t>s$ and $\frac{\partial^{n-1}}{\partial t^{n-1}}\,g_{0}(t,s)=-1$ if $t<s$.
		
		As consequence, since $v_{s}^{h}(b)=0$ for all $q\le h\le n-2$, and $g_{0}\in C^{n-2}(I\times I)$, we deduce, by recurrence, that  
		$\frac{\partial^h}{\partial t^h} g_{0}(t,s)=0$ for all $t>s$ and   $(-1)^{n-h}\frac{\partial^h}{\partial t^h} g_{0}(t,s)>0$ for all $t<s$. Thus, the third assertion holds.
	\end{itemize}
\end{proof}

As for the set of parameters of $M$ in which the function $v_{s}^{q}[M]$, $q\in\{1,\ldots,n-1\}$ maintains constant sign and monotony with respect to $M$, we have the following results:
\begin{lemma}\label{lem:signo_cte_parcialq1_parcialq2}
	Let $q_1,\, q_2\in \{1,\dots, n-1\}$, $q_1<q_2$ be such that $c_{q_1}+d_{q_1}=n-q_1$. Suppose that ${\{\sigma_1,\dots,\sigma_k\}}-{\{\varepsilon_1,\dots,\varepsilon_{n-k}\}}$ satisfies condition $(N_a)$. If there exists $M\in \mathbb{R}$ such that $\frac{\partial^{q_2}}{\partial t^{q_2}} g_{M}$ has constant sign on $I\times I$, then $\frac{\partial^{q_1}}{\partial t^{q_1}} g_{M}$ has constant sign on $I\times I$.
\end{lemma}
\begin{proof}
	Reasoning as in Lemma~\ref{lem:signo_cte_g_parcial}, we deduce that $v_{s}^{q_1}[M](t)=\frac{\partial^{q_1}}{\partial t^{q_1}} g_{M}(t,s)$ has at most $q_2-q_1$ zeros on $I$. 
	
	Now, condition $c_{q_1}+d_{q_1}=n-q_1$  implies that $v_{s}^{q_1}[M]$ has at least $q_2-q_1$ zeros of order belonging to the set $\{q_1, \ldots,q_2-1\}$ on the boundary. Thus, arguing as in the proof of  Lemma~\ref{lem:signo_cte_g_parcial} again, we deduce that $v_{s}^{q_1}[M]$ cannot have any zero on $(a,b)$. Thus, $v_{s}^{q_1}[M]$ has constant sign.
\end{proof}

Moreover, we recall the following facts:
\begin{itemize}
	\item If $n-k$ is even, then from Theorem \ref{signo} and \cite[Lemma 1.8.33]{system} we have that $g_{M}(t,s)>0$ for all $(t,s)\in (a,b)\times (a,b)$ if and only if $M\in (\lambda_{1},+\infty)$ or $M\in (\lambda_{1},\lambda_{2}]$, and it is monotone decreasing with respect to $M$ on such interval, where $\lambda_{1}<0$ is the first eigenvalue of $T_{n}[0]$ in $X_{\{\sigma_1,\dots,\sigma_k\}}^{\{\varepsilon_1,\dots,\varepsilon_{n-k}\}}$ and $\lambda_{2}>0$ (if the interval is bounded) is not an eigenvalue of $T_{n}[0]$ in $X_{\{\sigma_1,\dots,\sigma_k\}}^{\{\varepsilon_1,\dots,\varepsilon_{n-k}\}}$. The value of $\lambda_{2}$ is characterized in \cite[Sect.7 and Sect.8]{CabSaab}.
	\item If $n-k$ is odd, then from Theorem \ref{signo} and \cite[Lemma 1.8.25]{system} we have that $g_{M}(t,s)<0$ for all $(t,s)\in (a,b)\times (a,b)$ if and only if $M\in (-\infty,\lambda_{1})$ or $M\in [\overline{\lambda_{2}},\lambda_{1})$ and it is monotone decreasing with respect to $M$ on such interval, where $\lambda_{1}>0$ is the first eigenvalue of $T_{n}[0]$ in $X_{\{\sigma_1,\dots,\sigma_k\}}^{\{\varepsilon_1,\dots,\varepsilon_{n-k}\}}$ and $\overline{\lambda_{2}}<0$ (if the interval is bounded) is not an eigenvalue of $T_{n}[0]$ in $X_{\{\sigma_1,\dots,\sigma_k\}}^{\{\varepsilon_1,\dots,\varepsilon_{n-k}\}}$. The value of $\overline{\lambda_{2}}$ is characterized in \cite[Sect. 7 and Sect. 8]{CabSaab}.
\end{itemize}

Therefore, as a direct consequence of Theorem~\ref{signo}, Lemma~\ref{lem:signo_cte_g_parcial} and \cite[Theorem 3.6]{Monotonia}, the following result is deduced.
\begin{lemma}\label{lemma:monotonia_derivadas} The following properties hold:
	\begin{itemize}
		\item If $n-k$ is even and $\frac{\partial^{q}}{\partial t^{q}} g_{M}(t,s)>0$ on $(a,b)\times (a,b)$ for some $M\in \mathbb{R}$, then $\frac{\partial^{q}}{\partial t^{q}} g_{M}(t,s)>0$ if and only if $M\in (\lambda_{1},M_{q}]$, with $M_{q}\le \lambda_{2}$. Moreover, $\frac{\partial^{q}}{\partial t^{q}} g_{M}$ is monotone decreasing with respect to $M\in (\lambda_{1},M_{q}]$.
		
		\item If $n-k$ is even and $\frac{\partial^{q}}{\partial t^{q}} g_{M}(t,s)<0$ on $(a,b)\times (a,b)$ for some $M\in \mathbb{R}$, then $\frac{\partial^{q}}{\partial t^{q}} g_{M}(t,s)<0$ if and only if $M\in (\lambda_{1},M_{q}]$, with $M_{q}\le \lambda_{2}$. Moreover, $\frac{\partial^{q}}{\partial t^{q}} g_{M}$ is monotone increasing with respect to $M\in (\lambda_{1},M_{q}]$.
		
		\item If $n-k$ is odd and $\frac{\partial^{q}}{\partial t^{q}} g_{M}(t,s)>0$ on $(a,b)\times (a,b)$ for some $M\in \mathbb{R}$, then $\frac{\partial^{q}}{\partial t^{q}} g_{M}(t,s)>0$ if and only if $M\in [M_{q},\lambda_{1})$, with $M_{q}\ge \overline{\lambda_{2}}$. Moreover, $\frac{\partial^{q}}{\partial t^{q}} g_{M}$ is monotone increasing with respect to $M\in [M_{q},\lambda_{1})$. 
		
		\item If $n-k$ is odd and $\frac{\partial^{q}}{\partial t^{q}} g_{M}(t,s)<0$ on $(a,b)\times (a,b)$ for some $M\in \mathbb{R}$,  then $\frac{\partial^{q}}{\partial t^{q}} g_{M}(t,s)<0$ if and only if $M\in [M_{q},\lambda_{1})$, with $M_{q}\ge\overline{\lambda_{2}}$. Moreover, $\frac{\partial^{q}}{\partial t^{q}} g_{M}$ is monotone decreasing with respect to $M\in [M_{q},\lambda_{1})$.
	\end{itemize}
\end{lemma}

Let $q\in\{0,\dots,n-1\}$ be fixed. Suppose that  $(-1)^{n-q-c_{q}}\,v_{s}^{q}$ is strongly positive on  $I\times I$. Then, for each $s\in (a,b)$, we obtain the following limits:
\begin{eqnarray}\nonumber \ell_{1}^{q}(s)&: = &\lim_{t\rightarrow a^+}\dfrac{(-1)^{n-q-c_{q}}\,v_{s}^{q}(t)}{(t-a)^{\alpha^{q}}\,(b-t)^{\beta^{q}}} = \dfrac{(-1)^{n-q-c_{q}}\,(v_{s}^{q})^{(\alpha^{q})}(a)}{\alpha^{q}!\, (b-a)^{\beta^{q}}}\,,\\\nonumber
	\ell_{2}^{q}(s)& :=&\lim_{t\rightarrow b^-}\dfrac{(-1)^{n-q-c_{q}}\,v_{s}^{q}(t)}{(t-a)^{\alpha^{q}}\,(b-t)^{\beta^{q}}} = \dfrac{(-1)^{n-q-c_{q}-\beta^{q}}\,(v_{s}^{q})^{(\beta^{q})}(b)}{\beta^{q}!\, (b-a)^{\alpha^{q}}}. \end{eqnarray}

For each $s\in(a,b)$, let us consider the following function defined on $I$ by 
\begin{eqnarray*}
	\tilde{u}_{s}^{q}(t)=\left\{
	\begin{array}{ccc}
		\ell_{1}^{q}(s)\,,& t=a,\\
		\dfrac{(-1)^{n-q-c_{q}}\,v_{s}^{q}(t)}{(t-a)^{\alpha^{q}}\,(b-t)^{\beta^{q}}}\,,& t\in (a,b),\\
		\ell_{2}^{q}(s)\,,& t=b.
	\end{array}
	\right.
\end{eqnarray*}
It is clear that  $\tilde{u}_{s}^{q}>0$ on $[a,b]$ for all $s\in (a,b)$. 

Since $g_{0}(t,s)\in C^{n-2}(I\times I)$,  $\frac{\partial^{n-1}}{\partial t^{n-1}} g_{0} \in C^{\infty}((I\times I)\setminus\{(t,t)\;/\; t\in I\})$ and there exists $\lim\limits_{s\rightarrow t^{\pm}} \frac{\partial^{n-1}}{\partial t^{n-1}} g_{0}(t,s)\in \R$, we deduce that there exists $K^{q}>0$ such that $\tilde{u}_{s}^{q}(t)\leq K^{q}$ for every $(t,s)\in I\times I$  and $q\in \{1,\dots,n-1\}$. Therefore, the following functions							 		
\begin{eqnarray}\nonumber 
	\tilde{k}_{1}^{q}(s)&=&\min_{t\in I} \tilde{u}_{s}^{q}(t)\,,\quad s\in I\,,\\\nonumber\\\nonumber
	\tilde{k}_{2}^{q}(s) &=& \max_{t\in I}\tilde{u}_{s}^{q}(t)\,,\quad s\in I\,, \end{eqnarray}
are continuous on $I$ and positive in $(a,b)$.

Taking $\phi(t)=(t-a)^{\alpha^{q}}\,(b-t)^{\beta^{q}}>0$  on $(a,b)$, the function $v_{s}^{q}(t)$ satisfies the condition $(P_g)$ if $n-q-c_{q}$ is even with $k_{1}(s)=\tilde{k}_{1}^{q}(s)$ and $k_{2}(s)=\tilde{k}_{2}^{q}(s)$ and condition $(N_g)$  if $n-q-c_{q}$ is odd with $k_{1}(s)=-\tilde{k}_{2}^{q}(s)$ and $k_{2}(s)=-\tilde{k}_{1}^{q}(s)$.

\section{Study of the constant sign of the derivatives of Green's function}
In this section for any $q\in \{1,\dots,n-1\}$ be fixed,  we will give a proof of the main result that characterizes the constant sign of the partial derivative with respect to $t$ of order $q$ of the Green's function related to operator $T_{n}[M]$ in $X_{\{\sigma_{1},\dots,\sigma_{k}\}}^{\{\varepsilon_{1},\dots,\varepsilon_{n-k}\}}$.  

We distinguish three cases:
\begin{itemize}
	\item[(a)] $c_q\geq 1\;\;\text{and}\;\; d_q\geq 1$,
	\item[(b)] $c_{q}=n-q\;\; \text{and}\;\; d_{q}=0$,
	\item[(c)] $c_q=0,\;\; \text{and}\;\; d_{q}=n-q$,
\end{itemize} 
where $c_q$ and $d_q$ are defined in \eqref{cq} and \eqref{dq} and satisfy $c_q+d_q=n-q$.
\subsection{Case (a): {\boldmath $c_q+d_q=n-q$}, {\boldmath $c_q \geq 1$}  and {\boldmath $d_q \geq 1$} }
\begin{lemma}\label{indices}
	Let $q\in\{1,\dots,n-1\}$ be fixed. If $c_q+d_q=n-q$, $c_q\geq 1$ and $d_q\geq 1$, then the indices 	$z=k-(j-1)$ and $h=n-k-(r-1)$ are such that  \[\mu_{z}^{q}+q+\eta=n-1,\;\;\text{and}\;\; \rho_{h}^{q}+q+\gamma=n-1, \]
	with $j$ and $r$ defined in \eqref{j} and \eqref{r}, $\eta$ and $\gamma$ defined in \eqref{Ec::eta} and \eqref{Ec::gamma}.
\end{lemma}
\begin{proof}
	By the  definition of $\mu_{z}^{q}=\mu_{k-(j-1)}^{q}$ and $\rho_{h}^{q}=\rho_{n-k-(r-1)}^{q}$ we have that since $c_{q}\geq 1$, $\mu_{z}^{q}=\sigma_{k}-q\geq 0$ and since $d_{q}\geq 1$, $\rho_{h}^{q}=\varepsilon_{n-k}-q\geq 0$. Using Remark \ref{ref1}, we deduce that 
	$$\mu_{z}^{q}+q+\eta=\sigma_{k}-q+q+n-1-\sigma_{k}=n-1,$$
	and 
	$$\rho_{h}^{q}+q+\gamma=\varepsilon_{n-k}-q+q+n-1-\varepsilon_{n-k}=n-1.$$
\end{proof}
\begin{lemma}\label{pandz}
	Let $q\in\{1,\dots,n-1\}$ be fixed. If $c_q+d_q=n-q$, $c_q\geq 1$ and $d_q\geq 1$, then $z=p$ and $h=l$ with $z$ and $h$ defined in Lemma \ref{indices}, $p$ and $l$ defined in \eqref{p} and \eqref{l}.
\end{lemma}
\begin{proof}
	First, note that $z=|A_{q}|$ is the cardinal of the set $A_{q}=\{i\in\{1,\dots,k\}\,/\;\; \sigma_{i}\ge q\}$ and $p=|D_{q}|$ is the cardinal of the set $D_{q}=\{i\in\{1,\dots,k\}\,/\;\; \delta_{i}+q\le n-1 \}$. Moreover, by \eqref{cq}, it is clear that $c_{q}$ corresponds with the number of $\sigma$'s that are equal or greater than $q$, that is, $z=c_{q}$.
	
	On the other hand, we have that 
	\[\{\varepsilon_1,\dots,\varepsilon_{n-k},n-1-\delta_{k},\dots,\,n-1-\delta_{1}\}\equiv\{0,\dots,n-1\}.\]
	Among these $n$ elements (all the numbers between $0$ and $n-1$) there are $n-q$ elements that are equal or greater than $q$. By \eqref{dq}, we have that the number of $\varepsilon$'s that are equal or greater than $q$ is $d_{q}$. Therefore, the number of $\delta$'s such that $n-1-\delta_{i}\geq q$ is $n-q-d_{q}$, that is, $p=n-q-d_{q}$.
	Since, $c_{q}+d_{q}=n-q$, we have that $z=p$. In an analogous way, we can prove that $h=l$.
\end{proof}
\begin{remark}
	It should be noted that in the previous proof, since $j$ and $r$ depend  on $q$, we have that $z$ and $h$ also depend on $q$, but we will omit such dependence in the notation for the sake of simplicity. 
\end{remark}
\begin{remark}
	Indices $z$ and $h$ depend on the space $X_{\{\mu_{1}^{q},\dots,\mu_{k}^{q}\}}^{\{\rho_{1}^{q},\dots,\rho_{n-k}^{q}\}}$ and are unique. Indeed, for the rest of the components of space $X_{\{\mu_{1}^{q},\dots,\mu_{k}^{q}\}}^{\{\rho_{1}^{q},\dots,\rho_{n-k}^{q}\}}$ we have that $(\mu_{i}^{q}+q+\eta)\bmod{n}<n-1$ for $i\neq z$ with $i\in\{1,\dots,k\}$ and $(\rho_{l}^{q}+q+\gamma)\bmod{n}<n-1$ for $l\neq h$ with $l\in\{1,\dots,n-k\}$.
\end{remark}
Now, we present some results that provide sufficient conditions to ensure the constant sign of solutions of \eqref{Ec::T_n[M]} in the spaces $X_{\{\mu_{1}^{q},\dots,\mu_{k}^{q}\}}^{\{\rho_{1}^{q},\dots,\rho_{n-k}^{q}\}}$. Using Lemma \ref{lemma::1}, these results can be proved analogously to \cite[Proposition 6.7]{CabSaab} and  \cite[Proposition 6.9]{CabSaab}, respectively, and we omit them.

Let us consider the following spaces 
\begin{align}
	\label{space2}  X_{2}&:=\left\lbrace \begin{array}{cc}
		X_{\{\mu_ {1}^{q},\,\dots,\,\mu_{z-1}^{q},\,\mu_{z+1}^{q},\,\dots,\,\mu_{k}^{q}\}}^{\{\rho_{1}^{q},\,\dots,\,\rho_{n-k}^{q}|\beta^{q}\}}	\,,& \text{if $z\neq 1$,}\\&\\
		X_{\{\mu_ {2}^{q},\,\dots,\,\mu_{k}^{q}\}}^{\{\rho_{1}^{q},\,\dots,\,\rho_{n-k}^{q}|\beta^{q}\}}\,,& \text{if $z=1$,}\end{array}\right.\\ 	
	\label{space3} X_{3}&:=\left\lbrace \begin{array}{cc}
		X_{\{\mu_ {1}^{q},\,\dots,\,\mu_{z-1}^{q},\,\mu_{z+1}^{q},\,\dots,\mu_{k}^{q}|\alpha^{q} \}}^{\{\rho_{1}^{q},\,\dots,\, \rho_{n-k}^{q}\}}\,,& \text{if $z\neq 1$,}\\&\\
		X_{\{\mu_ {2}^{q},\,\dots,\,\mu_{k}^{q}|\alpha^{q}\}}^{\{\rho_{1}^{q},\,\dots,\,\rho_{n-k}^{q}\}}\,,& \text{if $z=1$,}\end{array}\right. 	
\end{align}
\begin{equation}\label{space22}
	X_{4}:=\left\lbrace \begin{array}{cc}
		X_{\{\mu_ {1}^{q},\,\dots,\,\mu_{k}^{q}|\alpha^{q}\}}^{\{\rho_{1}^{q},\,\dots,\,\rho_{h-1}^{q},\,\rho_{h+1}^{q},\,\dots,\,\rho_{n-k}^{q}\}}	\,,& \text{if $h\neq 1$,}\\&\\
		X_{\{\mu_ {1}^{q},\,\dots,\,\mu_{k}^{q}|\alpha^{q}\}}^{\{\rho_{2}^{q},\,\dots,\,\rho_{n-k}^{q}\}}\,,& \text{if $h=1$,}\end{array}\right. 
\end{equation}
and
\begin{equation} 
	\label{space33}
	X_{5}:=\left\lbrace \begin{array}{cc}
		X_{\{\mu_ {1}^{q},\,\dots,\,\mu_{k}^{q} \}}^{\{\rho_{1}^{q},\,\dots,\,\rho_{h-1}^{q},\, \rho_{h+1}^{q},\,\dots,\, \rho_{n-k}^{q}|\beta^{q}\}}\,,& \text{if $h\neq 1$,}\\&\\
		X_{\{\mu_ {1}^{q},\,\dots,\,\mu_{k}^{q}\}}^{\{\rho_{2}^{q},\,\dots,\,\rho_{n-k}^{q}|\beta^{q}\}}\,,& \text{if $h=1$,}\end{array}\right. 	
\end{equation}
where $z$ and $h$ are defined in Lemma \ref{indices}.

\begin{proposition}
	\label{P::1}
	Let $q\in \{1,\dots,n-1\}$. Suppose that ${\{\sigma_1,\dots,\sigma_k\}}-{\{\varepsilon_1,\dots,\varepsilon_{n-k}\}}$ satisfies condition $(N_a)$, $c_q+d_q=n-q$, $c_{q}\geq 1$ and $d_{q}\geq 1$. If $u\in C^n(I)$ is a solution of \eqref{Ec::T_n[M]} on $(a,b)$, satisfying the boundary conditions:
	\begin{eqnarray}
		\label{Ec::sigma1} u^{(\mu_ {1}^{q})}(a)=\cdots=u^{(\mu_{z-1}^{q})}(a)=u^{(\mu_{z+1}^{q})}(a)=&\cdots&=u^{(\mu_{k}^{q})}(a)\hspace{0.3cm}=0\,,\\
		\label{Ec::epsilon1} u^{(\rho_{1}^{q})}(b)=&\cdots&=u^{(\rho_{n-k}^{q})}(b)=0\,,
	\end{eqnarray}
	then it does not have any zero on $(a,b)$ provided that one of the following assertions is satisfied:
	\begin{itemize}
		\item Let $n-k$ be even:
		\begin{itemize}
			\item If $k>1$, $\mu_{z}^{q}\neq z-1$ and  $M\in[\lambda_3^{q},\lambda_2^{q}]$, where:
			\begin{itemize}
				\item $\lambda_3^{q}<0$ is the biggest negative eigenvalue of $T_{n}[0]$ in $X_{3}$.
				\item $\lambda_2^{q}>0$ is the least positive eigenvalue of $T_{n}[0]$ in $X_{2}$.
			\end{itemize}
			\item If $k>1$, $\mu_{z}^{q}=z-1$ and  $M\in[\lambda_{1},\lambda_2^{q}]$, where:
			\begin{itemize}
				\item $\lambda_{1}<0$ is the biggest negative eigenvalue of $T_{n}[0]$ in $X_{\{\sigma_{1},\dots,\sigma_{k}\}}^{\{\varepsilon_1,\dots,\varepsilon_{n-k}\}}.$
				\item $\lambda_2^{q}>0$ is the least positive eigenvalue of $T_{n}[0]$ in $X_{2}$.
			\end{itemize}
			\item If $k=1$, $\mu_{1}^{q}\neq 0$ and $M\in[\lambda_3^{q},+\infty)$, where:
			\begin{itemize}
				\item $\lambda_3^{q}<0$ is the biggest negative eigenvalue of $T_n[0]$ in $X_{\{\alpha^{q}\}}^{\{\rho_{1}^{q},\dots,\rho_{n-1}^{q}\}}$, where $\alpha^{q}=0$.
			\end{itemize}
			\item If $k=1$, $\mu_{1}^{q}= 0$ and $M\in[\lambda_{1},+\infty)$, where:
			\begin{itemize}
				\item $\lambda_{1}<0$ is the biggest negative eigenvalue of $T_{n}[0]$ in $X_{\{\sigma_{1}\}}^{\{\varepsilon_1,\dots,\varepsilon_{n-1}\}}.$
			\end{itemize}
		\end{itemize}
		\item Let $n-k$ be  odd:
		\begin{itemize}
			\item If $k>1$, $\mu_{z}^{q}\neq z-1$ and  $M\in[\lambda_2^{q},\lambda_3^{q}]$, where:
			\begin{itemize}
				\item $\lambda_3^{q}>0$ is the least positive eigenvalue of $T_{n}[0]$ in $X_{3}$.
				\item $\lambda_2^{q}<0$ is the biggest negative eigenvalue of $T_{n}[0]$ in $X_{2}$.
			\end{itemize}
			\item If $k>1$, $\mu_{z}^{q}= z-1$ and  $M\in[\lambda_2^{q},\lambda_{1}]$, where:
			\begin{itemize}
				\item $\lambda_{1}>0$ is the least positive eigenvalue of $T_{n}[0]$ in $X_{\{\sigma_{1},\dots,\sigma_{k}\}}^{\{\varepsilon_1,\dots,\varepsilon_{n-k}\}}.$.
				\item $\lambda_2^{q}<0$ is the biggest negative eigenvalue of $T_{n}[0]$ in $X_{2}$.
			\end{itemize}
			\item If $k=1$, $\mu_{1}^{q}\neq 0$ and $M\in(-\infty,\lambda_3^{q}]$, where:
			\begin{itemize}
				\item $\lambda_3^{q}>0$ is the least positive eigenvalue of $T_n[0]$ in $X_{\{\alpha^{q}\}}^{\{\rho_{1}^{q},\dots,\rho_{n-1}^{q}\}}$, where $\alpha^{q}=0$.
			\end{itemize}
			\item If $k=1$, $\mu_{1}^{q}=0$ and $M\in(-\infty,\lambda_1]$, where:
			\begin{itemize}
				\item $\lambda_1>0$ is the least positive eigenvalue of $T_n[0]$ in $X_{\{\sigma_{1}\}}^{\{\varepsilon_{1},\dots,\varepsilon_{n-1}\}}$.
			\end{itemize}
		\end{itemize}
	\end{itemize}
\end{proposition}
\begin{proposition}
	\label{P::2}
	Let $q\in \{1,\dots,n-1\}$. Suppose that ${\{\sigma_1,\dots,\sigma_k\}}-{\{\varepsilon_1,\dots,\varepsilon_{n-k}\}}$ satisfies condition $(N_a)$, $c_q+d_q=n-q$, $c_{q}\geq 1$ and $d_{q}\geq 1$.  If $u\in C^n(I)$ is a  solution of \eqref{Ec::T_n[M]} on $(a,b)$ satisfying the boundary conditions:
	\begin{eqnarray}
		\label{Ec::sigma2} u^{(\mu_{1}^{q})}(a)=&\cdots&=u^{(\mu_{k}^{q})}(a)\hspace{0.3cm}=0\,,\\
		\label{Ec::epsilon2} u^{(\rho_{1}^{q})}(b)=&\cdots&=u^{(\rho_{h-1}^{q})}(b)=u^{(\rho_{h+1}^{q})}(b)=\cdots=u^{(\rho_{n-k}^{q})}(b)=0\,,
	\end{eqnarray}
	then it does not have any zero on $(a,b)$ provided that one of the following assertions is satisfied:
	\begin{itemize}
		\item Let $n-k$ be even:
		\begin{itemize}
			\item If $\rho_{h}^{q}\neq h-1$ and   $M\in[\lambda_5^{q},\lambda_4^{q}]$, where:
			\begin{itemize}
				\item $\lambda_5^{q}<0$ is the biggest negative eigenvalue of $T_{n}[0]$ in $X_{5}$.
				\item $\lambda_4^{q}>0$ is the least positive eigenvalue of $T_{n}[0]$ in $X_{4}.$	
			\end{itemize}
			
			\item If $\rho_{h}^{q}=h-1$ and   $M\in[\lambda_{1},\lambda_4^{q}]$, where:
			\begin{itemize}
				\item $\lambda_{1}<0$ is the biggest negative eigenvalue of $T_{n}[0]$ in $X_{\{\sigma_1,\dots,\sigma_{k}\}}^{\{\varepsilon_1,\dots,\varepsilon_{n-k}\}}.$ 
				\item $\lambda_4^{q}>0$ is the least positive eigenvalue of $T_{n}[0]$ in $X_{4}$.
			\end{itemize}

		\end{itemize}
		
		\item Let $n-k$ be odd:
		\begin{itemize}
			\item If $k<n-1$, $\rho_{h}^{q}\neq h-1$ and $M\in[\lambda_4^{q},\lambda_5^{q}]$, where:
			\begin{itemize}
				\item $\lambda_5^{q}>0$ is the least positive eigenvalue of $T_{n}[0]$ in $X_{5}$.
				\item $\lambda_4^{q}<0$ is the biggest negative eigenvalue of $T_{n}[0]$ in  $X_{4}$.
			\end{itemize}
			
			\item If $k<n-1$, $\rho_{h}^{q}= h-1$ and $M\in[\lambda_4^{q},\lambda_{1}]$, where:
			\begin{itemize}
				\item $\lambda_{1}>0$ is the least positive eigenvalue of $T_{n}[0]$ in $X_{\{\sigma_1,\dots,\sigma_{k}\}}^{\{\varepsilon_1,\dots,\varepsilon_{n-k}\}}.$
				\item $\lambda_4^{q}<0$ is the biggest negative eigenvalue of $T_{n}[0]$ in  $X_{4}$.
			\end{itemize}
			
			\item If $k=n-1$, $\rho_{1}^{q}\neq 0$ and $M\in(-\infty,\lambda_5^{q}]$, where:
			\begin{itemize}
				\item $\lambda_5^{q}>0$ is the least positive eigenvalue of $T_{n}[0]$ in $X_{\{\mu_{1}^{q},\dots,\mu_{n-1}^{q}\}}^{\{\beta^{q}\}}$, where $\beta^{q}=0$.
			\end{itemize}
			\item If $k=n-1$, $\rho_{1}^{q}=0$ and $M\in(-\infty,\lambda_{1}]$, where:	
			\begin{itemize}
				\item $\lambda_{1}>0$ is the least positive eigenvalue of $T_{n}[0]$ in $X_{\{\sigma_1,\dots,\sigma_{n-1}\}}^{\{\varepsilon_{1}\}}$.
			\end{itemize}	
			
		\end{itemize}
	\end{itemize}
\end{proposition}

\begin{example}
	Let us consider the operator $T_{4}[M]$ with the boundary conditions:
	$$u'(0)=u'''(0)=u(1)=u'(1)=0.$$
	The hypotheses of Propositions \ref{P::1} and \ref{P::2} are satisfied only for the value $q=1$. Moreover, $z=2$, $h=1$, and $v_{s}^1[M]$ satisfies the conditions of the space $X_{\{0,2\}}^{\{0,3\}}$.
	
	From Proposition \ref{P::1}, we can affirm that any nontrivial solution of the problem 
	\begin{equation}\label{prob1}
		T_4[M]=0,\; t\in [0,1],\;\; u(0)=u(1)=u'''(1)=0\,,
	\end{equation}
	does not have any zero on $(0,1)$ for $M\in[\lambda_{3}^{1},\lambda_{2}^{1}]$, where $\lambda_3^{1}<0$ is the biggest negative eigenvalue of $T_{4}[0]$ in $X_{\{0,1\}}^{\{0,3\}}$ and $\lambda_2^{1}>0$ is the least positive eigenvalue of $T_{4}[0]$ in $X_{\{1\}}^{\{0,1,3\}}$.
	
	The eigenvalues of $T_{4}[0]$ in $X_{\{0,1\}}^{\{0,3\}}$ are given by $-\lambda^{4}$, where $\lambda $ is a positive solution of $\sin\left(\lambda\right)=0\,$. The smallest positive solution of this equation is $\pi$, so  $\lambda_{3}^{1}=-\pi^{4}$ is the biggest negative eigenvalue of $T_{4}[0]$ in $X_{\{0,1\}}^{\{0,3\}}$.
	
	Similarly, the eigenvalues of $T_{4}[0]$ in $X_{\{1\}}^{\{0,1,3\}}$ are given by $\lambda^{4}$, where $\lambda$ is a positive solution of $$\left(-1+ e^{\sqrt{2}\,m}\right)\, \cos\left(\frac{m}{\sqrt{2}}\right)+\left(1+e^{\sqrt{2}\,m}\right)\, \sin\left(\frac{m}{\sqrt{2}}\right)=0.$$
	
	Denoting by $m_{2}$ the smallest positive solution of this equation, we deduce that\linebreak $\lambda_{2}^{1}=m_{2}^{4}\approx 3.34^4$ is the least positive eigenvalue of $T_{4}[0]$ in $X_{\{0\}}^{\{0,1,3\}}$.
	
	It is easy to verify that the solutions of problem \eqref{prob1} are given as multiples of the following  expression:
	\begin{itemize}
		\item If $ M=-m^4<0$,
		\begin{equation*}\begin{split}
				u(t)= & \cosh(m-m\, t)\, \sin(m)- \sin(m\,t)- \cosh(m)\, \sin(m-m\,t)-\cos(m-m\, t)\, \sinh(m) \\ & + \sinh(m\,t)+ \cos(m)\, \sinh(m-m\,t).
		\end{split}\end{equation*}
		\item If $M=0$,
		\[u(t)=t^2-t. \]
		\item If $M=m^4>0$,
		\begin{equation*}\begin{split}
				u(t)=&\,  e^{-\frac{m\,t}{\sqrt{2}}}\left(e^{\sqrt{2}\,m}\left(-1+ e^{\sqrt{2}\,m\,t}\right)\, \cos\left(\frac{m(-2+t)}{\sqrt{2}}\right)-e^{\sqrt{2}\,m}\left(-1 + e^{\sqrt{2}\,m}\right)\, \cos\left(\frac{m\,t}{\sqrt{2}}\right)\right) \\
				& + e^{-\frac{m\,t}{\sqrt{2}}} \, \left(-1+e^{\sqrt{2}\,m}\right)\left(-e^{\sqrt{2}\,m}+e^{\sqrt{2}\,m\,t}\right)\, \sin\left(\frac{m\,t}{\sqrt{2}}\right).
		\end{split}\end{equation*}
	\end{itemize}
	%	\[\left\lbrace \begin{aligned}
		%		&\,\cosh(m-m\, t)\, \sin(m)- \sin(m\,t)- \cosh(m)\, \sin(m-m\,t)\\
		%		&\,-\cos(m-m\, t)\, \sinh(m) + \sinh(m\,t)+ \cos(m)\, \sinh(m-m\,t)\,,\\ 
		%		&\,\text{if}\;\; M=-m^4<0\,,\\
		%		&\,t^2-t\,,\;\;\text{if}\;\; M=0\,,\\
		%		&\,e^{-\frac{m\,t}{\sqrt{2}}}\left(e^{\sqrt{2}\,m}\left(-1+ e^{\sqrt{2}\,m\,t}\right)\, \cos\left(\frac{m(-2+t)}{\sqrt{2}}\right)-e^{\sqrt{2}\,m}\left(-1 + e^{\sqrt{2}\,m}\right)\, \cos\left(\frac{m\,t}{\sqrt{2}}\right)\right.\\
		%		&\,\left.+\left(-1+e^{\sqrt{2}\,m}\right)\left(-e^{\sqrt{2}\,m}+e^{\sqrt{2}\,m\,t}\right)\, \sin\left(\frac{m\,t}{\sqrt{2}}\right)\right)\,,\\	
		%		&\,\text{if}\;\; M=m^4>0.\,\end{aligned}\right. \]
	Analogously, from Proposition \ref{P::2}, we conclude that any solution of the problem \begin{equation}\label{prob2}
		T_{4}[M]\,u(t)=0,\; t\in[0,1],\;\; u(0)=u''(0)=u'''(1)=0,
	\end{equation}
	does not have any zero on $(0,1)$ for $M\in[\lambda_{1},\lambda_{4}^{1}]$, where $\lambda_1<0$ is the biggest negative eigenvalue of $T_{4}[0]$ in $X_{\{1,3\}}^{\{0,1\}}$ and $\lambda_4^{1}>0$ is the least positive eigenvalue of $T_{4}[0]$ in $X_{\{0,1,2\}}^{\{3\}}$.
	
	The eigenvalues of $T_{4}[0]$ in $X_{\{1,3\}}^{\{0,1\}}$ are given by $-\lambda^{4}$, where $\lambda $ is a positive solution of $$\sin(\lambda) + \cos(\lambda)\, \tanh(\lambda)=0\,.$$
	Now, if $m_{1}$ is the smallest positive solution of this equation, then  $\lambda_{1}=-m_{1}^{4}\approx -2.36^4$ is the biggest negative eigenvalue of $T_{4}[0]$ in $X_{\{0,1\}}^{\{0,3\}}$.
	
	On the other hand, the eigenvalues of $T_{4}[0]$ in $X_{\{0,1,2\}}^{\{3\}}$ are given by $\lambda^{4}$, where $\lambda$ is a positive solution of $\cos\left(\frac{m}{\sqrt{2}}\right)=0.$ The smallest positive solution of this equation is $\frac{\sqrt{2}}{2}\pi$, so  $\lambda_{4}^{1}=\left(\frac{\sqrt{2}}{2}\pi\right)^{4}$ is the least positive eigenvalue of $T_{4}[0]$ in $X_{\{0,1,2\}}^{\{3\}}$.
	
	It is immediate to show that solutions of problem \eqref{prob2} are given as multiples of:	
	\begin{itemize}
		\item If $M=-m^4<0$,
		\[u(t)=\sin(m\, t)+ \cos(m)\, \operatorname{sech}(m) \sinh(m\, t).\]
		\item If $M=0$,
		\[u(t)=t.\]
		\item If $M=m^4>0$,
		\begin{equation*}\begin{split}
				u(t)= & \, e^{-\frac{m\,t}{\sqrt{2}}}\left(\left(-1+ e^{\sqrt{2}\,m\,(1+t)}\right)\, \cos\left(\frac{m(-1+t)}{\sqrt{2}}\right)+\left(-e^{\sqrt{2}\,m} + e^{\sqrt{2}\,m\,t}\right)\, \cos\left(\frac{m(1+t)}{\sqrt{2}}\right)\right) \\
				& + e^{-\frac{m\,t}{\sqrt{2}}} \left(\left(1+ e^{\sqrt{2}\,m\,(1+t)}\right)\, \sin\left(\frac{m(-1+t)}{\sqrt{2}}\right)+\left(e^{\sqrt{2}\,m} + e^{\sqrt{2}\,m\,t}\right)\, \sin\left(\frac{m(1+t)}{\sqrt{2}}\right)\right).
		\end{split}\end{equation*}
	\end{itemize}
	%	\[\left\lbrace \begin{aligned}
		%		&\,\sin(m\, t)+ \cos(m)\, \operatorname{sech}(m) \sinh(m\, t)\,,\;\;\text{if}\;\; M=-m^4<0\,,\\
		%		&\,t\,,\;\;\text{if}\;\; M=0\,,\\
		%		&\,e^{-\frac{m\,t}{\sqrt{2}}}\left(\left(-1+ e^{\sqrt{2}\,m\,(1+t)}\right)\, \cos\left(\frac{m(-1+t)}{\sqrt{2}}\right)+\left(-e^{\sqrt{2}\,m} + e^{\sqrt{2}\,m\,t}\right)\, \cos\left(\frac{m(1+t)}{\sqrt{2}}\right)\right.\\
		%		&\,\left.+\left(1+ e^{\sqrt{2}\,m\,(1+t)}\right)\, \sin\left(\frac{m(-1+t)}{\sqrt{2}}\right)+\left(e^{\sqrt{2}\,m} + e^{\sqrt{2}\,m\,t}\right)\, \sin\left(\frac{m(1+t)}{\sqrt{2}}\right)\right)\,,\\	
		%		&\,\text{if}\;\; M=m^4>0.\,\end{aligned}\right. \]
\end{example}

\begin{remark}
	It should be noted that in the above propositions we have used the fact proven in Lemma \ref{spectrum} that the first nonzero eigenvalues of the spaces $X_{\{\sigma_1,\dots,\sigma_{k}\}}^{\{\varepsilon_1,\dots,\varepsilon_{n-k}\}}$ and $X_{\{\mu_{1}^{q},\dots,\mu_{k}^{q}\}}^{\{\rho_{1}^{q},\dots,\rho_{n-k}^{q}\}}$ coincide.
\end{remark}

Let us define the following spaces 
\begin{equation}
	\label{space2adjunto}  X_{2}^*:=\left\lbrace \begin{array}{cc}
		X_{\{\tau_{1},\,\dots,\,\tau_{n-k}|\eta\}}^{\{\delta_{1},\,\dots,\,\delta_{z-1},\,\delta_{z+1},\,\dots,\,\delta_{k}\}}	\,,& \text{if $z\neq 1$,}\\&\\
		X_{\{\tau_ {1},\,\dots,\,\tau_{n-k}|\eta\}}^{\{\delta_{2},\,\dots,\,\delta_{k}\}}\,,& \text{if $z=1$,}\end{array}\right.
\end{equation}
and 
\begin{equation}
	\label{space4adjunto} X_{4}^* :=\left\lbrace \begin{array}{cc}
		X_{\{\tau_ {1},\,\dots,\,\tau_{h-1},\,\tau_{h+1},\,\dots,\,\tau_{n-k}\}}^{\{\delta_{1},\,\dots,\,\delta_{k}|\gamma\}}	\,,& \text{if $h\neq 1$,}\\&\\
		X_{\{\tau_ {2},\,\dots,\,\tau_{n-k}\}}^{\{\delta_{1},\,\dots,\,\delta_{k}|\gamma\}}\,,& \text{if $h=1$.}\end{array}\right.
\end{equation}

\begin{lemma}\label{adjuntos1}
	$\lambda\neq 0$ is an eigenvalue of $T_{n}[0]$ in $X_{2}^{*}$ defined in \eqref{space2adjunto} if and only if $\lambda$ is an eigenvalue of $T_{n}[0]$ in  $X_{2}$. In particular, $(\lambda_2^*)^q=\lambda_{2}^q$ where $(\lambda_2^*)^{q}\neq 0$ is the least positive (biggest negative) eigenvalue of $T_{n}[0]$ in $X_{2}^*$  and $\lambda_{2}^q\neq 0$ is the least positive (biggest negative) eigenvalue of $T_{n}[0]$ in $X_{2}$.
\end{lemma}
\begin{proof}
	Suppose that $z\neq 1$ (the case $z=1$ is proved similarly). By the definition of adjoint space of  $X_{\{\mu_{1}^{q},\,\dots,\,\mu_{k}^{q}\}}^{\{\rho_{1}^{q},\,\dots,\,\rho_{n-k}^{q}\}}$ we have that 
	\[\{\mu_1^{q},\,\dots,\,\mu_{z-1}^{q},\,\mu_{z}^{q},\,\mu_{z+1}^{q},\,\dots,\,\mu_k^{q},\,n-1-\tau_{n-k}^{q},\,\dots,\, n-1-\tau_{1}^{q}\}=\{0,\,\dots,\,n-1\},\] 
	and 
	\[\{\rho_1^{q},\,\dots,\,\rho_{n-k}^{q},\, n-1-\delta_{k}^{q},\,\dots,\, n-1-\delta_{1}^{q}\}=\{0,\,\dots,\, n-1\}.\] 
	Since, $\mu_{z}^{q}=\sigma_{k}-q$ and $\beta^{q}=n-1-\delta_{k}^{q}$ we deduce that 
	\begin{equation}\label{equivalencia1}
		\begin{tikzcd}
			\arrow{r}{*} X_{ \{\mu_1^{q},\,\dots,\,\mu_{z-1}^{q},\,\mu_{z+1}^{q},\,\dots,\,\mu_k^{q}\}}^{\{\rho_1^q,\dots,\rho_{n-k}^q|\beta^{q}\}} \arrow{r}{*} & X_{\{\tau_{1}^{q},\dots,\tau_{n-k}^{q}|n-1-\sigma_{k}+q\}}^{\{\delta_{1}^{q},\dots,\delta_{k-1}^{q}\}}.
		\end{tikzcd}
	\end{equation}
	
	Using that $\eta=n-1-\sigma_{k}$, the definition of  $X_{\{\tau_{1}^{q},\dots,\tau_{n-k}^{q}\}}^{\{\delta_{1}^{q},\dots,\delta_{k}^{q}\}}$ and the  diagram \eqref{diagrama} we infer that
	\begin{equation*}
		\begin{tikzcd}
			X_{\{\tau_{1}^{q},\dots,\tau_{n-k}^{q}|n-1-\sigma_{k}+q\}}^{\{\delta_{1}^{q},\dots,\delta_{k-1}^{q}\}} \arrow{r}{\phi^{q}}& X_{\{\tau_{1},\,\dots,\,\tau_{n-k}|\eta\}}^{\{\delta_{1},\,\dots,\,\delta_{p-1},\,\delta_{p+1},\,\dots,\,\delta_{k}\}}.
		\end{tikzcd}
	\end{equation*}
	From Lemma \ref{pandz} we know that $p=z$ and therefore \begin{equation}\label{equivalencia2}
		\begin{tikzcd}
			X_{\{\tau_{1}^{q},\dots,\tau_{n-k}^{q}|n-1-\sigma_{k}+q\}}^{\{\delta_{1}^{q},\dots,\delta_{k-1}^{q}\}} \arrow{r}{\phi^{q}}& X_{\{\tau_{1},\,\dots,\,\tau_{n-k}|\eta\}}^{\{\delta_{1},\,\dots,\,\delta_{z-1},\,\delta_{z+1},\,\dots,\,\delta_{k}\}}.
		\end{tikzcd}
	\end{equation}
	Considering Lemma \ref{spectrum} and the fact that the adjoint spaces have the same eigenvalues, from  \eqref{equivalencia1} and \eqref{equivalencia2}, we obtain that the spaces $X_{ \{\mu_1^{q},\,\dots,\,\mu_{z-1}^{q},\,\mu_{z+1}^{q},\,\dots,\,\mu_k^{q}\}}^{\{\rho_1^q,\dots,\rho_{n-k}^q|\beta^{q}\}}$ and $X_{\{\tau_{1},\,\dots,\,\tau_{n-k}|\eta\}}^{\{\delta_{1},\,\dots,\,\delta_{z-1},\,\delta_{z+1},\,\dots,\,\delta_{k}\}}$ have the same eigenvalues. In particular, $(\lambda_2^*)^q=\lambda_{2}^q$.
\end{proof}
Using similar argument to the previous result, we arrive to the next one.
\begin{lemma}\label{adjuntos2}
	$\bar \lambda\neq 0$ is an eigenvalue of $T_{n}[0]$ in $X_{4}^{*}$ defined in \eqref{space4adjunto} if and only if $\bar \lambda$ is an eigenvalue of $T_{n}[0]$ in  $X_{4}$. In particular, $(\lambda_4^*)^q=\lambda_{4}^q$ where $(\lambda_4^*)^{q}\neq 0$ is the least positive (biggest negative) eigenvalue of $T_{n}[0]$ in $X_{4}^*$  and $\lambda_{4}^q\neq 0$ is the least positive (biggest negative) eigenvalue of $T_{n}[0]$ in $X_{4}$.
\end{lemma}
Next, we state and prove the main result, which gives the characterization of the set of parameter $M$ where $v_{s}^{q}[M]$ is of constant sign.
\begin{theorem}\label{T::IP}
	Let $q\in \{1,\dots,n-1\}$. Suppose that ${\{\sigma_1,\dots,\sigma_k\}}-{\{\varepsilon_1,\dots,\varepsilon_{n-k}\}}$ satisfies condition $(N_a)$, $c_q+d_q=n-q$, $c_{q}\geq 1$ and $d_{q}\geq 1$. The following properties are fulfilled:
	\begin{itemize}
		\item Let $2\le k\le n-2$ and $n-k$ even:
		\begin{itemize}
			\item If $n-q-c_q$ is even, then $v_{s}^{q}[M]$ is strongly positive on $I\times I$ if, and only if, $M\in(\lambda_1,\lambda^{q}]$, and if $n-q-c_q$ is odd, then $v_{s}^{q}[M]$ is strongly negative on $I\times I$ if, and only if, $M\in(\lambda_1,\lambda^{q}]$ where:
			\begin{itemize}
				\item[*] $\lambda_1<0$ is the biggest negative eigenvalue of $T_{n}[0]$ in $X_{\{\sigma_1,\dots,\sigma_{k}\}}^{\{\varepsilon_1,\dots,\varepsilon_{n-k}\}}.$ 
				\item [*] $\lambda^{q}>0$ is the minimum between:
				\begin{itemize}
					\item [·] $\lambda_2^{q}>0$, the least positive eigenvalue of $T_{n}[0]$ in $X_{2}$ defined in \eqref{space2}.
					\item [·] $\lambda_4^{q}>0$, the least positive eigenvalue of $T_{n}[0]$ in $X_{4}$ defined in \eqref{space22}.
					
				\end{itemize}
			\end{itemize}
		\end{itemize}
		\item Let $2\leq k\leq n-2$ and $n-k$ odd:
		\begin{itemize}
			\item If $n-q-c_q$ is even, then $v_{s}^{q}[M]$ is strongly positive on $I\times I$ if, and only if, $M\in[\lambda^{q},\lambda_1)$, and if $n-q-c_q$ is odd, then $v_{s}^{q}[M]$ is strongly negative on $I\times I$ if, and only if, $M\in[\lambda^{q},\lambda_1)$ where:
			\begin{itemize}
				\item[*] $\lambda_1>0$ is the least positive eigenvalue of $T_{n}[0]$ in $X_{\{\sigma_1,\dots,\sigma_{k}\}}^{\{\varepsilon_1,\dots,\varepsilon_{n-k}\}}.$ 
				\item [*] $\lambda^{q}<0$ is the maximum between:
				\begin{itemize}
					\item [·] $\lambda_2^{q}<0$, the biggest negative eigenvalue of $T_{n}[0]$ in $X_{2}.$
					\item [·] $\lambda_4^{q}<0$, the biggest negative eigenvalue of $T_{n}[0]$ in $X_{4}$.
				\end{itemize}
			\end{itemize}
		\end{itemize}
		\item Let $k=1$ and $n>2$ odd:
		\begin{itemize}
			\item If $n-q-1$ is even, then $v_{s}^{q}[M]$ is strongly positive on $I\times I$ if, and only if, $M\in(\lambda_1,\lambda_2^{q}]$, and if $n-q-1$ is odd, then $v_{s}^{q}[M]$ is strongly negative on $I\times I$ if, and only if, $M\in(\lambda_1,\lambda_2^{q}]$ where:
			\begin{itemize}
				\item[*] $\lambda_1<0$ is the biggest negative eigenvalue of $T_{n}[0]$ in $X_{\{\sigma_1\}}^{\{\varepsilon_1,\dots,\varepsilon_{n-1}\}}.$ 
				\item [*] $\lambda_{2}^{q}>0$ is the least positive eigenvalue of $T_{n}[0]$ in $X_{2}$.
			\end{itemize}
		\end{itemize}
		\item Let $k=1$ and $n>2$ even:
		\begin{itemize} 
			\item If $n-q-1$ is even, then $v_{s}^{q}[M]$ is strongly positive on $I\times I$ if, and only if, $M\in[\lambda_2^{q},\lambda_1)$, and if $n-q-1$ is odd, then $v_{s}^{q}[M]$ is strongly negative on $I\times I$ if, and only if, $M\in[\lambda_2^{q},\lambda_1)$, where:
			\begin{itemize}
				\item[*] $\lambda_1>0$ is the least positive eigenvalue of $T_{n}[0]$ in $X_{\{\sigma_1\}}^{\{\varepsilon_1,\dots,\varepsilon_{n-1}\}}$. 
				\item [*] $\lambda_{2}^{q}<0$ is the biggest negative eigenvalue of $T_{n}[0]$ in $X_{2}$.
			\end{itemize}
		\end{itemize}
		\item If $k=n-1$, $n>2$ and $n-q-c_q=1$, then $v_{s}^{q}[M]$ is strongly negative on $I\times I$ if, and only if, $M\in[\lambda_2^{q},\lambda_1)$, where
		\begin{itemize}
			\item[*]$\lambda_1>0$ is the least positive eigenvalue of $T_{n}[0]$ in $X_{\{\sigma_1,\dots,\sigma_{n-1}\}}^{\{\varepsilon_1\}}.$
			\item[*] $\lambda_2^{q}<0$ is the biggest negative eigenvalue of $T_{n}[0]$ in $X_{2}$.  
		\end{itemize}
	\end{itemize}
\end{theorem}
\begin{proof}
	First, it has been shown in Theorem \ref{resultSigno} that $v_{s}^{q}$ satisfies the property $(P_{g})$ if $n-q-c_q$ is even and the property $(N_{g})$ if $n-q-c_q$ is odd. In addition, from Lemma~\ref{lemma:monotonia_derivadas} we know that if $v_{s}^{q}[M]$ has constant sign then the first eigenvalue $\lambda_{1}$ of $T_{n}[0]$ in $X_{\{\sigma_1,\dots,\sigma_{k}\}}^{\{\varepsilon_1,\dots,\varepsilon_{n-k}\}}$ is one of the extremes of the interval where $v_{s}^{q}[M]$ holds that sign.
	
	Again, taking into account  Lemma~\ref{lemma:monotonia_derivadas} on the monotony of  $v_{s}^{q}[M]$ with respect to $M$, the constant sign starts/ends at the first eigenvalue $\lambda_{1}$ of $T_{n}[0]$ in $X_{\{\sigma_1,\dots,\sigma_{k}\}}^{\{\varepsilon_1,\dots,\varepsilon_{n-k}\}}$. Thus, from Theorems \ref{signo} and \ref{resultSigno} we conclude that:
	\begin{itemize}
		\item If $n-k$ and $n-q-c_q$ are even, and $M\leq 0$, then $v_{s}^{q}[M]$ is strongly positive on $I\times I$ if, and only if, $M\in(\lambda_1,0]$.
		\item If $n-k$ is even, $n-q-c_q$ is odd and $M\leq 0$, then $v_{s}^{q}[M]$ is strongly negative on $I\times I$ if, and only if, $M\in(\lambda_1,0]$.
		\item  If $n-k$ is odd, $n-q-c_q$ is even and $M\geq 0$, then $v_{s}^{q}[M]$ is strongly positive on $I\times I$ if, and only if, $M\in[0,\lambda_1)$.
		\item  If $n-k$ and $n-q-c_q$ are odd, and $M\geq 0$, then $v_{s}^{q}[M]$ is strongly negative on $I\times I$ if, and only if, $M\in[0,\lambda_1)$.
	\end{itemize}
	
	We now need to determine the other extreme of the constant sign interval of $v_{s}^{q}[M]$ using Definition \ref{def:::1}, Propositions \ref{P::1} and \ref{P::2}. We divide the proof into five steps assuming that $n-k$ and $n-q-c_q$ are even, and $2\leq k\leq n-2$. For the rest of cases the proof is done in an analogous way.
	
	\begin{itemize}
		\item[Step 1.] Study of $v_{s}^{q}[M]$ at $s=a$.
		\item[Step 2.] Study of $v_{s}^{q}[M]$ at $s=b$.
		\item[Step 3.] Study of $v_{s}^{q}[M]$ at $t=a$.
		\item[Step 4.] Study of $v_{s}^{q}[M]$ at $t=b$.
		\item[Step 5.] Study of $v_{s}^{q}[M]$  on $(a,b)\times(a,b)$.
	\end{itemize}
	
	Let us denote
	\[g_M(t,s)=\left\lbrace \begin{array}{cc}
		g_M^1(t,s)\,,&a\leq s\leq t\leq b\,,\\\\
		g_M^2(t,s)\,,&a<t<s<b\,,\end{array}\right. \]
	as the related Green's function of $T_{n}[M]$ in $X_{\{\sigma_1,\dots,\sigma_k\}}^{\{\varepsilon_1,\dots,\varepsilon_{n-k}\}}.$
	
	\begin{itemize}
		\item[\textbf{Step 1.}] Study of the function $v_{s}^{q}[M]$ at $s=a$.
	\end{itemize}
	Let us consider the function  $$w_{M,q}(t)=\dfrac{\partial^\eta}{\partial s^\eta}\Big(\dfrac{\partial^{q}}{\partial t^{q}} g_M^1(t,s)\Big)_{\mid s=a},$$ where $\eta$ has been defined in \eqref{Ec::eta}.
	
	If $\eta>0$, since $\{0,\dots,\eta-1\}\subset \{\tau_{1},\dots,\tau_{n-k}\}$ using the equality \eqref{gg2} we obtain that
	\[g_M^1(t,a)=\dfrac{\partial}{\partial s}  g_M^1(t,a)=\cdots=\dfrac{\partial^{\eta-1}}{\partial s^{\eta-1}} g_M^1(t,a)=0\,,\;\; \text{for all}\; t\in (a,b]. \]
	Therefore, for all $t\in (a,b]$ and $q\in \{1,\dots,n-1\}$ we infer that 
	\[\dfrac{\partial^{q}}{\partial t^{q}} g_M^1(t,a)=\dfrac{\partial^{q}}{\partial t^{q}} \Big(\dfrac{\partial}{\partial s}  g_M^1(t,s)\Big)_{\mid s=a}=\cdots=\dfrac{\partial^{q}}{\partial t^{q}}\Big(\dfrac{\partial^{\eta-1}}{\partial s^{\eta-1}} g_M^1(t,s)\Big)_{\mid s=a}=0\,,\]
	equivalently,
	\[\dfrac{\partial^{q}}{\partial t^{q}} g_M^1(t,a)=\dfrac{\partial}{\partial s} \Big(\dfrac{\partial^{q}}{\partial t^{q}} g_M^1(t,s)\Big)_{\mid s=a}=\cdots=\dfrac{\partial^{\eta-1}}{\partial s^{\eta-1}} \Big(\dfrac{\partial^{q}}{\partial t^{q}} g_M^1(t,s)\Big)_{\mid s=a}=0\,.\]
	
	Note that it is necessary that $w_{M,q}>0$ to guarantee the strongly positive sign of the function $v_{s}^{q}[M]$. Indeed, if there exists $t^*\in [a,b]$, such that $w_{M,q}(t^*)<0$, then there exists $\rho(t^*)>a$ such that $v_{s}^{q}[M](t^*)<0$ for all $s\in(a,\rho(t^*))$, which contradicts the strongly positive sign. 
	
	Since the function $v_{s}^{q}[M]$ is a solution of  $T_{n}[M]\,v_{s}^{q}[M](t)=0$ for $t\neq s$, we have that
	\[\dfrac{\partial^\eta}{\partial s^\eta}\left( T_{n}[M]\, v_{s}^{q}[M](t)\right)_{\mid s=a}=T_{n}[M]\,w_{M,q}(t)=0\,,\quad t\in(a,b]\,, \]
	that is, the function $w_{M,q}$ solves the equation $T_{n}[M]\,u(t)=0$, $t\in (a,b]$. 
	
	It remains to obtain the boundary conditions satisfied by the function $w_{M,q}$.
	
	Using the matrix of the Green's function given in \eqref{Ec:MG}, the equality \eqref{relacion}, the expression of $g_{n-j}$, $j=1,\dots,n-1$ given in \eqref{Ec::gj} and the following identity 
	\begin{equation*}  \dfrac{\partial}{\partial t}\Big(\dfrac{\partial^{q} }{\partial t^{q}}\, G_{M}(t,s)\Big)=A\,\dfrac{\partial^{q} }{\partial t^{q}} G_{M}(t,s),\quad \text{for all } t\in I\backslash \left\lbrace s\right\rbrace,\, q\in\{1,\dots,n-1\},
	\end{equation*}
	it follows that 
	\[\dfrac{\partial^{\mu_{1}^{q}+q}}{\partial t^{\mu_{1}^{q}+q}} g_M^2(t,s)_{\mid t=a}=0,\, 	-\dfrac{\partial^{\mu_{1}^{q}+q+1}}{\partial t^{\mu_{1}^{q}+q}\partial s} g_M^2(t,s)_{\mid t=a}=0,\,\dots,\, 	(-1)^\eta\dfrac{\partial^{\mu_{1}^{q}+q+\eta}}{\partial t^{\mu_{1}^{q}+q}\partial s^\eta}g_M^2(t,s)_{\mid t=a}=0\,.\]
	%	\[\left\lbrace \begin{array}{r}
		%	\dfrac{\partial^{\mu_{1}^{q}+q}}{\partial t^{\mu_{1}^{q}+q}} %g_M^2(t,s)_{\mid t=a}=0\,,\\\\
		%	-\dfrac{\partial^{\mu_{1}^{q}+q+1}}{\partial t^{\mu_{1}^{q}+q}\partial s} g_M^2(t,s)_{\mid t=a}=0\,,\\\\
		%	\vdots\hspace{0.2cm}\\\\
		%	(-1)^\eta\dfrac{\partial^{\mu_{1}^{q}+q+\eta}}{\partial t^{\mu_{1}^{q}+q}\partial s^\eta}g_M^2(t,s)_{\mid t=a}=0\,.\end{array}\right.\]
	Since $(\eta+\mu_{1}^{q}+q)\mod{n}<n-1$, then we are not in an element of the diagonal of $\dfrac{\partial^{q} }{\partial t^{q}} G_{M}(t,s)$, and taking in account that the above equalities are satisfied for $s=a$, by continuity we deduce that
	%\[\dfrac{\partial^{\mu_{1}^{q}+q}}{\partial t^{\mu_{1}^{q}+q}}g_M^1(t,a)_{\mid t=a}=0, 	-\dfrac{\partial^{\mu_{1}^{q}+q+1}}{\partial t^{\mu_{1}^{q}+q}\partial s}g_M^1(t,a)_{\mid t=a}=0,\dots, 	(-1)^\eta\dfrac{\partial^{\mu_{1}^{q}+q+\eta}}{\partial t^{\mu_{1}^{q}+q}\partial s^\eta}g_M^1(t,a)_{\mid t=a}=0\,.\]
	\[\left\lbrace \begin{array}{r}
		\dfrac{\partial^{\mu_{1}^{q}+q}}{\partial t^{\mu_{1}^{q}+q}}g_M^1(t,s)_{\mid (t,s)=(a,a)}=0\,,\\\\
		-\dfrac{\partial^{\mu_{1}^{q}+q+1}}{\partial t^{\mu_{1}^{q}+q}\partial s}g_M^1(t,s)_{\mid (t,s)=(a,a)}=0\,,\\\\
		\vdots\hspace{0.5cm}\\\\
		(-1)^\eta\dfrac{\partial^{\mu_{1}^{q}+q+\eta}}{\partial t^{\mu_{1}^{q}+q}\partial s^\eta}g_M^1(t,s)_{\mid (t,s)=(a,a)}=0\,.\end{array}\right.\]
	
	Therefore,
	\[w_{M,q}^{(\mu_{1}^{q})}(a)=\dfrac{\partial^{\mu_{1}^{q}+q+\eta}}{\partial t^{\mu_{1}^{q}+q}\partial s^\eta}g_M^1(t,s)_{\mid (t,s)=(a,a)}=0\,.\]
	
	Making a similar reasoning we deduce that
	\[w_{M,q}^{(\mu_{2}^{q})}(a)=\cdots=w_{M,q}^{(\mu_{z-1}^{q})}(a)=w_{M,q}^{(\mu_{z+1}^{q})}(a)=\cdots=w_{M,q}^{(\mu_{k}^{q})}(a)=0\,.\]
	
	For $\mu_{z}^{q}$, we have the following equalities
	\[\dfrac{\partial^{\mu_{z}^{q}+q}}{\partial t^{\mu_{z}^{q}+q}} g_M^2(t,s)_{\mid t=a}=0,\, -\dfrac{\partial^{\mu_{z}^{q}+q+1}}{\partial t^{\mu_{z}^{q}+q}\partial s}g_M^2(t,s)_{\mid t=a}=0,\,\dots,\, 	(-1)^\eta\dfrac{\partial^{\mu_{z}^{q}+q+\eta}}{\partial t^{\mu_{z}^{q}+q}\partial s^\eta}g_M^2(t,s)_{\mid t=a}=0\,.\]
	%\[\left\lbrace \begin{array}{r}
		%	\dfrac{\partial^{\mu_{z}^{q}+q}}{\partial t^{\mu_{z}^{q}+q}} g_M^2(t,s)_{\mid t=a}=0\,,\\\\
		%	-\dfrac{\partial^{\mu_{z}^{q}+q+1}}{\partial t^{\mu_{z}^{q}+q}\partial s}g_M^2(t,s)_{\mid t=a}=0\,,\\\\
		%	\vdots\hspace{0.5cm}\\\\
		%	(-1)^\eta\dfrac{\partial^{\mu_{z}^{q}+q+\eta}}{\partial t^{\mu_{z}^{q}+q}\partial s^\eta}g_M^2(t,s)_{\mid t=a}=0\,.\end{array}\right.\]
	
	By assumption, we have that $\eta+\mu_{z}^{q}+q=n-1$, so we are in an element of the diagonal of  $\dfrac{\partial^{q} }{\partial t^{q}} G_{M}(t,s)$, and taking into account again that the previous equalities are satisfied for $s = a$ and that $\dfrac{\partial^{q} }{\partial t^{q}} G_{M}(t,s)$ has a jump equal to $1$ on the diagonal  we infer that
	%\[	\dfrac{\partial^{\mu_{z}^{q}+q}}{\partial t^{\mu_{z}^{q}+q}}g_M^1(t,a)_{\mid t=a}=0, -\dfrac{\partial^{\mu_{z}^{q}+q+1}}{\partial t^{\mu_{z}^{q}+q}\partial s}g_M^1(t,a)_{\mid t=a}=0,\dots, 	(-1)^\eta\dfrac{\partial^{\mu_{z}^{q}+q+\eta}}{\partial t^{\mu_{z}^{q}+q}\partial s^\eta}g_M^1(t,a)_{\mid t=a}=1\,.\]				
	\[\left\lbrace \begin{array}{r}
		\dfrac{\partial^{\mu_{z}^{q}+q}}{\partial t^{\mu_{z}^{q}+q}}g_M^1(t,s)_{\mid (t,s)=(a,a)}=0\,,\\\\
		-\dfrac{\partial^{\mu_{z}^{q}+q+1}}{\partial t^{\mu_{z}^{q}+q}\partial s}g_M^1(t,s)_{\mid (t,s)=(a,a)}=0\,,\\\\
		\vdots\hspace{0.5cm}\\\\
		(-1)^\eta\dfrac{\partial^{\mu_{z}^{q}+q+\eta}}{\partial t^{\mu_{z}^{q}+q}\partial s^\eta}g_M^1(t,s)_{\mid (t,s)=(a,a)}=1\,.\end{array}\right.\]
	
	Therefore,
	\[w_{M,q}^{(\mu_{z}^{q})}(a)=\dfrac{\partial^{\mu_{z}^{q}+q+\eta}}{\partial t^{\mu_{z}^{q}+q}\partial s^\eta}g_M^1(t,s)_{\mid (t,s)=(a,a)}=(-1)^\eta\,.\]

	Let's see what conditions $w_{M,q}$ satisfies at $t=b$. To do this, proceeding with the previous argument we obtain that	
	\[		\dfrac{\partial^{\rho_{1}^{q}+q}}{\partial t^{\rho_{1}^{q}+q}}g_M^1(t,s)_{\mid t=b}=0,\, -\dfrac{\partial^{\rho_{1}^{q}+q+1}}{\partial t^{\rho_{1}^{q}+q}\partial s}g_M^1(t,s)_{\mid t=b}=0,\,\dots,\, 	(-1)^\eta\dfrac{\partial^{\rho_{1}^{q}+q+\eta}}{\partial t^{\rho_{1}^{q}+q}\partial s^\eta}g_M^1(t,s)_{\mid t=b}=0\,.\]				
	%\[\left\lbrace \begin{array}{r}
		%	\dfrac{\partial^{\rho_{1}^{q}+q}}{\partial t^{\rho_{1}^{q}+q}}g_M^1(t,s)_{\mid t=b}=0\,,\\\\
		%	-\dfrac{\partial^{\rho_{1}^{q}+q+1}}{\partial t^{\rho_{1}^{q}+q}\partial s}g_M^1(t,s)_{\mid t=b}=0\,,\\\\
		%	\vdots\hspace{0.5cm}\\\\
		%	(-1)^\eta\dfrac{\partial^{\rho_{1}^{q}+q+\eta}}{\partial t^{\rho_{1}^{q}+q}\partial s^\eta}g_M^1(t,s)_{\mid t=b}=0\,.\end{array}\right.\]
	
	The above equalities are satisfied at $s=a$ and since $b\neq a$ we have that
	\[w_{M,q}^{(\rho_{1}^{q})}(b)=\dfrac{\partial^{\rho_{1}^{q}+q+\eta}}{\partial t^{\rho_{1}^{q}+q}\partial s^\eta}g_M^1(t,s)_{\mid (t,s)=(b,a)}=0\,.\]
	
	Analogously we obtain that 
	\[w_{M,q}^{(\rho_{2}^{q})}(b)=\dots=w_{M,q}^{(\rho_{n-k}^{q})}(b)=0\,.\]

	If $\eta=0$, then it follows from the matrix argument we made earlier that the function $w_{M,q}(t)=\frac{\partial^{q} }{\partial t^{q}} g_{M}^{1}(t,a)$ satisfies the same boundary conditions \eqref{Ec::sigma1}--\eqref{Ec::epsilon1}.

	Then from Proposition \ref{P::1} and Theorem \ref{resultSigno} we conclude that $w_{M,q}> 0$ on $(a,b)$ for all $M\in [0,\lambda_2^{q}]$.
	
	\vspace{0.5cm}
	Now,  let us see that $v_{s}^{q}[M]$ cannot be positive for $M>\lambda_2^{q}$.

	Suppose that there exists $\widehat{M}>\lambda_2^{q}$ such that $v_{s}^{q}[\widehat{M}]$ is positive. Then, by monotony of $v_{s}^{q}[M]$ we have  for every $M\in [\lambda_2^{q},\widehat{M}]$ that  $w_{\widehat{M},q}\leq w_{M,q}\leq w_{\lambda_2^{q},q}$.
	
	In particular, $w_{\widehat{M},q}^{(\beta^{q})}(b)\geq w_{M,q}^{(\beta^{q})}(b)\geq w_{\lambda_2^{q},q}^{(\beta^{q})}(b)=0$ if $\beta^{q}$ is even and  $w_{\widehat{M},q}^{(\beta^{q})}(b)\leq w_{M,q}^{(\beta^{q})}(b)\leq w_{\lambda_2^{q},q}^{(\beta^{q})}(b)=0$ if $\beta^{q}$ is odd.
	
	If $w_{\widehat{M},q}^{(\beta^{q})}(b)\neq 0$, then there exists $\rho>0$ such that $w_{\widehat{M},q}(t)<0$ for all $t\in (b-\rho,b)$, which contradicts our assumption. Hence,
	\[0=w_{\widehat{M},q}^{(\beta^{q})}(b)=w_{M,q}^{(\beta^{q})}(b)= w_{\lambda_2^{q},q}^{(\beta^{q})}(b)\,,\quad \forall M\in [\lambda_2^{q},\widehat{M}], \]
	and this fact contradicts the discrete character of the spectrum.
	
	Thus, we conclude that, if $M\in [0,\lambda_2^{q}]$, then
	\[\forall t\in (a,b)\,,\quad \exists \rho^{q}(t)>0\ \mid \ v_{s}^{q}[M](t)>0\ \forall s\in(a,a+\rho^{q}(t))\,.\]
	
	Moreover, if $M>\lambda_2^{q}$, then $v_{s}^{q}[M]$ is not positive.
	\begin{itemize}
		\item[\textbf{Step 2.}] Study of the function $v_{s}^{q}[M]$ at $s=b$.
	\end{itemize}
	
	We study the following function by applying similar reasoning to Step 1. In this case we consider 
	\[y_{M,q}(t)=\dfrac{\partial ^\gamma}{\partial s^\gamma} \Big(\dfrac{\partial^{q}}{\partial t^{q}} g_M^2(t,s) \Big)_{\mid s=b}\,,\]
	where $\gamma$ has been defined in \eqref{Ec::gamma}.
	
	In this case, using arguments analogous to Step 1 we obtain that if $\gamma>0$, then
	\[\dfrac{\partial^{q}}{\partial t^{q}}  g_M^2(t,b)=\dfrac{\partial}{\partial s}g_M^2(t,s)_{\mid s=b}=\cdots =\dfrac{\partial^{\gamma-1}}{\partial s^{\gamma-1}} \Big(\dfrac{\partial^{q}}{\partial t^{q}} g_M^2(t,s) \Big)_{\mid s=b}=0\,.\]
	
	Moreover, if $\gamma$ is even, then $y_{M,q}\geq 0$ and if $\gamma$ is odd, then $y_{M,q}\leq 0$.
	
	Again, we have that 
	\[T_{n}[M]\,y_{M,q}(t)=0\,,\quad \forall\,t\in [a,b)\,.\]
	
	In the same way than above, studying $\dfrac{\partial^{q} }{\partial t^{q}} G_{M}(t,s)$ to determine the boundary conditions of $y_{M,q}$ we infer that
	\[\begin{split}
		y_{M,q}^{(\mu_{1}^{q})}(a)=\cdots=y_{M,q}^{(\mu_{k}^{q})}(a)=&\,0\,,\\
		y_{M,q}^{(\rho_{1}^{q})}(b)\hspace{0.3cm}=\cdots=y_{M,q}^{(\rho_{h-1}^{q})}(b)=y_{M,q}^{(\rho_{h+1}^{q})}(b)=\cdots=y_{M,q}^{(\rho_{n-k}^{q})}(b)=&\,0\,,\\
		\hspace{3cm}y_{M,q}^{(\rho_{h}^{q})}(b)=&\,(-1)^{\gamma+1}\,.
	\end{split}\]
	Analogously, if $\gamma=0$, then $y_{M,q}(t)=\frac{\partial^{q}}{\partial t^{q}} g_{M}^{2}(t,b)$ satisfies the conditions \eqref{Ec::sigma2}--\eqref{Ec::epsilon2} by an argument similar to the case $\gamma>0$.

	Then from Proposition \ref{P::2}  and Theorem \ref{resultSigno} we conclude that $y_{M,q}>0$ if $\gamma$ is even and $y_{M,q}<0$ if $\gamma$ is odd on $(a,b)$ for all $M\in [0,\lambda_4^{q}]$.

	As in the previous step it can be seen that for $M>\lambda_4^{q}$, the function $v_{s}^{q}[M]$ is never positive. Therefore, we conclude from this step that, if $M\in [0,\lambda_4^{q}]$, then
	\[\forall t\in (a,b)\,,\quad \exists \rho^{q}(t)>0\ \mid \ v_{s}^{q}[M](t)>0\ \forall s\in(b-\rho^{q}(t),b)\,.\]
	
	Moreover, if $M>\lambda_4^{q}$, then $v_{s}^{q}[M]$ is cannot be positive.
	\begin{itemize}
		\item[\textbf{Step 3.}] Study of the function $v_{s}^{q}[M]$ at $t=a$.
	\end{itemize}
	Let us denote 
	\[\widehat g_{(-1)^nM}(t,s)=\left\lbrace \begin{array}{cc}
		\widehat g_{(-1)^n M}^1(t,s)\,, &a\leq s\leq t\leq b\,,\\\\
		\widehat g_{(-1)^n M}^2(t,s)\,,&a<t<s<b\,,\end{array}\right. \]
	as the related Green's function of $\widehat T_n[(-1)^nM]$ in $X_{\ \,\{\sigma_1,\dots,\sigma_k\}}^{*\{\varepsilon_1,\dots,\varepsilon_{n-k}\}}=X_{\{\tau_1,\dots,\tau_{n-k}\}}^{\{\delta_1,\dots,\delta_{k}\}}$.

	In this case, we study the following function by applying similar reasoning to Step 1: 
	\[\widehat{w}_{M,q}(t)=(-1)^n\dfrac{\partial^{\alpha^{q}+q} }{\partial\,s^{\alpha^{q}+q}}  \widehat{g}_{(-1)^nM}^1(t,s)_{\mid s=a}\,,\]
	which is equivalent, using equality \eqref{Ec::gg1}, to study the function
	\begin{equation*}
		\widehat w_{M,q}(s)=\dfrac{\partial^{\alpha^{q}+q}}{\partial\,t^{\alpha^{q}+q}}  g_{M}^2(t,s)_{\mid t=a }=\dfrac{\partial^{\alpha^{q}}}{\partial\,t^{\alpha^{q}}} \Big( \dfrac{\partial^{q}}{\partial t^{q}} g_{M}^2(t,s) \Big)_{\mid t=a}\,,\quad s\in I\,. 
	\end{equation*}
	In this case,  if $\alpha^{q}>0$ using that $\{0,\dots,\alpha^{q}-1\}\subset \{\mu_{1}^{q},\dots,\mu_{k}^{q}\}$ we have  
	\[ \dfrac{\partial^{q}}{\partial t^{q}} g_{M}^{2}(a,s)=\dfrac{\partial}{\partial t}\Big( \dfrac{\partial^{q}}{\partial t^{q}} g_{M}^2(t,s) \Big)_{\mid t=a}=\cdots=\dfrac{\partial^{\alpha^{q}-1}}{\partial t^{\alpha^{q}-1}}\Big( \dfrac{\partial^{q}}{\partial t^{q}} g_{M}^2(t,s) \Big)_{\mid t=a}=0\,,\quad \forall s\in(a,b).\]
	
	As in the Steps 1 and 2 we can deduce that
	$$\widehat T_n[(-1)^nM]\,\widehat w_{M,q}(t)=0,\;\;t\in (a,b].$$
	
	Using the arguments of Step 1, we can affirm that if there exists $t^*\in(a,b)$ such that $\widehat w_{M,q}(t^*)<0$, then $v_{s}^{q}[M]$ is not positive. 
	
	By the first row of \eqref{relacion2} and \eqref{Ec::gjh} we deduce that:
	{\small \[ \dfrac{\partial^{\tau_1}}{\partial t^{\tau_1}}\widehat g_{(-1)^nM}^2(t,s)_{\mid t=a}=0,\,- \dfrac{\partial^{\tau_1+1}}{\partial t^{\tau_1}\partial s}\widehat g_{(-1)^nM}^2(t,s)_{\mid t=a}=0,\,\dots,(-1)^{\alpha^{q}+q} \dfrac{\partial^{\tau_1+\alpha^q+q}}{\partial t^{\tau_1}\partial s^{\alpha^{q}+q}}\widehat g_{(-1)^nM}^2(t,s)_{\mid t=a}=0\,.\]}
	%{	\[\left\lbrace \begin{array}{rl}
			%	\dfrac{\partial^{\tau_1}}{\partial t^{\tau_1}}\widehat g_{(-1)^nM}^2(t,s)_{\mid t=a}&=0\,,\\\\
			%	- \dfrac{\partial^{\tau_1+1}}{\partial t^{\tau_1}\partial s}\widehat g_{(-1)^nM}^2(t,s)_{\mid t=a}&=0\,,\\\\
			%	\vdots\hspace{0.5cm}\\\\
			
			%	(-1)^{\alpha^{q}+q} \dfrac{\partial^{\tau_1+\alpha^q+q}}{\partial t^{\tau_1}\partial s^{\alpha^{q}+q}}\widehat g_{(-1)^nM}^2(t,s)_{\mid t=a}&=0\,.\end{array}\right.\]}
	
	Since $(\tau_1+\alpha^q+q)\mod{n}<n-1$, we do not reach any diagonal element, hence the previous equalities are satisfied for $s=a$, and we obtain:
	%{\small \[\dfrac{\partial^{\tau_1}}{\partial t^{\tau_1}}\widehat g_{(-1)^nM}^1(t,a)_{\mid t=a}=0,- \dfrac{\partial^{\tau_1+1}}{\partial t^{\tau_1}\partial s}\widehat g_{(-1)^nM}^1(t,a)_{\mid t=a}
		%	=0,\dots,(-1)^{\alpha^{q}+q} \dfrac{\partial^{\tau_1+\alpha^{q}+q}}{\partial t^{\tau_1}\partial s^{\alpha^{q}+q}}\widehat g_{(-1)^nM}^1(t,a)_{\mid t=a} =0\,.\]}
	{	\[\left\lbrace \begin{array}{rl}
			\dfrac{\partial^{\tau_1}}{\partial t^{\tau_1}}\widehat g_{(-1)^nM}^1(t,s)_{\mid (t,s)=(a,a)}&=0\,,\\\\
			- \dfrac{\partial^{\tau_1+1}}{\partial t^{\tau_1}\partial s}\widehat g_{(-1)^nM}^1(t,s)_{\mid (t,s)=(a,a)}
			&=0\,,\\\\
			\vdots\hspace{0.5cm}\\\\
			(-1)^{\alpha^{q}+q} \dfrac{\partial^{\tau_1+\alpha^{q}+q}}{\partial t^{\tau_1}\partial s^{\alpha^{q}+q}}\widehat g_{(-1)^nM}^1(t,s)_{\mid (t,s)=(a,a)} &=0\,.\end{array}\right.\]}
	
	So, we conclude that $\widehat{w}_{M,q}^{(\tau_1)}(a)=0\,.$
	
	Analogously we obtain that 
	$$\widehat{w}_{M,q}^{(\tau_2)}(a)=\cdots=\widehat{w}_{M,q}^{(\tau_{h-1})}(a)=\widehat{w}_{M,q}^{(\tau_{h+1})}(a)=\cdots=\widehat{w}_{M,q}^{(\tau_{n-k})}(a)=0.$$
	
	For $\tau_{h}$, we have the following equalities 
	{\small \[\dfrac{\partial^{\tau_{h}}}{\partial t^{\tau_{h}}}\widehat g_{(-1)^nM}^2(t,s)_{\mid t=a}=0,\,-\dfrac{\partial^{\tau_{h}+1}}{\partial t^{\tau_{h}}\partial s}\widehat g_{(-1)^nM}^2(t,s)_{\mid t=a}=0,\,\dots,(-1)^{\alpha^{q}+q}\dfrac{\partial^{\tau_{h}+\alpha^{q}+q}}{\partial t^{\tau_{h}}\partial s^{\alpha^{q}+q}}\widehat g_{(-1)^nM}^2(t,s)_{\mid t=a} =0\,.\]}
	%{	\[\left\lbrace \begin{array}{rl}
			%		\dfrac{\partial^{\tau_{h}}}{\partial t^{\tau_{h}}}\widehat g_{(-1)^nM}^2(t,s)_{\mid t=a}&=0\,,\\\\
			%-\dfrac{\partial^{\tau_{h}+1}}{\partial t^{\tau_{h}}\partial s}\widehat g_{(-1)^nM}^2(t,s)_{\mid t=a}&=0\,,\\\\
			%\vdots\hspace{0.5cm}\\\\
			%(-1)^{\alpha^{q}+q} \dfrac{\partial^{\tau_{h}+\alpha^{q}+q}}{\partial t^{\tau_{h}}\partial s^{\alpha^{q}+q}}\widehat g_{(-1)^nM}^2(t,s)_{\mid t=a} &=0\,.\end{array}\right.\]}
	
	In this case, since $\tau_{h}+\alpha^q+q=n-1$, we reach a diagonal element of $\widehat G(t,s)$ given in \eqref{Ec:MGh}, and as consequence we obtain the following equalities for $s=a$:
	%{\small \[\dfrac{\partial^{\tau_{h}}}{\partial t^{\tau_{h}}}\widehat g_{(-1)^nM}^1(t,a)_{\mid t=a}=0,-\dfrac{\partial^{\tau_{h}+1}}{\partial t^{\tau_{h}}\partial s}\widehat g_{(-1)^nM}^1(t,a)_{\mid t=a}=0,\dots,	(-1)^{\alpha^{q}+q} \dfrac{\partial^{\tau_{h}+\alpha^q+q}}{\partial t^{\tau_{h}}\partial s^{\alpha^{q}+q}}\widehat g_{(-1)^nM}^1(t,a)_{\mid t=a} =1\,.\]}
	{	\[\left\lbrace \begin{array}{rl}
			\dfrac{\partial^{\tau_{h}}}{\partial t^{\tau_{h}}}\widehat g_{(-1)^nM}^1(t,s)_{\mid (t,s)=(a,a)}&=0\,,\\\\
			-\dfrac{\partial^{\tau_{h}+1}}{\partial t^{\tau_{h}}\partial s}\widehat g_{(-1)^nM}^1(t,s)_{\mid (t,s)=(a,a)} &=0\,,\\\\
			\vdots\hspace{0.5cm}\\\\
			(-1)^{\alpha^{q}+q} \dfrac{\partial^{\tau_{h}+\alpha^q+q}}{\partial t^{\tau_{h}}\partial s^{\alpha^{q}+q}}\widehat g_{(-1)^nM}^1(t,s)_{\mid (t,s)=(a,a)} &=1\,.\end{array}\right.\]}
	
	Therefore, $\widehat{w}_{M,q}^{(\tau_{h})}(a)=(-1)^{n-\alpha^{q}-q}\,.$
	
	Now, let us study the behavior of $\widehat w_{M,q}$ at $t=b$. Studying the $(n-k+1)^{\mathrm{th}}$ row of \eqref{relacion2}, we have for all $s\in(a,b)$:
	{\small \[\dfrac{\partial^{\delta_1}}{\partial t^{\delta_1}}\widehat g_{(-1)^nM}^1(t,s)_{\mid t=b}=0,- \dfrac{\partial^{\delta_1+1}}{\partial t^{\delta_1}\partial s}\widehat g_{(-1)^nM}^1(t,s)_{\mid t=b}=0,\dots,	(-1)^{\alpha^{q}+q} \dfrac{\partial^{\delta_1+\alpha^{q}+q}}{\partial t^{\delta_1}\partial s^{\alpha^{q}+q}}\widehat g_{(-1)^nM}^1(t,s)_{\mid t=b}=0\,.\]}						
	
	The above equalities are satisfied at $s=a$ and since $b\neq a$ we have that
	$\widehat{w}_{M,q}^{(\delta_1)}(b)=0$.
	
	Analogously we obtain that 
	\[\widehat{w}_{M,q}^{(\delta_2)}(b)=\cdots=\widehat{w}_{M,q}^{(\delta_{k})}(b)=0.\]
	
	Note that if $\alpha^{q}=0$, then from the same matrix argument above we deduce that the function $\widehat w_{M,q}(s)=\frac{\partial^{q} }{\partial t^{q}} g_{M}^{2}(a,s)$ verifies the same boundary conditions:
	%\[\begin{split}
		%\widehat{w}_{M,q}^{(\tau_2)}(a)=\cdots=\widehat{w}_{M,q}^{(\tau_{h-1})}(a)=\widehat{w}_{M,q}^{(\tau_{h+1})}(a)=\cdots=\widehat{w}_{M,q}^{(\tau_{n-k})}(a)=&0\,,\\
		%\widehat{w}_{M,q}^{(\delta_2)}(b)=\cdots=\widehat{w}_{M,q}^{(\delta_{k})}(b)=&0\,.
		%\end{split}\]
		\begin{eqnarray*}
			\widehat{w}_{M,q}^{(\tau_2)}(a)=\cdots=\widehat{w}_{M,q}^{(\tau_{h-1})}(a)=\widehat{w}_{M,q}^{(\tau_{h+1})}(a)=&\cdots&=\widehat{w}_{M,q}^{(\tau_{n-k})}(a)=0\,,\\
			\widehat{w}_{M,q}^{(\delta_2)}(b)=&\cdots&=\widehat{w}_{M,q}^{(\delta_{k})}(b)\hspace{0.3cm}=0\,.
		\end{eqnarray*}
		
		From \eqref{Ec::Tg} and \eqref{EC::Ad} we deduce that 
		$$\widehat T_{n}[0]\,v(t):=(-1)^{n}\,T_{n}^{*}[0]\,v(t)=T_{n}[0]\,v(t).$$

		Hence, using arguments similar to Step 1, we conclude that $\widehat w_{M,q}> 0$ on $(a,b)$ for all $M\in [0,(\lambda_4^*)^{q}]$, where $(\lambda_4^*)^{q}>0$ is the least positive eigenvalue of $\widehat T_{n}[0]\equiv T_{n}[0]$ in $X_{4}^*$ defined in \eqref{space4adjunto}.
		
		From Lemma \ref{adjuntos2}, we know that $(\lambda_4^*)^{q}=\lambda_4^{q}$. Thus, from this step we have that, if $M\in [0,\lambda_4^{q}]$, then 
		\[\forall s\in (a,b)\,,\quad \exists \rho^{q}(s)>0\ \mid \ v_{s}^{q}[M](t)>0\ \forall t\in(a,a+\rho^{q}(s))\,.\]
		\begin{itemize}
			\item[\textbf{Step 4.}] Study of the function $v_{s}^{q}[M]$ at $t=b$.
		\end{itemize}
		
		In this case, we study the following function by applying similar reasoning to Step 3 
		\[\widehat{y}_{M,q}(t)=(-1)^n\dfrac{\partial^{\beta^{q}+q} }{\partial\,s^{\beta^{q}+q}}  \widehat{g}_{(-1)^nM}^2(t,s)_{\mid s=b}\,,\]
		which is equivalent using equality \eqref{Ec::gg1} to studying the function 
		\begin{equation*}
			\widehat y_{M,q}(s)=\dfrac{\partial^{\beta^{q}+q} }{\partial\,t^{\beta^{q}+q}} g_{M}^1(t,s)_{\mid t=b}=\dfrac{\partial^{\beta^{q}}}{\partial\,t^{\beta^{q}}} \Big( \dfrac{\partial^{q}}{\partial t^{q}} g_{M}^2(t,s) \Big)_{\mid t=b}\,.\end{equation*}
		
		Since $\{0,\dots,\beta^{q}\}\subset \{\rho_{1}^{q},\dots,\rho_{n-k}\}$ if $\beta^{q}>0$ we obtain that
		\[\dfrac{\partial^{q}}{\partial t^{q}} g_M(b,t)=\dfrac{\partial}{\partial t} \Big(\dfrac{\partial^{q}}{\partial t^{q}}{g}_{M}(t,s)\Big)_{\mid t=b}=\cdots=\dfrac{\partial^{\beta^{q}-1}}{\partial t^{\beta^{q}-1}} \Big(\dfrac{\partial^{q}}{\partial t^{q}} {g}_{M}(t,s)\Big)_{\mid t=b}=0\,.\]
		
		Moreover, using arguments analogous to Step 1 we obtain that if  $\beta^{q}$ is even, then $\widehat y_{M,q}\geq 0$ and if  $\beta^{q}$ is odd, then $\widehat y_{M,q}\leq 0$.
		
		Again, we have that 
		\[\widehat T_{n}[(-1)^{n}M]\,\widehat y_{M,q}(t)=0\,,\quad \forall\,t\in [a,b)\,.\]
		
		A study analogous to Step 3 leads us to the fact that $\widehat y_{M,q}$ satisfies the boundary conditions 
		\begin{equation}\label{condcion}
			\begin{split}
				\widehat{y}_{M,q}^{(\tau_{1})}(a)=\cdots=\widehat{y}_{M,q}^{(\tau_{n-k})}(a)=&\,0\,,\\
				\widehat{y}_{M,q}^{(\delta_{1})}(b)=\cdots=\widehat{y}_{M,q}^{(\delta_{z-1})}(b)=\widehat{y}_{M,q}^{(\delta_{z+1})}(b)=\cdots=\widehat{y}_{M,q}^{(\delta_{k})}(b)=&\,0\,,\\
				y_{M,q}^{(\delta_{z})}(b)=&\,(-1)^{n-\beta^q-q+1}\,.
			\end{split}
		\end{equation}
		
		Analogously, if $\beta^{q}=0$, then $\widehat y_{M,q}(s)=\frac{\partial^{q}}{\partial t^{q}} g_{M}^{2}(b,s)$ satisfies the conditions \eqref{condcion} by an argument similar to the case $\beta^{q}>0$.
		
		Thus, we deduce that  $\widehat{y}_{M,q}> 0$ on $(a,b)$ if $\beta^{q}$ is even and $\widehat{y}_{M,q}< 0$ on $(a,b)$ if $\beta^{q}$ is odd for all $M\in [0,(\lambda_2^*)^{q}]$ , where $(\lambda_2^*)^{q}>0$ is the least positive eigenvalue of $\widehat T_{n}[0]\equiv T_{n}[0]$ in $X_{2}^*$ defined in \eqref{space2adjunto}.

		Again, from Lemma \ref{adjuntos1} we have that $(\lambda_2^*)^{q}=\lambda_2^{q}$. Therefore, we conclude from this step that, if  $M\in [0,\lambda_2^{q}]$, then
		\[\forall s\in (a,b)\,,\quad \exists \rho^{q}(s)>0\ \mid \ v_{s}^{q}[M](t)>0\ \forall t\in(b-\rho^q(t),b)\,.\]
		\begin{itemize}
			\item[\textbf{Step 5.}] Study of  the function $v_{s}^{q}[M]$ on $(a,b)\times(a,b)$.
		\end{itemize}
		In this step, we verify that the function $v_{s}^{q}[M](t)>0$ for all  $(t,s)\in I\times I$ if $M$ belongs to the given intervals. Let us denote $u_{M}^{s}(t)=g_{M}(t,s)$, for all $s\in (a,b)$. 
		
		The function $v_{s}^{q}[M](t)$ is in $C^{n-2-q}(I)$ and it satisfies the boundary conditions \eqref{condicion1}-\eqref{condicion2}.

		Using the definition of $v_{s}^{q}[M]$, the following equality holds for all  $s\in(a,b)$:
		
		\begin{equation}\label{Ec::us}
			T_{n-q}[0]\,v_{s}^{q}[M](t)=-M \,u_{M}^{s}(t)\,,\quad \forall\, t\neq s\,,\ t\in I\,.\end{equation}

		Let's see that for the values of the parameter $M$ for which $v_{s}^{q}[M](t)$ has constant sign in $I$, it cannot have a double zero in $(a,b)$, and this implies that the change of the sign must be at $t=a$ or $t=b$, and the result is proven. 
		
		We do the proof for $n-k$ even (the case $n-k$ odd is similarly proved). In this case, we study the behavior of $v_{s}^{q}[M](t)$ for $M>0$ and $u_{M}^{s}\geq0$. From \eqref{Ec::us}, we have that $$\dfrac{d^{n-q}}{dt^{n-q}}\,v_{s}^{q}[M](t)=\dfrac{\partial^{n}}{\partial t^{n}}\,g_{M}(t,s)\leq 0.$$
		Therefore, using the condition $(T_{d})$, since $v_1=\dots=v_n=1$, we have that  $$\dfrac{d^{n-q-1}}{dt^{n-q-1}}\,v_{s}^{q}[M](t)=\dfrac{\partial^{n-1}}{\partial t^{n-1}}\,g_{M}(t,s)$$ is a decreasing function, with two continuous components and a jump $+1$ at $t=s$. Then, it has at most two zeros on $I$ (see Figure \ref{Fig::1}).

		\begin{figure}[h]
			\centering			\includegraphics[width=10cm]{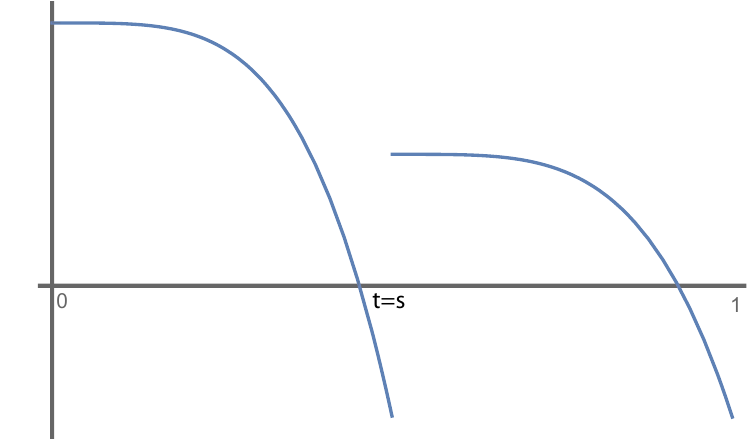}
			\caption{ $\dfrac{d^{n-q-1}}{dt^{n-q-1}}\,v_{s}^{q}[M](t)$, maximal oscillation with $t\in I=[0,1].$}\label{Fig::1}
		\end{figure}
		Now, we have that $$\dfrac{d^{n-q-2}}{dt^{n-q-2}}\,v_{s}^{q}[M](t)=\dfrac{\partial^{n-2}}{\partial t^{n-2}}\,g_{M}(t,s)$$ is a continuous function with at most four zeros on $I$ (See Figure  \ref{Fig::2}).
		
		\begin{figure}[h]
			\centering			\includegraphics[width=10cm]{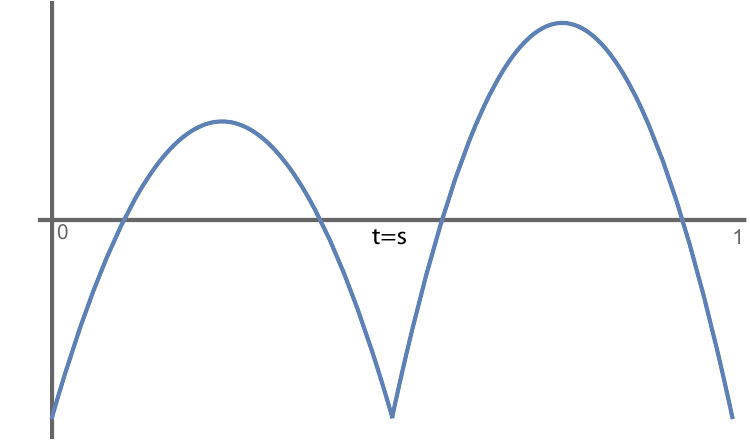}
			\caption{ $\dfrac{d^{n-q-1}}{dt^{n-q-1}}\,v_{s}^{q}[M](t)$, maximal oscillation with $t\in I=[0,1].$}\label{Fig::2}
		\end{figure}
		
		By recurrence, we have that $$\dfrac{d^{n-q-l}}{dt^{n-q-l}}\,v_{s}^{q}[M](t)=\dfrac{\partial^{n-l}}{\partial t^{n-l}}\,g_{M}(t,s)$$ has at most $l+2$ zeros on $I$. In particular, $v_{s}^{q}[M]$ has at most $n-q+2$ zeros on $I$. Moreover, since 
		\begin{eqnarray*}
			(v_{s}^{q})^{(\sigma_{1}^{q})}[M](a)=&\cdots&=(v_{s}^{q})^{(\sigma_{c_{q}}^{q})}[M](a)=0\,,\\
			(v_{s}^{q})^{(\varepsilon_{1}^{q})}[M](b)=&\cdots&=(v_{s}^{q})^{(\varepsilon_{d_{q}}^{q})}[M](b)\hspace{0.04cm}=0\,,
		\end{eqnarray*}
		$v_{s}^{q}[M]$ has $c_{q}+d_{q}=n-q$ zeros on the boundary. Thus, $v_{s}^{q}[M]$ has at most two zeros on $(a,b)$.
		
		Now, making an argument similar to the one made in \cite[Theorem 5.1]{CabSaab}, we deduce that the maximal oscillation (that is, the maximal number of zeros) is possible if and only if
		\begin{equation}\label{osilation}
			\left\lbrace \begin{array}{cc}
				(v_{s}^{q})^{(\alpha^{q})}[M](a)\le 0\,,& \text{if $n-q-c_q$ is even,}\\&\\
				(v_{s}^{q})^{(\alpha^{q})}[M](a)\geq 0\,,& \text{if $n-q-c_q$ is odd.}\end{array}\right. 	
		\end{equation}
		Since we are in the case $n-q-c_{q}$ even, the maximal oscillation is only possible if\linebreak $(v_{s}^{q})^{(\alpha^{q})}[M](a)\le 0$. However, by definition, $(v_{s}^{q})^{(\alpha^{q})}[M](a)\neq 0$ and, since $v_{s}^{q}[M]$ is nonnegative and $\alpha^{q}$ denotes the order of the smallest derivative at $t=a$ which is different from zero, necessarily $(v_{s}^{q})^{(\alpha^{q})}[M](a)>0$. Therefore the maximal oscillation can not happen (that is, $v_{s}^{q}[M]$ cannot have two zero on $(a,b)$). Thus, $v_{s}^{q}[M]$ has at most one simple zero on $(a,b)$, which is not possible due to the constant sign of $v_{s}^{q}[M]$, and we conclude that $v_{s}^{q}[M]$ does not have any zero on $(a,b)$.
		
		From this step we deduce that if $M>0$ and $u_{M}^{s}\geq 0$, then $v_{s}^{q}[M]>0$ on $(a,b)$.			
		
		In conclusion, from the previous steps the result is proved.
	\end{proof}
	In the sequel, we present some  examples of application of the previous result.
	\begin{example}
		Consider again the fourth order operator $T_{4}[M]$ coupled with the boundary conditions $X_{\{0,1,2\}}^{\{2\}}$ defined in Example \ref{ejemp}.
		
		The functions $\frac{\partial}{\partial t}{g}_{M}$ and $\frac{\partial^{2}}{\partial t^{2}}{g}_{M}$ satisfy the conditions of spaces  $X_{\{0,1,3\}}^{\{1\}}$ and $X_{\{0,2,3\}}^{\{0\}}$, respectively.
		
		In this case $k=3$ and $n-k=1$ is odd.	For $q=1$, we have that $c_q=2$, $d_q=1$, $z=2\neq 1$ and $\beta^{1}=0$. By applying Theorem \ref{T::IP} it follows that $\frac{\partial}{\partial t}{g}_{M}$ is strongly negative on  $[0,1]\times [0,1]$  if, and only if $M\in[\lambda_{2}^{1},\lambda_1)$, where
		\begin{itemize}
			\item[*]$\lambda_1>0$ is the least positive eigenvalue of $T_{4}[0]$ in $X_{\{0,1,2\}}^{\{2\}}.$
			\item[*] $\lambda_{2}^{1}<0$ is the biggest negative eigenvalue of $T_{4}[0]$ in $X_{\{0,3 \}}^{\{0,1\}}.$  
		\end{itemize}

		The eigenvalues of $T_{4}[0]$ in $X_{\{0,1,2\}}^{\{2\}}$ are given by $\lambda^{4}$, where $\lambda $ is a positive solution of 
		$$\tanh\Big(\lambda\sqrt{2}\Big)+\tan\Big(\frac{\lambda}{\sqrt{2}}\Big)=0.$$
		Let us denote  $m_{1}$ as the smallest positive solution of this equation. Then, $\lambda_{1}=m_{1}^{4}$ is the least positive eigenvalue of $T_{4}[0]$ in $X_{\{0,1,2\}}^{\{2\}}.$	
		
		In the same way, the eigenvalues of $T_{4}[0]$ in $X_{\{0,3 \}}^{\{0,1\}}$ are given by $-\lambda^{4}$, where $\lambda$ is a positive solution of $\sin(\lambda)=0$.
		
		In this case, $\lambda_{min}^{1}=\pi$ is the smallest positive solution of this equation. Then, $\lambda_{2}^{1}=-(\lambda_{min}^{1})^{4}=-\pi^4$ is the biggest negative eigenvalue of $T_{4}[0]$ in $X_{\{0,3 \}}^{\{0,1\}}.$
		
		By numerical approach, it can be seen that $\lambda_1=m_{1}^{4}\approx 3.34^{4}$. 
		
		For $q=2$, we have that $c_q=n-q-1=1$, $d_q=1$, $z=1$ and $\beta^{2}=1$. Again, by applying Theorem \ref{T::IP} it follows that $\frac{\partial^{2}}{\partial t^2}{g}_{M}$  is strongly negative on  $[0,1]\times [0,1]$ if, and only if  $M\in[\lambda_{2}^{2},\lambda_1)$, where
		\begin{itemize}
			\item[*]$\lambda_1>0$ is the least positive eigenvalue of $T_{4}[0]$ in $X_{\{0,1,2\}}^{\{2\}}.$
			\item[*] $\lambda_{2}^{2}<0$ is the biggest negative eigenvalue of $T_{4}[0]$ in $X_{\{2,3 \}}^{\{0,1\}}.$  
		\end{itemize}
		
		The eigenvalues of $T_{4}[0]$ in $X_{\{2,3 \}}^{\{0,1\}}$ are given by $-\lambda^{4}$, where $\lambda$ is a least positive solution of 
		$$2e^{2\lambda}+(e^{2\lambda}+1)\,\cos(\lambda)=0.$$
		
		By denoting $\lambda_{min}^{2}$ as the smallest positive solution of this equation, then $\lambda_{2}^{2}=-(\lambda_{min}^{2})^{4}$ is the biggest negative eigenvalue of $T_{4}[0]$ in $X_{\{2,3 \}}^{\{0,1\}}.$ 
		
		In this case, we have that  $\lambda_{2}^{2}=-(\lambda_{min}^{2})^{4}\approx -1.87^{4}$.
	\end{example}
	\begin{example}\label{ejemlono}
		We consider the fourth order operator $T_{4}[M]$ defined in the space of the boundary conditions $X_{\{0,1\}}^{\{1,3\}}$. 
		
		In this case $k=2$ and $n-k$ is even. For $q=1$ we have that $c_q=1$, $d_q=n-q-c_q=2$, $z=1$, $h=2$ and $\alpha^{1}=\beta^{1}=1$. Moreover, the function $\frac{\partial}{\partial t}{g}_{M}$ satisfies the conditions of space  $X_{\{0,3\}}^{\{0,2\}}$.
		
		By applying Theorem \ref{T::IP} we obtain that $\frac{\partial}{\partial t}{g}_{M}$ is strongly positive on  $[0,1]\times [0,1]$  if, and only if $M\in (\lambda_{1},\lambda^{1}]$, where
		\begin{itemize}
			\item[*]$\lambda_1<0$ is the least positive eigenvalue of $T_{4}[0]$ in $X_{\{0,1\}}^{\{1,3\}}.$
			\item [*] $\lambda^{1}>0$ is the minimum between:
			\begin{itemize}
				\item [·] $\lambda_2^{1}>0$, the least positive eigenvalue of $T_{4}[0]$ in $X_{\{3\}}^{\{0,1,2\}}.$
				\item [·] $\lambda_4^{1}>0$ is the least positive eigenvalue of $T_{4}[0]$ in $X_{\{0,1,3\}}^{\{0\}}.$
			\end{itemize} 
		\end{itemize}
		
		By numerical approximation, we obtain that $\lambda_1\approx -2.36^{4}$, $\lambda_2^{1}\approx 2.22^4$ and $\lambda_4^{1}\approx 4.44^4$. Therefore, $\frac{\partial}{\partial t}{g}_{M}$ is strongly positive on $[0,1]\times [0,1]$ if, and only if, $M\in(-2.36^4,2.22^4]$.
	\end{example}
	As a consequence of the Theorem \ref{T::IP} we arrive at the following result.
	\begin{corollary} \label{T::cv} Let $q\in \{1,\dots,n-1\}$. Suppose that ${\{\sigma_1,\dots,\sigma_k\}}-{\{\varepsilon_1,\dots,\varepsilon_{n-k}\}}$ satisfies condition $(N_a)$, $c_q+d_q=n-q$, $c_{q}\geq 1$ and $d_{q}\geq 1$. If either  $\sigma_k=q+c_q-1$ or $\varepsilon_{n-k}=q+d_q-1$, we have the following properties:
		\begin{itemize}
			\item If $n-q-c_q$ is even, then  there is not any $M\in \mathbb{R}$ such that $v_{s}^{q}[M]$ is strongly negative on $I\times I$. 
			\item If $n-q-c_q$ is odd, then there is not any $M\in \mathbb{R}$ such that $v_{s}^{q}[M]$ is strongly positive on $I\times I$.
		\end{itemize}
	\end{corollary}
	\begin{proof}
		If $\sigma_k=q+c_q-1$, then $\eta=n-1-\sigma_{k}=n-q-c_q$.
		
		We consider
		$$w_{M,q}(t)=\dfrac{\partial^\eta}{\partial s^\eta}\Big(\dfrac{\partial^{q}}{\partial t^{q}} g_M^1(t,s)\Big)_{\mid s=a},$$ 
		defined in Step 1 of the proof of Theorem \ref{T::IP}.
		
		We know that this function satisfies the following boundary conditions:	
		\[\begin{split}
			w_{M,q}^{(\mu_ {1}^{q})}(a)=\cdots=w_{M,q}^{(\mu_{z-1}^{q})}(a)=w_{M,q}^{(\mu_{z+1}^{q})}(a)=\cdots=w_{M,q}^{(\mu_{k}^{q})}(a)=&\,0\,,\\
			w_{M,q}^{(\rho_{1}^{q})}(b)=\cdots=w_{M,q}^{(\rho_{n-k}^{q})}(b)=&\,0\,,\\
			w_{M,q}^{(\mu_{z}^{q})}(a)=&\,(-1)^{\eta}=(-1)^{n-q-c_q}\,.
		\end{split}\]	
		
		Hence, if $n-q-c_q$ is even, then there exists $\rho>0$, such that $w_{M,q}(t)>0$ for all $t\in(a,a+\rho)$. So, $v_{s}^{q}[M]$ cannot be negative for any real $M$.
		
		Now, if $n-q-c_q$ is odd, then there exists $\rho>0$, such that $w_{M,q}(t)<0$ for all $t\in(a,a+\rho)$. Thus, $v_{s}^{q}[M]$ cannot be positive for any $M\in\mathbb{R}$.	
		
		Analogously, if $\varepsilon_{n-k}=q+d_q-1$, then $\gamma=n-\varepsilon_{n-k}-1=n-q-d_q=c_q$. In this case, we consider the function 
		\[y_{M,q}(t)=\dfrac{\partial ^\gamma}{\partial s^\gamma} \Big(\dfrac{\partial^{q}}{\partial t^{q}} g_M^2(t,s) \Big)_{\mid s=b}\,,\]
		defined in Step 2 of the proof of Theorem \ref{T::IP}. 
		
		We know that for all $M\in \mathbb{R}$, the function $y_{M,q}$ satisfies the following boundary conditions:
		
		\[\begin{split}
			y_{M,q}^{(\mu_{1}^{q})}(a)=\cdots=y_{M,q}^{(\mu_{k}^{q})}(a)=&\,0\,,\\
			y_{M,q}^{(\rho_{1}^{q})}(b)=\cdots=y_{M,q}^{(\rho_{h-1}^{q})}(b)=y_{M,q}^{(\rho_{h+1}^{q})}(b)=\cdots=y_{M,q}^{(\rho_{n-k}^{q})}(b)=&\,0\,,\\
			y_{M,q}^{(\rho_{h}^{q})}(b)=&\,(-1)^{c_q+1}\,.
		\end{split}\]

		Hence, if $n-q-c_q$ and $c_q$ are even, then there exists $\rho>0$, such that $y_{M,q}(t)>0$ for all $t\in(b-\rho,b)$. So, $v_{s}^{q}[M]$ cannot be negative for any real $M$.
		
		Moreover, if $n-q-c_q$ is even and $c_q$ odd, then there exists $\rho>0$, such that $y_{M,q}(t)<0$ for all $t\in(b-\rho,b)$. So, $v_{s}^{q}[M]$ cannot be positive for any real $M$.
		
		Now, if $n-q-c_q$ is odd and $c_q$ even, then there exists $\rho>0$, such that $y_{M,q}(t)<0$ for all $t\in(b-\rho,b)$. So, $v_{s}^{q}[M]$ cannot be positive for any real $M$.
		
		Finally, if $n-q-c_q$ and $c_q$ are odd, then there exists $\rho>0$, such that $y_{M,q}(t)>0$ for all $t\in(b-\rho,b)$.
		
		As consequence, $v_{s}^{q}[M]$ cannot be negative for any real $M$. 		
	\end{proof}
	Now, we give a necessary condition of the nonpositive (nonnegative) sign of $v_{s}^{q}[M]$. This characterization will be an interval in which $v_{s}^{q}[M]$ has a nonpositive (nonnegative) sign whose infimum or supremum is given by the first eigenvalue $\lambda_{1}$ of $T_{n}[M]$ in $X_{\{\sigma_1,\dots,\sigma_{k}\}}^{\{\varepsilon_1,\dots,\varepsilon_{n-k}\}}.$
	\begin{theorem}\label{T::necessary}
		Let $q\in \{1,\dots,n-1\}$. Suppose that ${\{\sigma_1,\dots,\sigma_k\}}-{\{\varepsilon_1,\dots,\varepsilon_{n-k}\}}$ satisfies condition $(N_a)$, $c_q+d_q=n-q$, $c_{q}\geq 1$ and $d_{q}\geq 1$. If  $\mu_{z}^{q}\neq z-1$ and $\rho_{h}^{q}\neq h-1$, then the following properties are fulfilled:
		\begin{itemize}
			\item Let $n-k$ be even:
			\begin{itemize}
				\item   If $n-q-c_q$ is even and  $v_{s}^{q}[M]$ is nonpositive on $I\times I$, then $M\in[\lambda_{*}^{q},\lambda_1)$, and  if $n-q-c_q$ is odd and  $v_{s}^{q}[M]$ is nonnegative on $I\times I$, then $M\in[\lambda_{*}^{q},\lambda_1)$, where 
				\begin{itemize}
					\item [*] $\lambda_{1}<0$ is the biggest negative eigenvalue of $T_{n}[0]$  in $X_{\{\sigma_1,\dots,\sigma_{k}\}}^{\{\varepsilon_1,\dots,\varepsilon_{n-k}\}}.$ 
					\item [*]$\lambda_{*}^{q}<0$ is the maximum between:
					\begin{itemize}
						\item [·] $\lambda_3^{q}<0$ is the biggest negative  eigenvalue of $T_{n}[0]$ in $X_{3}$ defined in \eqref{space3}.
						\item [·] $\lambda_5^{q}<0$, the biggest negative eigenvalue of $T_{n}[0]$ in  $X_{5}$ defined in \eqref{space33}.
					\end{itemize}
				\end{itemize}
			\end{itemize}
			\item Let $n-k$ be odd:
			\begin{itemize}
				\item  If $n-q-c_q$ is even and $v_{s}^{q}[M]$ is nonpositive on $I\times I$, then $M\in (\lambda_{1},\lambda_{*}^{q}]$, and  if $n-q-c_q$ is odd and $v_{s}^{q}[M]$ is nonnegative on $I\times I$, then $M\in (\lambda_{1},\lambda_{*}^{q}]$  where 
				\begin{itemize}
					\item [*] $\lambda_{1}>0$ is the least positive eigenvalue of $T_{n}[0]$  in $X_{\{\sigma_1,\dots,\sigma_{k}\}}^{\{\varepsilon_1,\dots,\varepsilon_{n-k}\}}.$ 
					\item [*]$\lambda_{*}^{q}>0$ is the minimum between:
					\begin{itemize}
						\item [·] $\lambda_3^{q}>0$, the biggest negative  eigenvalue of $T_{n}[0]$ in $X_{3}$.
						\item [·] $\lambda_5^{q}>0$, the biggest negative eigenvalue of $T_{n}[0]$ in  $X_{5}$.
					\end{itemize}
				\end{itemize}
			\end{itemize}
		\end{itemize}		
	\end{theorem}
	\begin{proof}
		Since $\mu_{z}^{q}\neq z-1$ and  $\rho_{h}^{q}\neq h-1$ we have from Propositions \ref{P::1} and \ref{P::2} that there are eigenvalues $\lambda_{3}^{q}$ and $\lambda_{5}^{q}$ with  corresponding related eigenfunctions of constant sign. We suppose that $n-k$ and $n-q-c_q$ are even. For the other cases the proof is done in an analogous way.
		
		Let us assume that there exists $M^*\notin [\lambda_{*}^{q},\lambda_1)$, such that $v_{s}^{q}[M^*]$ is nonpositive on $I\times I$. From Theorem \ref{T::IP}, we can affirm that $M^*<\lambda_{1}$. So, $M^*<\lambda_{*}^{q}$  and for all $M\in[M^*,\lambda_1)$ the function $v_{s}^{q}[M]$ is nonpositive on $I\times I$ because the set where its takes nonpositive values is an interval. By the monotone decreasing character of $v_{s}^{q}[M]$ we infer that
		$$v_{s}^{q}[\lambda_{*}^{q}](t) \le v_{s}^{q}[M](t) \le v_{s}^{q}[M^*](t) \le 0.$$
		
		So, in particular \[w_{\lambda_{*}^{q},q}(t)\le  w_{M,q}(t)\le w_{M^*,q}(t)\le 0 \,,\]
		and
		\[\left\lbrace \begin{array}{cc} y_{\lambda_{*}^{q},q}(t)\le y_{M,q}(t)\le y_{M^*,q}(t) \le 0\,,&\text{if } \gamma \text{ is even,}\\\\
			0\leq y_{M^*,q}(t)\leq y_{M,q}(t)\leq y_{\lambda_{*}^{q},q}(t)\,,& \text{if } \gamma \text{ is odd.}
		\end{array} \right. \]
		
		If $\lambda_{*}^{q}=\lambda_3^{q}$, then  $w_{\lambda_{*}^{q},q}^{(\alpha^{q})}(a)=0$. So, we conclude that, for all $M\in[M^*,\lambda_{*}^{q})$, $w_{M,q}^{(\alpha^{q})}(a)=0$, which contradicts the discrete character of the spectrum of $T_{n}[0]$ in $X_{3}$.
		
		If $\lambda_{*}^{q}=\lambda_5^{q}$, then  $y_{\lambda_{*}^{q},q}^{(\beta^{q})}(b)=0$. So, we conclude that, for all $M\in[M^*,\lambda_{*}^{q})$,  $y_{M,q}^{(\beta^{q})}(b)=0$, which contradicts the discrete character of the spectrum of $T_{n}[0]$ in $X_{5}$.
		
		So, we arrive to a contradiction, thus the result is proved.	
	\end{proof}
	Next, we present an example to illustrate the previous result.
	\begin{example}
		Consider the operator $T_{5}[M]$ defined in the space of the boundary conditions $X_{\{0,2,3\}}^{\{1,3\}}$.  
		
		For $q=1$ we have that $c_q=2$, $n-q-c_q=d_q=2$ is even, $z=h=2$, $\alpha^{1}=0$  and $\beta^{1}=1$. Moreover, the function $\frac{\partial}{\partial t}{g}_{M}$ satisfies the conditions of space  $X_{\{1,2,4\}}^{\{0,2\}}$ that fulfills the assumptions of Theorem \ref{T::necessary}. Therefore, if  $\frac{\partial}{\partial t}{g}_{M}$ is nonpositive on  $[0,1]\times [0,1]$, then $M\in[\lambda_{*}^{1},\lambda^{1})$, where
		\begin{itemize}
			\item[*] $\lambda_1<0$ is the biggest negative eigenvalue of $T_{5}[0]$ in $X_{\{0,2,3\}}^{\{1,3\}}.$
			\item [*] $\lambda_{*}^{1}<0$ is the maximum between:
			\begin{itemize}
				\item [·] $\lambda_3^{1}<0$, the biggest negative  eigenvalue of $T_{5}[0]$ in $X_{\{0,1,4\}}^{\{0,2\}}.$
				\item [·] $\lambda_5^{1}<0$ is the biggest negative eigenvalue of $T_{5}[0]$ in  $X_{\{1,2,4\}}^{\{0,1\}}.$
			\end{itemize} 
		\end{itemize}
		
		By numerical approximation, we obtain that $\lambda_1\approx -2.23^{5}$, $\lambda_3^{1}\approx -3.67^5$ and  $\lambda_5^{1}\approx -2.88^5$. Thus, if $\frac{\partial}{\partial}{g}_{M}$ is nonpositive on $[0,1]\times [0,1]$, then $M\in[-2.88^5,-2.23^5)$.
	\end{example}
	\subsection{Case (b): {\boldmath $c_q=n-q$}  and {\boldmath $d_q=0$} }
	In this subsection we deal with the constant sign of $v_{s}^{q}[M]$ for the case $c_q=n-q$ and $d_q=0$. We arrive to the next theorem.
	\begin{theorem}\label{T::IP1}
		Let $q\in \{1,\dots,n-1\}$. Suppose that ${\{\sigma_1,\dots,\sigma_k\}}-{\{\varepsilon_1,\dots,\varepsilon_{n-k}\}}$ satisfies condition $(N_a)$, $c_{q}+d_{q}=n-q$, $c_{q}=n-q$ and $d_{q}=0$. The following properties are fulfilled:  
		\begin{itemize}
			\item If $n-k$ is even, then $v_{s}^{q}[M]$ is nonnegative on $I\times I$ if, and only if, $M\in(\lambda_1,0]$, where $\lambda_1<0$ is the biggest negative eigenvalue of $T_{n}[0]$ in $X_{\{\sigma_1,\dots,\sigma_{k}\}}^{\{\varepsilon_1,\dots,\varepsilon_{n-k}\}}.$ 
			\item If $n-k$ is odd, then $v_{s}^{q}[M]$ is nonnegative on $I\times I$ if, and only if, $M\in[0,\lambda_1)$, where $\lambda_1>0$ is the least positive eigenvalue of $T_{n}[0]$ in $X_{\{\sigma_1,\dots,\sigma_{k}\}}^{\{\varepsilon_1,\dots,\varepsilon_{n-k}\}}.$ 
		\end{itemize}
	\end{theorem}
	\begin{proof}
		Let us assume that $n-k$ is even.  For the case $n-k$ odd the proof is done in an analogous way. We know from the Theorem \ref{resultSigno} that $v_{s}^{q}\in X_{\{0,\dots,c_q\}}$ satisfies an initial problem and it is  nonnegative on $[a,b]$. Moreover, $v_{s}^{q}(t)=\frac{\partial^q}{\partial t^q} g_{0}(t,s)>0$ if $s<t$ and $v_{s}^{q}(t)=\frac{\partial^q}{\partial t^q} g_{0}(t,s)=0$ if $s>t$. 
		
		Taking into account Lemma~\ref{lemma:monotonia_derivadas}, we obtain that 
		$v_{s}^{q}(t)=\frac{\partial^q}{\partial t^q} g_{0}(t,s)>0$ on $(a,b)\times (a,b)$ if and only if $M\in (\lambda_1,M_q],$ with $M_q>\lambda_{1}$. Moreover, $\frac{\partial^q}{\partial t^q} g_{M}$ is monotone decreasing with respect to $M\in (\lambda_{1},M_{q}]$.
		
		Let's see that $M_q=0$. Using the monotone decreasing character of $v_{s}^{q}[M]$ with respect to $M\in (\lambda_{1},M_{q}]$ and since $v_{s}^{q}(t)=\frac{\partial^q}{\partial t^q} g_{0}(t,s)=0$ if $s>t$ we have that $v_{s}^{q}[M_q](t)<v_{s}^{q}(t)=0$ if $M_q>0$. On the other hand, if $s<t$ then there exists $(t^{*},s^{*})\in (a,b)\times (a,b)$ with $s^{*}<t^{*}$ such that $0<v_{s^{*}}^{q}[M_q](t^{*})<v_{s^{*}}^{q}(t^{*})$ if $M_q>0$. Thus, $v_{s}^{q}[M]$ changes sign for $M_q>0$ and this completes the proof.
	\end{proof}
	\begin{example}
		Consider the operator $T_{4}[M]$ coupled with the boundary conditions in the space $X_{\{2,3\}}^{\{0,1\}}.$ 
		
		In this case $n-k$ is even. For $q=2$ we have that $c_q=n-q=2$ and $d_q=0$. Then, by Theorem \ref{T::IP1} we have that  $\frac{\partial^{2}}{\partial t^{2}}{g}_{M}$ is nonnegative on  $[0,1]\times [0,1]$ if, and only if,  $M\in(\lambda_{1},0]$, where $\lambda_1<0$ is the biggest negative eigenvalue of $T_{4}[0]$ in $X_{\{2,3\}}^{\{0,1\}}.$

		By numerical approximation, we obtain that $\lambda_1\approx -1.8751^{4}$. Thus, $\frac{\partial^{2}}{\partial t^{2}}{g}_{M}$ is nonnegative on  $[0,1]\times [0,1]$ if, and only if,  $M\in(-1.8751^{4},0]$.
		
		For $q=3$ we have that $c_q=n-q=1$ and $d_q=0$. Again, by Theorem \ref{T::IP1} we have that  $\frac{\partial^{3}}{\partial t^{3}}{g}_{M}$ is nonnegative on  $[0,1]\times [0,1]$ if, and only if,  $M\in(-1.8751^{4},0]$.
	\end{example}
	\subsection{Case (c): {\boldmath $c_q =0$}  and {\boldmath $d_q=n-q$} }
	In this section, arguing in a similar manner than in previous one, we can characterize the constant sign of $v_{s}^{q}[M]$ for the case $c_q=0$ and $d_q=n-q$.
	
	Using Theorem \ref{resultSigno}, Lemma~\ref{lemma:monotonia_derivadas}, and reasoning in an analogous way to the proof of Theorem \ref{T::IP1}, we get to the next result.
	\begin{theorem}\label{T::IP2}
		Let $q\in \{1,\dots,n-1\}$. Suppose that ${\{\sigma_1,\dots,\sigma_k\}}-{\{\varepsilon_1,\dots,\varepsilon_{n-k}\}}$ satisfies condition $(N_a)$, $c_{q}+d_{q}=n-q$, $c_{q}=0$ and $d_{q}=n-q$. Then, the following properties are fulfilled:
		\begin{itemize}
			\item Let $n-k$ be even:
			\begin{itemize}
				\item If $n-q$ is even, then $v_{s}^{q}[M]$ is nonnegative on $I\times I$ if, and only if, $M\in(\lambda_1,0]$, where $\lambda_1<0$ is the biggest negative eigenvalue of $T_{n}[0]$ in $X_{\{\sigma_1,\dots,\sigma_{k}\}}^{\{\varepsilon_1,\dots,\varepsilon_{n-k}\}}.$ 
				\item If $n-q$ is odd, then  $v_{s}^{q}[M]$ is nonpositive on $I\times I$ if, and only if, $M\in (\lambda_1,0]$, where $\lambda_1<0$ is the biggest negative eigenvalue of $T_{n}[0]$ in $X_{\{\sigma_1,\dots,\sigma_{k}\}}^{\{\varepsilon_1,\dots,\varepsilon_{n-k}\}}.$ 
			\end{itemize}
			\item Let $n-k$ be odd:
			\begin{itemize}
				\item If $n-q$ is even, then $v_{s}^{q}[M]$ is nonnegative on $I\times I$ if, and only if, $M\in[0,\lambda_1)$, where $\lambda_1>0$ is the least positive eigenvalue of $T_{n}[0]$ in $X_{\{\sigma_1,\dots,\sigma_{k}\}}^{\{\varepsilon_1,\dots,\varepsilon_{n-k}\}}.$ 
				\item If $n-q$ is odd, then  $v_{s}^{q}[M]$ is nonpositive on $I\times I$ if, and only if, $M\in [0,\lambda_1)$, where $\lambda_1>0$ is the least positive eigenvalue of $T_{n}[0]$ in $X_{\{\sigma_1,\dots,\sigma_{k}\}}^{\{\varepsilon_1,\dots,\varepsilon_{n-k}\}}.$ 
			\end{itemize}
		\end{itemize}
	\end{theorem}
	\begin{example}
		Consider the operator $T_{6}[M]$ coupled with the boundary conditions in the space $X_{\{0,2\}}^{\{1,3,4,5\}}.$

		In this case $n-k=4$ is even. For $q=3$ we have that $c_q=0$ and $d_q=n-q=3$ is odd. Then, by Theorem \ref{T::IP2} we have that  $\frac{\partial^{3}}{\partial t^{3}}{g}_{M}$ is nonpositive on  $[0,1]\times [0,1]$ if, and only if,  $M\in(\lambda_{1},0]$, where $\lambda_1<0$ is the biggest negative eigenvalue of $T_{6}[0]$ in $X_{\{0,2\}}^{\{1,3,4,5\}}.$

		By numerical approximation, we obtain that $\lambda_1\approx -1.953^{6}$. Thus, $\frac{\partial^{3}}{\partial t^{3}}{g}_{M}$ is nonpositive on  $[0,1]\times [0,1]$ if, and only if,  $M\in(-1.953^{6},0]$.
		
		For $q=4$ we have that $c_q=0$ and $d_q=n-q=2$ is even. Again, from Theorem \ref{T::IP2} we have that  $\frac{\partial^{4}}{\partial t^{4}}{g}_{M}$ is nonnegative on  $[0,1]\times [0,1]$ if, and only if,  $M\in(-1.953^{6},0]$.
		
		Finally, for $q=5$ we have that $c_q=0$ and $d_q=n-q=1$ is odd. Then,  $\frac{\partial^{5}}{\partial t^{5}}{g}_{M}$ is nonpositive on  $[0,1]\times [0,1]$ if, and only if,  $M\in(-1.953^{6},0]$.
	\end{example}
	\section{Application to Nonlinear Problems}\label{sec7.3}
	In this section, we will study the existence of a nontrivial positive solution of the following nonlinear problem 
	\begin{equation}\label{probnonlineal}
		\left\{\begin{aligned}
			T_n[M]u(t)&=f\left(t,u(t),\dots,u^{(n-1)}(t)\right),\ t\in I,\\
			u^{(\sigma_1)}(a)&= \cdots= u^{(\sigma_k)}(a)= 0, \\
			u^{(\varepsilon_1)}(b)&= \cdots= u^{(\varepsilon_{n-k})}(b)= 0,
		\end{aligned} \right.
	\end{equation}
	with $f:I\times [0,\infty)\times \R^{n}\rightarrow [0,\infty)$ a continuous function that satisfies some adequate conditions that we will detail later.
	
	The existence of positive solutions of the nonlinear problem \eqref{probnonlineal} will be deduced from the index theory applied to compact operators defined in suitable cones.
	
	Suppose that there are $r$ derivatives of  constant sign of the Green's function $g_{M}$ related to linear problem of \eqref{probnonlineal}, with $r\in \{1,\dots, n-1\}$. Let us denote by $q_1,\dots, q_r\in \{1,\dots,n-1\}$ the indices of such derivatives, with 
	\[ q_{1}< q_{2}< \dots< q_{r}.\]
	Moreover, let us denote by 
	$$S:=S_{1}\cup S_{2}^{I}\cup S_{2}^{F}\cup \{0\}$$
	where 
	$$S_{1}:=\{\{q_{1},\dots,q_{l}\}\;/\;\; c_{q_{l}}+d_{q_{l}}=n-q_{l}, c_{q_{l}}\geq 1\;\; \text{and}\;\; d_{q_{l}}\geq 1 \},$$
	$$S_{2}^{I}:=\{\{q_{l+1},\dots,q_{r}\}\;/\;\; c_{q_{l+1}}=n-q_{l+1}\;\; \text{and}\;\; d_{q_{l+1}}=0 \},$$
	$$S_{2}^{F}:=\{\{q_{l+1},\dots,q_{r}\}\;/\;\;  c_{q_{l+1}}=0\;\; \text{and}\;\; d_{q_{l+1}}=n-q_{l+1}\},$$
	and $\{0\}$ denoting the constant sign of the Green's function $g_{M}$. 
	
	%Let us define $\bar{q}=\max S$.
	It is clear that $S_{2}^{I}=\varnothing$ or/and $S_{2}^{F}=\varnothing$. We assume that $S_{2}^{I}=\varnothing$ and $S_{2}^{F}\neq\varnothing$ (for the other cases, i.e., $S_{2}^{F}= \varnothing$ and either $S_{2}^{I}=\varnothing $ or  $S_{2}^{I}\neq \varnothing$, the arguments are analogous).

	Along this section, we will assume that $n-k$ is even (for $n-k$ odd an analogous reasoning could be made). We know from Lemma \ref{lem:signo_cte_parcialq1_parcialq2} and Theorem \ref{T::IP} that for any $q\in S_{1}$, $\frac{\partial^{q} }{\partial t^{q}} g_{M}$ has constant sign for all $M\in (\lambda_{1},\lambda^{q_{l}}]$ (the constant sign interval of the largest derivative of $S$ which belongs to $S_1$), with $\lambda_{1}$ and $\lambda^{q_{l}}$ characterized in Theorem \ref{T::IP}. 
	
	Moreover, for all $M\in  (\lambda_{1},\lambda^{q_{l}}]$ we have that 
	$$g_{M}(t,s)\geq 0 \;\;\text{and}\;\; (-1)^{d_{q_{i}}}\,\frac{\partial^{q_i} }{\partial t^{q_i}} g_{M}(t,s)\geq 0\;\; \text{for all}\;\; i\in \{1,\dots,l\},$$
	where $d_{q_{i}}=n-q_i-c_{q_{i}}$ for $i=1,\dots, l$.
	
	Let us consider the following condition introduced in \cite[page 182]{LoNolineal} as follows:
	\begin{itemize}
		\item[$(P_{g_{1}}$)] Suppose that there are continuous functions $\phi$, $k_1$ and $k_2 $ such that $\phi(s)>0$ for all $s\in (a,b)$ and $0<k_1(t)<k_2(t)$ for all $t\in (a,b)$, satisfying:
		\[\phi(s)\,k_1(t)\leq G(t,s)\leq \phi(s)\, k_2(t)\,,\quad \text{for all } (t,s)\in I \times I,\]
		where $G$ is a suitable integral kernel of certain integral operator.
	\end{itemize}
	
	Using the characterization of \cite[Theorem 8.1]{CabSaab} and Theorem \ref{T::IP}, with a similar argument to the one made in \cite[Theorem 5.2]{LoNolineal}, the following result is proved.
	
\begin{lemma}\label{lem:Pg1}
Suppose that $n-k$ is even. Then, for any  $M\in (\lambda_{1},\lambda^{q_{l}})$, the functions $g_{M}$, $(-1)^{d_{q_{i}}}\,\frac{\partial^{q_{i}}}{\partial t^{q_{i}}} g_{M}$, $i\in \{1,\dots,l\}$, satisfy the condition $(P_{g_{1}})$, that is, 
there exist continuous functions $\phi$, $k_{1}$,  $k_{2}$, $0<k_1(t)<k_2(t)$ for all $t\in (a,b)$, satisfying
\[\phi(s)\,k_1(t)\leq g_{M}(t,s)\leq \phi(s)\, k_2(t),\quad \text{for all } (t,s)\in I \times I,\]
and $k_{1}^{q_i}$, $k_{2}^{q_i}$, with $0< k_{1}^{q_i}(t) < k_{2}^{q_i}(t)$ for all $t\in (a,b)$  and  $i\in \{1,\dots,l\}$ such that
\[\phi(s)\,k_1^{q_{i}}(t)\leq (-1)^{d_{q_{i}}}\,\frac{\partial^{q_{i}}}{\partial t^{q_{i}}} g_{M}(t,s)(t,s)\leq \phi(s)\, k_2^{q_{i}}(t),\quad \text{for all } (t,s)\in I \times I, \quad i=1,\dots,l,\]
where $\phi(s)=(s-a)^{\eta}\,(b-s)^{\gamma}$, with $\eta$ and $\gamma$ defined in \eqref{Ec::eta} and \eqref{Ec::gamma}.
\end{lemma} 

\begin{proof}
We make the proof only for the function $g_{M}$. The same argument holds true for the functions $(-1)^{d_{q_{i}}}\,\frac{\partial^{q_{i}}}{\partial t^{q_{i}}} g_{M}$, $i=1,\dots,l$, by taking the same function $\phi$ and using Steps 1 and 2 of Theorem \ref{T::IP}.
	
 For all $M\in (\lambda_{1},\lambda^{q_{l}})$, we know from Steps 1 and 2 of \cite[Theorem 8.1]{CabSaab} that 
	\[\frac{\partial^{\eta}}{\partial t^{\eta}} g_{M}(t,s)_{\mid s=a}>0\quad\text{and}\quad (-1)^{\gamma}\,\frac{\partial^{\gamma}}{\partial t^{\gamma}} g_{M}(t,s)_{\mid s=b}>0,\;\; \text{for all}\;\; t\in (a,b).\]

%{\color{red} NO ENTIENDO LOS CAMBIOS (CREO QUE SON PRINCIPALMENTE DE NOTACIÓN) QUE PROPONE ALBERTO AQUI. REVISAR ESTA PARTE.}

%{\bf CREO QUE LA NOTACION NO ES CORRECTA. DONDE EST\'AN DEFINIDOS $w_0$ y $y_0$? }

	Let us define the following function \[v_{M}^{t}(s)=\dfrac{g_{M}(t,s)}{(s-a)^{\eta}\,(b-s)^{\gamma}}.\]
	It is clear that  $v_{M}^{t}(s)>0$ on $(a,b)\times (a,b)$ for all $M\in (\lambda_{1},\lambda^{q_{l}})$. Moreover, for each $t\in (a,b)$ we have that 
	\begin{equation*}
		\begin{aligned}
			h_{1}(t)&=\lim\limits_{s\rightarrow a^+}\dfrac{g_{M}(t,s)}{(s-a)^{\eta}\,(b-s)^{\gamma}} = \dfrac{\frac{\partial^{\eta}}{\partial t^{\eta}} g_{M}(t,s)_{\mid s=a}}{\eta!\, (b-a)^{\gamma}}\in (0,\infty)\,,\\
			h_{2}(t)&=\lim\limits_{s\rightarrow b^-}\dfrac{g_{M}(t,s)}{(s-a)^{\eta}\,(b-s)^{\gamma}}= \dfrac{(-1)^{\gamma}\,\frac{\partial^{\gamma}}{\partial t^{\gamma}} g_{M}(t,s)_{\mid s=b}}{\gamma!\, (b-a)^{\eta}}\in (0,\infty). 
		\end{aligned}
	\end{equation*}

	For each $t\in(a,b)$, let us define $\tilde{v}_{M}^{t}$ as the continuous extension of $v_{M}^{t}$ to the interval $I$, that is 
	\begin{eqnarray*}
		\tilde{v}_{M}^{t}(s)=\left\{
		\begin{array}{ccc}
			h_{1}(t)\,,& s=a,\\
			v_{M}^{t}(s)\,,& s\in (a,b),\\
			h_{2}(t)\,,& s=b.
		\end{array}
		\right.
	\end{eqnarray*}
	Then,  $\tilde{v}_{M}^{t}(s)>0$ on $[a,b]$ for all $t\in (a,b)$. Therefore, the following functions							 		
	\begin{equation*}
		\begin{aligned}
			k_{1}(t)&=\min_{s\in I} \tilde{v}_{M}^{t}(s)\,,\quad t\in I,\\
			k_{2}(t) &=\max_{s\in I}\tilde{v}_{M}^{t}(s)\,,\quad t\in I, 
		\end{aligned}
	\end{equation*}
	are continuous on $I$ and positive in $(a,b)$.
	
	Taking $\phi(s)=(s-a)^{\eta}\,(b-s)^{\gamma}>0$  on $(a,b)$, the function $g_{M}$ satisfies the condition $(P_{g_{1}})$.
\end{proof}

	Let us consider $I_{1}=[a_{1},b_{1}]$ and $I_{1}^{q_{i}}=[a_{1}^{q_{i}},b_{1}^{q_{i}}]$, $i=1,\dots, l$, subintervals of $I$ such that ${|k_{1}(t)|>0}$ for all $t\in I_{1}$ and $|k_{1}^{q_{i}}(t)|>0$ for all $t\in I_{1}^{q_{i}}$. Notice that, due to their continuity on $I$, such functions have constant sign on their corresponding intervals. Furthermore, let us define 
	\begin{equation*}
		\begin{aligned}
			k_{1}&=\max_{t\in I} |k_{1}(t)|,\;\; m_{1}\hspace{0.1cm}=\min_{t\in I_{1}} |k_{1}(t)|,\;\; k_{2}=\max_{t\in I} |k_{2}(t)|,\\
			k_{1}^{q_{i}}&=\max_{t\in I} |k_{1}^{q_{i}}(t)|,\;\; k_{2}^{q_{i}}=\max_{t\in I} |k_{2}^{q_{i}}(t)|,
		\end{aligned}
	\end{equation*}
	for $i=1,\dots,l$.
	
	We assume that the nonlinear part of equation satisfies the following conditions:
	\begin{itemize}
		\item[$(H_{1})$] 	$\displaystyle\liminf\limits_{\max\{|x_1|,\ldots,|x_{n}|\}\to 0} \min_{t\in I} \frac{f\left(t,x_1,\ldots,x_{n}\right)}{|x_1|+\cdots+|x_{n}|}=+\infty.$
		\item[$(H_{2})$] $\displaystyle	\limsup\limits_{\min\{|x_1|,\ldots,|x_{n}|\}\to \infty} \max_{t\in I} \frac{f\left(t,x_1,\ldots,x_{n}\right)}{|x_1|+\cdots+|x_{n}|}=0.$
	\end{itemize}
	
	Let $X\equiv (C^{n}(I),||\cdot||)$ be the real Banach space endowed with the norm 
	\[||u||=\max\{||u||_{\infty},\ldots,||u^{(n)}||_{\infty}\},\;\; \text{for all}\;\; u\in X,\]
	where $||u||_{\infty}=\max_{t\in I}\{ |u(t)|\}$ and consider operator 
	$\mathcal{L}_{M}:X\to X$, defined as
	\begin{equation}\label{operador}
		\mathcal{L}_{M}\, u(t):=\int_{a}^{b} g_{M}(t,s)\,f\left(s,u(s),\dots,u^{(n-1)}(s)\right)\,ds, \quad t \in I.
	\end{equation}
We will apply the fixed point index theory to the operator $\mathcal{L}_{M}$ to guarantee the existence of a fixed point of such operator.

	Let us define the following cone:
	{\small$$K=\left\{u\in X:\;\;
	\begin{aligned} 
		&u(t)\geq 0, (-1)^{d_{q_{i}}}\,u^{(q_{i})}(t)\geq 0,\,i\in\{1,\dots,l\},\,(-1)^{n-q_{j}}\,u^{(q_{j})}(t)\geq 0,\,j\in \{l+1,\dots,r\},\, t\in I,\\
		&u(t)\geq \frac{k_{1}(t)}{k_{2}} \|u\|_{\infty},\, (-1)^{d_{q_{i}}}\,u^{(q_{i})}(t)\geq \frac{k_{1}^{q_{i}}(t)}{k_{2}^{q_{i}}} \|u^{(q_{i})}\|_{\infty},\,i\in\{1,\dots,l\},\, t\in I
	\end{aligned}
	\right\}.$$}
	\begin{remark}
		Note that if $S_{2}^{F}=\varnothing$ and $S_{2}^{I}\neq\varnothing$ we take the following cone defined as:
			{\small$$K^{*}=\left\{u\in X:\;\;
			\begin{aligned} 
				&u(t)\geq 0, (-1)^{d_{q_{i}}}\,u^{(q_{i})}(t)\geq 0,\,i\in\{1,\dots,l\},\, u^{(q_{j})}(t)\geq 0,\,j\in \{l+1,\dots,r\},\, t\in I,\\
				&u(t)\geq \frac{k_{1}(t)}{k_{2}} \|u\|_{\infty},\, (-1)^{d_{q_{i}}}\,u^{(q_{i})}(t)\geq \frac{k_{1}^{q_{i}}(t)}{k_{2}^{q_{i}}} \|u^{(q_{i})}\|_{\infty},\,i\in\{1,\dots,l\},\, t\in I
			\end{aligned}
			\right\}.$$}
		On the other hand, if $S_{2}^{I}=\varnothing$ and $S_{2}^{F}=\varnothing$ we have that $l=r$ and we work with the cone defined as:
	$$\bar{K}=\left\{u\in X:\;\;
			\begin{aligned} 
				&u(t)\geq 0, (-1)^{d_{q_{i}}}\,u^{(q_{i})}(t)\geq 0,\,i\in\{1,\dots,l\},\, t\in I,\\
				&u(t)\geq \frac{k_{1}(t)}{k_{2}} \|u\|_{\infty},\, (-1)^{d_{q_{i}}}\,u^{(q_{i})}(t)\geq \frac{k_{1}^{q_{i}}(t)}{k_{2}^{q_{i}}} \|u^{(q_{i})}\|_{\infty},\,i\in\{1,\dots,l\},\, t\in I
			\end{aligned}
			\right\}.$$
	\end{remark}
	\begin{remark}
		Note that we can take the following family of cones $K_{q_{m}}\subset K$ defined as:
		$$K_{q_{m}}=\left\{u\in X:\;\;
		\begin{aligned} 
			&u(t)\geq 0,\dots, (-1)^{d_{q_{m}}}\,u^{(q_{m})}(t)\geq 0,\,t\in I,\\
			&u(t)\geq \frac{k_{1}(t)}{k_{2}} \|u\|_{\infty},\dots, (-1)^{d_{q_{m}}}\,u^{(q_{m})}(t)\geq \frac{k_{1}^{q_{m}}(t)}{k_{2}^{q_{m}}} \|u^{(q_{m})}\|_{\infty},\, t\in I
		\end{aligned}
		\right\},$$ if  $0\leq m \le l$ and 
		{\small$$K_{q_{m}}=\left\{u\in X:\;\;
			\begin{aligned} 
				&u(t)\geq 0, (-1)^{d_{q_{i}}}\,u^{(q_{i})}(t)\geq 0,\,i\in\{1,\dots,l\},\,(-1)^{n-q_{j}}\,u^{(q_{j})}(t)\geq 0,\,j\in \{l+1,\dots,m\},\, t\in I,\\
				&u(t)\geq \frac{k_{1}(t)}{k_{2}} \|u\|_{\infty},\, (-1)^{d_{q_{i}}}\,u^{(q_{i})}(t)\geq \frac{k_{1}^{q_{i}}(t)}{k_{2}^{q_{i}}} \|u^{(q_{i})}\|_{\infty},\,i\in\{1,\dots,l\},\, t\in I
			\end{aligned}
			\right\}.$$} if  $l< m \le r$.
		
		% In this case, 
		%$$g_{M}(t,s)\geq 0,\, (-1)^{d_{q_{1}}}\,\frac{\partial^{q_{1}}}{\partial t^{q_{1}}} g_{M}(t,s)\geq 0,\,\dots, (-1)^{d_{q_{m}}}\,\frac{\partial^{q_{m}} }{\partial t^{q_{m}}} g_{M}(t,s)\geq 0,\;\; \text{for all } M\in (\lambda_{1},\lambda^{q_{m}}],$$
		%where $(\lambda_{1},\lambda^{q_{m}}]$ denotes the interval of constant sign of the derivative $\frac{\partial^{q_{m}} }{\partial t^{q_{m}}} g_{M}(t,s)$ characterized in Theorem \ref{T::IP}. 
	\end{remark}
	\begin{remark}
		If $n-k$ is odd, then we obtain by a similar argument to Lemma~\ref{lem:Pg1} that  functions $g_{M}$ and $(-1)^{d_{q_{i}}}\,\frac{\partial^{q_{i}}}{\partial t^{q_{i}}} g_{M}$, $i\in \{1,\dots,l\}$, satisfy the following condition (with obvious notation):
		\begin{itemize}
			\item[$(N_{g_{1}}$)] Suppose that there are continuous functions $\phi, k_1$ and $k_2 $ such that $\phi(s)>0$ for all $s\in (a,b)$ and $k_1(t)<k_2(t)<0$ for all $t\in (a,b)$, satisfying:
			\[\phi(s)\,k_1(t)\leq G(t,s)\leq \phi(s)\, k_2(t)\,,\quad \text{for all } (t,s)\in I \times I \,.\]
		\end{itemize}
		In this case, we would look for a nontrivial positive solution of the following nonlinear problem 
		\begin{equation*}
			\left\{\begin{aligned}
				&T_n[M]u(t)+f\left(t,u(t),\dots,u^{(n-1)}(t)\right)=0,\ t\in I,\\
				&u^{(\sigma_1)}(a)= \cdots= u^{(\sigma_k)}(a)= 0, \\
				&u^{(\varepsilon_1)}(b)= \cdots= u^{(\varepsilon_{n-k})}(b)= 0.
			\end{aligned} \right.
		\end{equation*}
		
	\end{remark}

	In order to prove an existence result of problem \ref{probnonlineal}, we introduce the following definition and preliminary results.
	\begin{definition}
		Let $X$ be a Banach space, $\Omega\subset X$ open and $T: \Omega \rightarrow X$ a continuous map. We say that $T$ is compactly fixed if the set of fixed points of $T$ is compact.
	\end{definition}
	
	We will denote by $\operatorname{Fix}(T)$ the set of fixed points of $T$. 
	
	Next lemma compiles some classical results regarding the fixed point index formulated in \cite[Theorems 6.2, 7.3 and 7.11]{granas} in a more general framework.
	
	In particular, given  $X$ a Banach space, $K\subset X$ a cone and $\Omega \subset K$ an arbitrary open subset, $\partial\,\Omega$ will denote the boundary of $\Omega$ in the relative topology in $K$, induced by the topology of $X$.  
	\begin{lemma} \label{l-Amann-fixed-point}
		Let $X$ be a Banach space, $K\subset X$ a cone and $\Omega \subset K$ an arbitrary open subset with $0\in\Omega$. Assume that $T\colon \overline \Omega \rightarrow K$ is a compact and compactly fixed operator such that $x\neq T x$ for all $x\in\partial\,\Omega$. 
		
		Then the fixed point index $i_K(T,\Omega)$ has the following properties:
		\begin{enumerate}
			\item If $x\neq \mu\, T x$ for all $x\in\partial\, \Omega$ and for every  $\mu\le 1$, then $i_K(T,\Omega)=1.$
			\item If $\Omega$ is bounded and there exists $e\in K\setminus\{0\}$ such that $x\neq T x+\lambda\, e$ for all $x\in\partial\, \Omega$ and all $\lambda>0$, then $i_K(T,\Omega)=0$.
			\item If $i_K(T,\Omega)\neq 0$, then $T$ has a fixed point in $\Omega$.
			\item If $\Omega_1$ and $\Omega_2$ are two open and disjoint sets such that $\operatorname{Fix}(T) \subset \Omega_1\cup \Omega_2 \subset \Omega$, then
			\[i_K(T,\Omega)=i_K(T,\Omega_1) + i_K(T,\Omega_2). \]
			%		\item Let $\Omega^1$ be an open set with $\overline \Omega^1\subset \Omega$. If $i_K(T,\Omega)=1$ and $i_K(T,\Omega^1)=0$, then $T$ has a fixed point in $\Omega \setminus \overline \Omega^1$. The same result holds if $i_K(T,\Omega)=0$ and $i_K(T,\Omega^1)=1$.
		\end{enumerate}
	\end{lemma}
	
	\begin{remark}
		Note that, in Item $2$ in previous lemma, it is required that $\Omega$ is bounded. However, the other assertions hold for an arbitrary open set, which might be unbounded.	
	\end{remark}
	
	Using Items $1$ and $2$ in Lemma~\ref{l-Amann-fixed-point}, it is possible to deduce the following corollary. The proof would be analogous to that of \cite[Theorem 2.3.3]{Guo_general}.
	\begin{corollary}\label{cor-GuoLaks-prelim-adaptado}
		Let $X$ be a Banach space, $K\subset X$ a cone and $\Omega \subset K$ an open set such that $0\in\Omega$. Assume that $T\colon \overline \Omega \rightarrow K$ is a compact and compactly fixed operator such that $x\neq T x$ for all $x\in\partial\,\Omega$. Then
		\begin{enumerate}
			\item If $T \,x \not\succeq x$ for all $x\in \partial\,\Omega$ then $i_K(T,\Omega)=1$.
			\item If $\Omega$ is bounded and, moreover, $T \,x \not\preceq x$ for all $x\in \partial\,\Omega$, then $i_K(T,\Omega)=0$.
		\end{enumerate}
	\end{corollary}

Next, we prove the following theorem to ensure the existence of positive solutions.

	\begin{theorem}\label{theoremmain}
		Suppose that the conditions $(H_{1})$ and $(H_{2})$ hold. Then, for all $M\in (\lambda_{1},\bar{\lambda}]$ where $$\bar{\lambda}:=\left\{
		\begin{aligned}
		0,\;\; \text{if}\;\; l<r,\\
		\lambda^{q_{l}}\;\; \text{if}\;\; l=r,
		\end{aligned}
		\right.$$ the nonlinear problem \eqref{probnonlineal} has at least a nontrivial solution $u\in K$.
	\end{theorem}
	\begin{proof}
		Consider the operator $\mathcal{L}_{M}$ defined in \eqref{operador}. Since $g_{M}\geq 0$ for all $M\in (\lambda_{1},\bar{\lambda}]$, $f\geq0$ and the fixed points of the operator $\mathcal{L}_{M}$ coincide with the solutions of problem \eqref{probnonlineal}, we deduce that these solutions are nonnegative. 		
		
		We show that $\mathcal{L}_{M}$ is well-defined in $K$, that is, $\mathcal{L}_{M}(K)\subset K$.

		Let $u\in K$, it is immediate to verify that  $\mathcal{L}_{M}\,u \ge 0$ on $I$. Moreover, we have that:
			\begin{equation*}
			(-1)^{d_{q_{i}}}\,(\mathcal{L}_{M}\,u)^{(q_{i})}(t)=\displaystyle \int_{a}^{b} (-1)^{d_{q_{i}}}\,\frac{\partial^{q_{i}}}{\partial t^{q_{i}}}g_{M}(t,s)\,\, f\left(s,u(s),\dots,u^{(n-1)}(s)\right)\,ds\geq 0,
		\end{equation*}
		\begin{equation*}
			\begin{aligned}
				\mathcal{L}_{M}\,u(t):&=\displaystyle \int_{a}^{b} g_{M}(t,s)\,\,f\left(s,u(s),\dots,u^{(n-1)}(s)\right)\,ds\\ 
				&\geq \displaystyle \int_{a}^{b} k_{1}(t)\,\phi(s)\,\,f\left(s,u(s),\dots,u^{(n-1)}(s)\right)\,ds\\
				&=\frac{k_{1}(t)}{k_{2}} \displaystyle \int_{a}^{b} k_{2}\,\phi(s)\, \,f\left(s,u(s),\dots,u^{(n-1)}(s)\right)\,ds\\
				&\geq \frac{k_{1}(t)}{k_{2}} \displaystyle \int_{a}^{b} \left\{\max_{t\in I} g_{M}(t,s)\right\} \,f\left(s,u(s),\dots,u^{(n-1)}(s)\right)\,ds\\
				&\geq \frac{k_{1}(t)}{k_{2}} \max_{t\in I}\left\{\displaystyle \int_{a}^{b}  g_{M}(t,s)\, f\left(s,u(s),\dots,u^{(n-1)}(s)\right)\,ds\right\}=\frac{k_{1}(t)}{k_{2}}\,\|\mathcal{L}_{M}\,u\|_{\infty},
			\end{aligned}
		\end{equation*}
		and 
		\begin{equation*}
			\begin{aligned}
				(-1)^{d_{q_{i}}}\,(\mathcal{L}_{M}\,u)^{(q_{i})}(t)&=\displaystyle \int_{a}^{b} (-1)^{d_{q_{i}}}\,\frac{\partial^{q_{i}}}{\partial t^{q_{i}}}g_{M}(t,s)\,\, f\left(s,u(s),\dots,u^{(n-1)}(s)\right)\,ds\\
				&\geq \displaystyle \int_{a}^{b} k_{1}^{q_{i}}(t)\,\phi(s)\,\,f\left(s,u(s),\dots,u^{(n-1)}(s)\right)\,ds\\
				&=\frac{k_{1}^{q_{i}}(t)}{k_{2}^{q_{i}}} \displaystyle \int_{a}^{b} k_{2}^{q_{i}}\,\phi(s)\,\,f\left(s,u(s),\dots,u^{(n-1)}(s)\right)\,ds\\
				&\geq \frac{k_{1}^{q_{i}}(t)}{k_{2}^{q_{i}}} \displaystyle \int_{a}^{b} \left\{\sup_{t\in I}\, (-1)^{d_{q_{i}}}\,\frac{\partial^{q_{i}}}{\partial t^{q_{i}}}g_{M}(t,s)\right\}\,f\left(s,u(s),\dots,u^{(n-1)}(s)\right)\,ds\\
				&\geq \frac{k_{1}^{q_{i}}(t)}{k_{2}^{q_{i}}} \sup_{t\in I}\left\{\displaystyle \int_{a}^{b}  (-1)^{d_{q_{i}}}\,\frac{\partial^{q_{i}}}{\partial t^{q_{i}}}g_{M}(t,s)\,f\left(s,u(s),\dots,u^{(n-1)}(s)\right)\,ds\right\}\\
				&=\frac{k_{1}^{q_{i}}(t)}{k_{2}^{q_{i}}}\,\|(-1)^{d_{q_{i}}}\,(\mathcal{L}_{M}\,u)^{(q_{i})}\|_{\infty}=\frac{k_{1}^{q_{i}}(t)}{k_{2}^{q_{i}}}\,\|(\mathcal{L}_{M}\,u)^{(q_{i})}\|_{\infty},
			\end{aligned}
		\end{equation*}
		for all $t\in I$ and $i\in \{1,\dots,l\}$.
		
		On the other hand, for $j\in \{l+1,\dots, r\}$ and $t \in I$, we have that 
			\begin{equation*}
				(-1)^{n-q_{j}}\,(\mathcal{L}_{M}\,u)^{(q_{j})}(t)=\displaystyle \int_{a}^{b} (-1)^{n-q_{j}}\,\frac{\partial^{q_{j}}}{\partial t^{q_{j}}}g_{M}(t,s)\,\, f\left(s,u(s),\dots,u^{(n-1)}(s)\right)\,ds\geq 0.
		\end{equation*}
		Thus, $\mathcal{L}_{M}(K)\subset K$.
		
		By the continuity of $f$ and using standard techniques one can prove that operator $\mathcal{L}_{M}$ is compact.
		
		Consider $u\in K\cap \partial\Omega_1$. Let us choose

		$$\epsilon_1>\frac{k_{2}}{m_{1}^{2} \displaystyle \int_{a_{1}}^{b_{1}}  \phi(s)\, ds}.$$

		From assumption $(H_{1})$, there exists $p>0$ such that when $\|u\|<p$ we have that 
		\begin{equation*}
			f\left(t,u(t),\dots,u^{(n-1)}(t)\right)\ge 
			\epsilon_1\,\left(|u(t)|+\cdots+|u^{(n-1)}(t)|\right),\;\;\text{for all}\;\; t\in I.
		\end{equation*}
		
		Now, we show that  $\mathcal{L}_{M}\, u\not\preceq u$, for all $u\in K\cap \partial\Omega_1$ with $$\Omega_1=\{u\in K:\ \|u\|<p\},$$ for some $p>0$. Here $\preceq$ denotes the order induced by the cone $K$.
		
		Then, we have that for $t\in I_1$, 
		\begin{equation*}
			\begin{aligned}
				\mathcal{L}_{M}\,u(t)&=\displaystyle \int_{a}^{b} g_{M}(t,s)\,\,f\left(s,u(s),\dots,u^{(n-1)}(s)\right)\,ds\\
				&\geq \displaystyle \int_{a}^{b} k_{1}(t)\,\phi(s)\,\,f\left(s,u(s),\dots,u^{(n-1)}(s)\right)\,ds\\
				&\geq \displaystyle \int_{a_{1}}^{b_{1}} k_{1}(t)\,\phi(s)\,\,f\left(s,u(s),\dots,u^{(n-1)}(s)\right)\,ds\\
				&\geq m_{1}\,\displaystyle \int_{a_{1}}^{b_{1}} \phi(s)\,\,f\left(s,u(s),\dots,u^{(n-1)}(s)\right)\,ds\\
				&\geq m_{1}\,\epsilon_{1}\,\displaystyle \int_{a_{1}}^{b_{1}} \phi(s)\,\left(|u(s)|+\cdots+|u^{(n-1)}(s)|\right)\,ds\\
				& \geq \frac{m_{1}\,\epsilon_{1}}{k_2}\,\displaystyle \int_{a_{1}}^{b_{1}} \phi(s)\,k_{1}(s)\,\|u\|_\infty\,ds\\
				&\geq \frac{m_{1}^2\,\epsilon_{1}}{k_2}\,\|u\|_\infty \displaystyle \int_{a_{1}}^{b_{1}} \phi(s)\,ds>\|u\|_\infty.
			\end{aligned}
		\end{equation*}

		Therefore, $\mathcal{L}_{M}\,u(t)>u(t)$ for all $t\in I_{1}$ and so it is proved that $\mathcal{L}_{M}\,u\not\preceq u$ for all $u\in K\cap \partial\Omega_1$. As a consequence (see  \cite[Theorem 2.3.3]{Guo_general}), we have that
		\[i_K(\mathcal{L}_M, \, \Omega_1)=0. \]
		
		\bigskip
		On the other hand, the regularity of the Green's function $g_{M}$ %($n-k\geq q_{r}$) 
		allows us 
		%{\color{red} CREO QUE EN ESTA MODIFICACIÓN NO SE TIENE QUE n-k\geq q_{r}}
		guarantee that there exist $N_0, N_{q_{i}}\in \R$, $i\in \{1,\dots,r\}$ with $N_0>0$ and $N_{q_{i}}>0$ such that 
		\[\max_{t\in I} \displaystyle \int_{a}^{b} g_{M}(t,s)\,ds\leq N_0 \;\; \text{and}\;\; \sup_{t\in I} \displaystyle \int_{a}^{b} (-1)^{d_{q_{i}}}\,\frac{\partial^{q_{i}}}{\partial t^{q_{i}}}g_{M}(t,s)\,ds\leq N_{q_{i}},\;\;i\in\{1,\dots,r\}. \] 
		
		Let us choose 
		$$\epsilon_2<\min\left\{ \frac{1}{n\,N_0},\frac{1}{n\,N_{q_{i}}}:\;\;i\in \{1,\dots,r\}  \right\}.$$
		
		By hypothesis $(H_2)$, there exists $\tilde{M}>0$ such that if ${\min\left\{|u(t)|,\ldots,|u^{(n-1)}(t)|\right\}\geq \tilde{M}}$ we have that 
		\begin{equation*}
			f\left(t,u(t),\ldots,u^{(n-1)}(t)\right)
			\le \epsilon_{2}\,\left(|u(t)|+\cdots+|u^{(n-1)}(t)|\right)
			\le n\,\epsilon_{2}\,\|u\|,\;\;\text{for all}\;\;t\in I.
		\end{equation*}

		Consider $q>\{p,\tilde{M}\}$ and let us define the following subset of $K$
		$$\Omega_2=\bigcup_{i=0}^{r} \Big\{u\in K:\,\min_{t\in I} |u^{(q_i)}(t)|<q\Big\}.$$
		
		Since the set $\Omega_{2}$ is unbounded in the cone $K$, as we have pointed out earlier,  the fixed point index of operator $\mathcal{L}_{M}$ with respect to $\Omega_2$ is only defined in the case that the set of fixed points of operator $\mathcal{L}_{M}$ in $\Omega_2$, that is, $(id-\mathcal{L}_{M})^{-1}(\{0\})\cap \Omega_2$, is compact.

		Arguing as in the proof of \cite[Theorem 3.2]{Min}, we can assume that $i_K(\mathcal{L}_M,\Omega_2)$ can be defined in this case. Indeed, since $id-\mathcal{L}_{M}$ is a continuous operator, it is obvious that $(id-\mathcal{L}_{M})^{-1}(\{0\})\cap\Omega_{2}$ is closed. Moreover, if such set is unbounded, we would have infinite fixed points of operator $\mathcal{L}_{M}$ on $\Omega_{2}$ and, as a direct consequence, Problem \eqref{probnonlineal} has an infinite number of positive solutions on $\Omega_{2}$.
		Thus, we may assume that such set is bounded, and, from this hypothesis, is not difficult to verify that $(id-\mathcal{L}_{M})^{-1}(\{0\})\cap\Omega_{2}$ is equicontinuous.
		
		Now, we prove that $\|\mathcal{L}_{M}\,u\|\leq \|u\|$ for all $u\in K \cap\partial \Omega_2$.
		
		Consider $u\in K\cap \partial\Omega_2$. Therefore, 
		$$\min\left\{\min_{t\in I} |u^{(q_i)}(t)|:\,i\in \{0,\dots,r\}\right\}=q> \tilde{M}.$$
		
		Then, 
		\begin{equation*}
			\begin{aligned}
				\|\mathcal{L}_{M}\,u\|_{\infty}&=\max_{t\in I}\,\displaystyle \int_{a}^{b} g_{M}(t,s)\,f\left(s,u(s),\dots,u^{(n-1)}(s)\right)\,ds\leq n\,\epsilon_{2}\,\|u\|\max_{t\in I}\,\displaystyle \int_{a}^{b} g_{M}(t,s)\,ds\\
				&\leq n\,\epsilon_{2}\,\|u\|\,N_0< \|u\|,
			\end{aligned}
		\end{equation*} 
		and 
		\begin{equation*}
			\begin{aligned}
				\|(\mathcal{L}_{M}\,u)^{(q_{i})}\|_{\infty}&=\sup_{t\in I}\,\Big|\displaystyle \int_{a}^{b} \frac{\partial^{q_{i}}}{\partial t^{q_{i}}} g_{M}(t,s)\,f\left(s,u(s),\dots,u^{(n-1)}(s)\right)\,ds \Big|\\
				&\leq n\,\epsilon_{2}\,\|u\|\sup_{t\in I}\,\displaystyle \int_{a}^{b} (-1)^{d_{q_{i}}}\,\frac{\partial^{q_{i}}}{\partial t^{q_{i}}} g_{M}(t,s)\, f\left(s,u(s),\dots,u^{(n-1)}(s)\right)\,ds
				\\&\leq n\,\epsilon_{2}\,\|u\|\,N_{q_{i}}< \|u\|,
			\end{aligned}
		\end{equation*} 
		for all $i\in {1,\dots,r}$.
		
Thus, $\|\mathcal{L}_{M}\,u\|< \|u\|$, for all $u\in K\cap \partial\Omega_2$, and as a consequence (see \cite[Corollary 7.4]{granas}) we deduce that
		\[i_K(\mathcal{L}_M, \, \Omega_2)=1. \]
		
		Therefore, we conclude that the operator $\mathcal{L}_{M}$ has a fixed point in ${K\cap (\bar{\Omega_2}\setminus \Omega_1)}$ which is a solution of problem \eqref{probnonlineal}.
		
		Thus, in both situations ($(id-\mathcal{L}_{M})^{-1}(\{0\})\cap\Omega_{2}$ bounded or unbounded), we may ensure that there is a nontrivial solution of Problem \eqref{probnonlineal} on the cone $K$.
	\end{proof}
	In the sequel, we present an example to illustrate our result.
	\begin{example}
		Let us consider again the space studied in Example \ref{ejemlono}: $X_{\{0,1\}}^{\{1,3\}}$. In this case, we have from \cite[Theorem 8.1]{CabSaab} that 
		$g_{M}\geq 0$ on $[0,1]\times [0,1]$ if, and only if, $M\in(\lambda_{1},\lambda_{2}]$, where 
		\begin{itemize}
			\item[*]$\lambda_1<0$ is the least positive eigenvalue of $T_{4}[0]$ in $X_{\{0,1\}}^{\{1,3\}}.$
			\item [*] $\lambda_{2}>0$ is the minimum between:
			\begin{itemize}
				\item [·] $\lambda_2'>0$, the least positive eigenvalue of $T_{4}[0]$ in $X_{\{0\}}^{\{0,1,3\}}.$
				\item [·] $\lambda_2''>0$ is the least positive eigenvalue of $T_{4}[0]$ in $X_{\{0,1,2\}}^{\{1\}}.$
			\end{itemize} 
		\end{itemize}
		
		Since $X_{\ \, \{0,1,2\}}^{*\{1\}}=X_{\{0\}}^{\{0,1,3\}}$, we have that $X_{\{0\}}^{\{0,1,3\}}$ and $X_{\{0,1,2\}}^{\{1\}}$ have the same eigenvalues. So,  $\lambda_{2}'=\lambda_{2}''$. The eigenvalues of $T_{4}[0]$ in $X_{\{0\}}^{\{0,1,3\}}$ are given by $\lambda^4$, where $\lambda$ is a positive solution of the equation $\sin\left(\frac{m}{\sqrt{2}}\right)=0.$ 
		
		Then, $m_1=\pi\sqrt{2}$ is the smallest positive solution of this equation and so $\lambda_{2}'=\lambda_{2}''=4\pi^4 \approx 389.636 $ is the least positive solution of  of $T_{4}[0]$ in $X_{\{0\}}^{\{0,1,3\}}$.
		
		By numerical approach, we have that $\lambda_{1}\approx -2.36^4\approx -31.2852$. Thus, $g_{M}\geq 0$ on $[0,1]\times [0,1]$ if, and only if, $M\in(-2.36^4,4\pi^4]$.
		
		Therefore, the function $g_{M}(t,s)$ satisfies the property $(P_{g_{1}})$ for all $M\in(-2.36^4,4\pi^4]$.
		
		\vspace*{2pt}
		We know from Example \ref{ejemlono} that $\frac{\partial}{\partial t} g_{M}(t,s)\geq 0$ on $[0,1]\times [0,1]$ if, and only if, ${M\in(-2.36^4,2.22^4]}$. So, the function $\frac{\partial}{\partial t} g_{M}(t,s)$ satisfies the property $(P_{g_{1}})$ for all $M\in(-2.36^4,2.22^4]$.
		
		In particular, $g_{0}(t,s)$ and $\frac{\partial}{\partial t} g_{0}(t,s)$ satisfy the property $(P_{g_{1}})$.
		
		\vspace*{2pt}
		The function $\frac{\partial^2}{\partial t^2} g_{M}(t,s)$ changes sign because $\frac{\partial}{\partial t} g_{M}(t,s)$ vanishes at points $t=0$ and $t=1$.
		
		\vspace*{2pt}
		On the other hand, we have from Theorem \ref{T::IP2} that $\frac{\partial^3}{\partial t^3} g_{M}(t,s)$ is nonpositive on ${[0,1]\times [0,1]}$ if, and only if, $M\in(\lambda_{1},0]$. Thus, $\frac{\partial^3}{\partial t^3} g_{M}(t,s)\le 0$ on $[0,1]\times [0,1]$ if, and only if, $M\in(-2.36^4,0]$.
		
		\vspace*{4pt}
		Let us study for the particular case of $M=0$ the following nonlinear problem
		{\small\begin{equation}\label{ProblemaNonlineal1}
			\left\{
			\begin{aligned}
				u^{(4)}(t)&=(t^4+1)\left(e^{-\sqrt{(u(t))^2+(u'(t))^{2}+(u''(t))^{2}+(u'''(t))^{2}}}+\frac{1}{\ln(e+|u(t)|+|u'(t)|+|u''(t)|+|u'''(t)|)}\right),\;\; t\in [0,1],\\
				u(0)&=u'(0)=0,\\
				u'(1)&=u'''(1)=0.
			\end{aligned}
			\right.
		\end{equation}}
		In this case, the function $f(t,x_{1},x_{2},x_{3},x_4)=(t^4+1)\,\left(e^{-\sqrt{x_{1}^{2}+x_{2}^2+x_{3}^2+x_4^2}}+\frac{1}{\ln(e+|x_{1}|+|x_{2}|+|x_{3}|+|x_4|)}\right)$ is continuous and satisfies conditions $(H_1)$ and $(H_{2})$.
		
		From the adjoint boundary conditions, we have that $\eta=2$ and $\gamma=0$ and so $\phi(s)=s^2$. 
		
		Using the Mathematica Software package (see \cite{Maquez}) to calculate the expressions of the Green's functions, we have that $g_{0}(t,s)$ is given by 
		\begin{equation*}
			g_{0}(t,s)=\left\{
			\begin{aligned}
				&-\frac{1}{12}\,s^2\left(2s+3\,(t-2)\,t\right),\;\; 0\leq s\leq t\leq 1,\\
				&\frac{1}{12}\,t^2\left(-3\,(s-2)\,s-2t\right),\;\; 0\leq t< s\leq 1.
			\end{aligned}
			\right.
		\end{equation*} 
		Moreover, we have that \[\tilde{v}(t,s)=\left\{
		\begin{aligned}
			&-\frac{1}{12}\,\left(2s+3\,(t-2)\,t\right),\;\; 0\leq s\leq t<1,\\
			&\frac{1}{12}\,\frac{t^2}{s^2}\left(-3\,(s-2)\,s-2t\right),\;\; 0< t< s\leq 1,\,t\neq 1,
		\end{aligned}
		\right.\]
		where $\tilde{v}(t,s)$ is the continuous extension of $\frac{g_{0}(t,s)}{\phi(s)}$ to $(0,1)\times [0,1]$. 
		
		Since 
		\[\frac{\partial}{\partial s}\tilde{v}(t,s)=\left\{
		\begin{aligned}
			&-\frac{1}{6},\;\; 0\leq s\leq t\leq 1,\\
			&\frac{1}{6}\,\frac{t^2}{s^3}\left(2t-3s\right),\;\; 0< t< s\leq 1,
		\end{aligned}
		\right.\] 
		is negative for all $(t,s)\in (0,1)\times [0,1]$ we have that $\tilde{v}(t,s)$ is decreasing as a function of $s$ for all $t\in (0,1)$. So,
		\begin{equation*}
			\begin{aligned}
				k_{1}(t)&=\min_{s\in I} \tilde{v}(t,s)=\tilde{v}(t,1)=\frac{t^2}{12}\, \left(3-2t\right),\\
				k_{2}(t) &=\max_{s\in I}\tilde{v}(t,s)=\tilde{v}(t,0)=\frac{1}{4}\,t\left(2-t\right). 
			\end{aligned}
		\end{equation*}
		Taking into account that these two functions are increasing on $[0,1]$, we have that
		$$k_{1}=k_{1}(1)=\frac{1}{6}\;\; \text{and}\;\; k_{2}=k_{2}(1)=\frac{1}{4}.$$
		
		Let us choose $I_{1}=[\frac{1}{6},1]$. Then, we have that
		$$m_{1}=\min_{t\in I_{1}}\,k_{1}(t)=k_{1}\left(\frac{1}{6}\right)=\frac{1}{162}.$$
		
		The first derivative of $g_{0}(t,s)$ is given by the expression 
		\begin{equation*}
			\frac{\partial }{\partial t}\,g_{0}(t,s)=\left\{
			\begin{aligned}
				&\frac{1}{2}\,s^2\left(1-t\right),\;\; 0\leq s\leq t\leq 1,\\
				&\frac{1}{2}\,t\left(s\,(2-s)-t\right),\;\; 0\leq t< s\leq 1.
			\end{aligned}
			\right.
		\end{equation*} 
		Analogously, it is not difficult to verify that \[\tilde{v}^{1}(t,s)=\left\{
		\begin{aligned}
			&\frac{1}{2}\,\left(1-t\right),\;\; 0\leq s\leq t<1,\\
			&\frac{1}{2}\,\frac{t}{s^2}\,\left(s\,(2-s)-t\right),\;\; 0< t< s\leq 1,
		\end{aligned}
		\right.\]
		where $\tilde{v}^{1}(t,s)$ is the continuous extension of $\frac{\frac{\partial}{\partial t}g_{0}(t,s)}{\phi(s)}$ to $(0,1)\times [0,1]$. 
		
		Since
		\[\frac{\partial}{\partial s}\tilde{v}^{1}(t,s)=\left\{
		\begin{aligned}
			&-\frac{1}{2},\;\; 0\leq s\leq t\leq 1,\\
			&\frac{t}{s^3}\,\left(t-s\right),\;\; 0< t< s\leq 1,
		\end{aligned}
		\right.\] 
		is negative for all $(t,s)\in (0,1)\times [0,1]$ we have that $\tilde{v}^{1}(t,s)$ is decreasing as a function of $s$ for all $t\in (0,1)$. So,
		\begin{equation*}
			\begin{aligned}
				k_{1}^{1}(t)&=\min_{s\in I} \tilde{v}^{1}(t,s)=\tilde{v}^{1}(t,1)=\frac{1}{2}\, t\,\left(1-t\right),\\
				k_{2}^{1}(t) &=\max_{s\in I}\tilde{v}^{1}(t,s)=\tilde{v}^{1}(t,0)=\frac{1}{2}\,\left(1-t\right). 
			\end{aligned}
		\end{equation*}
		In this case, we have that
		$$k_{1}^{1}=k_{1}^{1}\left(\frac{1}{2}\right)=\frac{1}{6}\;\; \text{and}\;\; k_{2}^{1}=k_{2}^{1}(0)=\frac{1}{2}.$$
		
		Therefore, the cone $K$ that we use to localize the solution is given by
		$$K=\left\{u\in C^{3}(I):\;\;
		\begin{aligned}  
			&u(t)\geq 0,\, u'(t)\geq 0,\, u'''(t)\le 0\;\; t\in [0,1],\\
			& u(t)\geq \frac{1}{3}\,t^2\,(3-2\,t)\,\|u \|_{\infty},\,u'(t)\geq t\,(1-t)\,\left\Vert u'\right\Vert_{\infty}
		\end{aligned}
		\right\}.$$
		
		Thus, from Theorem \ref{theoremmain}, there is at least one nontrivial solution $u\in K$ of problem \eqref{ProblemaNonlineal1}.
		
		\vspace*{.3cm}
		
		Finally, if we consider the following nonlinear problem for $M \in \R$:
		{\footnotesize\begin{equation}\label{ProblemaNonlineal}
			\left\{
			\begin{aligned}
				u^{(4)}(t)+M\,u(t)&=(t^4+1)\left(e^{-\sqrt{(u(t))^2+(u'(t))^{2}+(u''(t))^{2}+(u'''(t))^{2}}}+\frac{1}{\ln(e+|u(t)|+|u'(t)|+|u''(t)|+|u'''(t)|)}\right),\;\; t\in [0,1],\\
				u(0)&=u'(0)=0,\\
				u'(1)&=u'''(1)=0.
			\end{aligned}
			\right.
		\end{equation}}
		As we have pointed out at the beginning of this section, we know that $g_{M}$ satisfies the property $(P_{g_{1}})$ for all $M\in(-2.36^4,4\pi^4]$ for suitable functions $k_{1}[M](t)$ and $k_{2}[M](t)$, $t\in [0,1]$ that depend of the value of parameter $M$.
		
		Also, $\frac{\partial}{\partial t} g_{M}$ satisfies the property $(P_{g_{1}})$ for all $M\in(-2.36^4,2.22^4]$ for adequate functions $k_{1}^{1}[M]$ and $k_{2}^{1}[M](t)$, $t\in [0,1]$.
		
		Let us denote $k_{2}[M]=\max_{t\in [0,1]} |k_{2}[M](t)|$ and $k_{2}^{1}[M]=\max_{t\in [0,1]} |k_{2}^{1}[M](t)|$. 
		
		\vspace*{2pt}
		In this case, we have three situations depending of the value of $M$ where to localize the solution. They are given by the following cones:
		\begin{itemize}
			\item If $M\in(-2.36^4,4\pi^4]$, we consider the cone
			$$K_{0}=\left\{u\in C(I):\;\;
			\begin{aligned}  
				&u(t)\geq 0\;\; t\in [0,1],\\
				& u(t)\geq \frac{k_{1}[M](t)}{k_{2}[M]}\,\|u \|_{\infty}
			\end{aligned}
			\right\}.$$
			\item If $M\in(-2.36^4,2.22^4]$, we consider the cone
			$$K_{1}=\left\{u\in C^{1}(I):\;\;
			\begin{aligned}  
				&u(t)\geq 0,\, u'(t)\geq 0\;\; t\in [0,1],\\
				& u(t)\geq \frac{k_{1}[M](t)}{k_{2}[M]}\,\|u \|_{\infty},\,u'(t)\geq \frac{k_{1}^{1}[M](t)}{k_{2}^{1}[M]}\,\|u' \|_{\infty}
			\end{aligned}
			\right\}.$$
			\item If $M\in(-2.36^4,0]$, we consider the cone 
			$$K_{3}=\left\{u\in C^{3}(I):\;\;
			\begin{aligned}  
				&u(t)\geq 0,\, u'(t)\geq 0,\, u'''(t)\le 0\;\; t\in [0,1],\\
				& u(t)\geq \frac{k_{1}[M](t)}{k_{2}[M]}\,\|u \|_{\infty},\,u'(t)\geq \frac{k_{1}^{1}[M](t)}{k_{2}^{1}[M]}\,\|u' \|_{\infty}
			\end{aligned}
			\right\}.$$
		\end{itemize}
		Therefore, from Theorem \ref{theoremmain}, there is at least one nontrivial solution $u\in K_{i}$ for each $i\in \{0,1,3\}$ of problem \eqref{ProblemaNonlineal} according to the values of the parameter $M$.
		
		 Notice that, if $M \in (-31,2852,0]$, we can ensure the existence of a solution of problem~\eqref{ProblemaNonlineal}, which is positive and increasing on $(0,1]$ and whose first derivative is concave in $I$.
			
			Moreover, we can ensure the existence of a positive and increasing solution on $(0,1]$ of problem \eqref{ProblemaNonlineal}, whenever $M \in (-31.2852, 24.2891]$.
			
			Finally, if we look for positive solutions, not necessarily increasing, we can ensure their existence for all  $M \in (-31.2852, 389.636]$.
			
			Obviously, as much restrictive are the imposed conditions on the solutions we are looking for, as smaller is the interval of parameters $M$ for which we can ensure the existence of the finding solutions. In this sense we point out that $ \lambda_2' \approx  16\cdot \lambda_2^{1}.$
		
	\end{example}
	
	\section*{Acknowledgements}
	 The three authors were partially supported by Grant PID2020-113275GB-I00, funded by MCIN/AEI/10.13039/501100011033 and by “ERDF A way of making Europe” of the	“European Union”, and by Xunta de Galicia (Spain), project ED431C 2023/12.

\end{document}